\documentclass[12pt]{article}
\usepackage{graphicx} 

\usepackage[english]{babel}
\usepackage{amsfonts}
\usepackage{mathtools}
\usepackage[left=2cm,right=2cm,
top=2cm,bottom=2cm,bindingoffset=0cm]{geometry}

\usepackage{amssymb,amsmath,amsthm,amsfonts,bbm}
\usepackage[pdftex,colorlinks=true,linkcolor=blue,citecolor=blue,urlcolor=red,unicode=true,hyperfootnotes=false,bookmarksnumbered]{hyperref}

\usepackage{enumerate}

\theoremstyle{plain}
\newtheorem{theorem}{Theorem}[section]
\newtheorem{lemma}[theorem]{Lemma}
\newtheorem{claim}[theorem]{Claim}

\theoremstyle{definition} 

\newtheorem{remark}[theorem]{Remark}

\makeatletter
\def\@gifnextchar#1#2#3{\let\@tempe#1\def\@tempa{#2}\def\@tempb{#3}%
  \futurelet\@tempc\@gifnch}

\def\@gifnch{\ifx\@tempc\@sptoken\let\@tempd\@tempb%
  \else\ifx\@tempc\@tempe\let\@tempd\@tempa\else\let\@tempd\@tempb\fi\fi\@tempd}

\def\SK@set#1{\left\{#1\right\}}
\def\SK@@set#1#2{\{#1\,:\,
    \begin{array}{@{}l@{}}#2\end{array}
\}}
\def\SK@mset#1{\left\{\!\!\left\{#1\right\}\!\!\right\}}
\def\SK@@mset#1#2{\{\!\!\{#1\,:\,
    \begin{array}{@{}l@{}}#2\end{array}
\}\!\!\}}
\def\BIG@set#1{\Big\{#1\Big\}}
\def\BIG@@set#1#2{\Big\{#1\:\Big|\:
    \begin{array}{@{}l@{}}#2\end{array}
\Big\}}
\newcommand{\Set}[1]{\@gifnextchar\bgroup{\SK@@set{#1}}{\SK@set{#1}}}
\newcommand{\Mset}[1]{\@gifnextchar\bgroup{\SK@@mset{#1}}{\SK@mset{#1}}}
\newcommand{\Bigset}[1]{\@gifnextchar\bgroup{\BIG@@set{#1}}{\BIG@set{#1}}}
\makeatother

\newcommand{\1}{\mathbbm{1}}

\author{
Itai Benjamini\thanks{Weizmann Institute of Science, itai.benjamini@weizmann.ac.il},
Georgii Zakharov\thanks{The University of Oxford, zakharovg@maths.ox.ac.uk}
,
Maksim Zhukovskii\thanks{The University of Sheffield, m.zhukovskii@sheffield.ac.uk}
}

\date{}

\title{Majority dynamics on finite trees}

\begin{document}

\maketitle

\begin{abstract}
For an arbitrary finite tree $T$, we find the exact value of the wort-case stabilisation time of majority dynamics on $T$. We also prove that for a perfect rooted cubic tree $T$ with diameter $D$ and uniformly random initial opinions, the dynamics stabilises in time $\tau\in(D/4,D/3)$ with high probability.
\end{abstract}


\section{Introduction}


Majority dynamics is a fundamental process on networks where each individual interacts with its neighbours by changing its opinion on the opinion of the majority. This simply described and natural model has been of interest in
biophysics~\cite{McCulloch90}, social psychology~\cite{Yin19} and computer science~\cite{Alistarh15}. Formally, this process is described as follows. Initially each individual $i\in V$ considered as a vertex of a graph $G$ on a (not necessarily finite) set of vertices $V$ has an {\it initial opinion} $\xi_0(i)\in\{-1,1\}$. Then, at every time step $t\in\mathbb{N}$, they adopt the majority opinion of their neighbours, i.e. 
$$
\xi_t(i)=\mathrm{sign}\sum_{j\in N_G(i)}\xi_{t-1}(j),
$$
where $N_G(i)$ is the set of neighbours of $i$ in $G$. For convenience, we only consider graphs with odd degrees, so $\sum_{j\in N_G(i)}\xi_t(j)$ is never equal to 0.





It was proved by Goles and Olivos \cite{GolesOlivos} that the process stabilises for every finite graph $G$. More formally, the following {\it period two property} holds true: for every vertex $i$ and all large enough $t$, $\xi_{t+2}(i)=\xi_t(i)$.
However, the question on estimating the time until stabilisation 
$$
\tau(G;\xi_0)=\min\{t:\xi_{t+2}(i)=\xi_t(i)\text{ for all }i\in V(G)\}
$$ 
is still open. In~\cite{PoljakTurzik}, it is proven that, 
 for every $G$ and every $\xi_0\in\{-1,1\}^{V(G)}$, $\tau(G;\xi_0)\leq|E(G)|-\frac{1}{2}|V(G)|$. It is known however that for most graphs and most assignments of initial opinions, $\tau$ is much smaller: it was proven in~\cite{Fountoulakis} that if $p\gg n^{-1/2}$ then
 with high probability\footnote{With probability approaching 1 as $n\to\infty$. In what follows, we will write simply `whp' for brevity.} 
 $\tau:=\tau(G(n,p);\xi_0)\leq 4$ for $\xi_0\in\{-1,1\}^{V(G)}$ chosen uniformly at random. Here, as usual $G(n,p)$ denotes a binomial random graph on $[n]:=\{1,\ldots,n\}$ with every edge drawn independently with probability $p$. This was improved in~\cite{Chakraborti}: if $p\gg n^{-3/5}\log n$, then whp $\tau\leq 6$.  
In \cite{BenjaminiEtAl,TamuzTessler} the stabilisation phenomena was also studied for infinite graphs. It was proven in~\cite{BenjaminiEtAl} that for any unimodular transitive graph $G$ and any $\mathrm{Aut}(G)$-invariant distribution of initial opinions, the opinions of every vertex stabilise with probability 1, and the expected time until $\xi_t(i)$ stabilises is at most $2d$, where $d$ is the degree of $G$. Many natural classes of transitive infinite graphs are unimodular. In particular, the above is true for infinite regular trees $T$ and independent identically distributed $\xi_0(i)$, $i\in V(T)$. Our results imply a qualitative refinement of this result for regular trees of degree 3 (see Claim~\ref{cl:infinite_3-regular_tree} below). The exact value of $\max_{\xi_0}\tau(G=C_n;\xi_0)$ for a cycle was obtained in~\cite{AAMAS}. For further background on majority dynamics and related processes, see the survey~\cite{survey}.\\


In this paper, we study majority dynamics on {\it finite} trees. First of all, for a given tree $T$, we find the exact value of the worst-case stabilisation time
$$
\tau(T):=\max_{\xi_0\in\{-1,1\}^{V(T)}}\tau(T;\xi_0).
$$
Note that, due to~\cite{PoljakTurzik}, $\tau(T)\leq\frac{1}{2}|V(T)|-1$. As follows from the theorem below, this result is not best possible for any $T$ with $|V(T)|\geq 5$. 


\begin{theorem} Let all vertices of a tree $T$ have odd degrees and $|V(T)|\geq 5$.
Let $\mathcal{Q}$ be a set of paths $Q = v_1 \ldots v_{n}$ such that
\begin{itemize}
\item at most $\frac{1}{2}(\mathrm{deg}_Tv_n-1)$ of neighbours of $v_n$ in $T$ are leaves,
\item less than $\frac{1}{2}(\mathrm{deg}_Tv_i-1)$ of neighbours of $v_i$ in $T$ are leaves for $i \in [n-1]$.
\end{itemize}
For every $Q \in \mathcal{Q}$, let $t(Q)$ be defined as follows
$$
t(Q)= 
\begin{cases}
    |V(Q)| + 1, & \text{if } v_n \text{ is adjacent to a leaf}\\
    |V(Q)|,              & \text{otherwise.}
\end{cases}
$$
Then $\tau(T)=\max_{Q \in \mathcal{Q}}t(Q)$.
\label{th:1}
\end{theorem}




In particular, for {\it perfect} trees (in order to satisfy the condition that all degrees are odd, we assume here that a root has the same degree as all non-leaf vertices) with diameter $D=D(T)$ we have that $\tau(T) = D - 3$.\\

A natural question to ask, does $\tau$ drop significantly in the average case, i.e. when $\xi_0$ is chosen uniformly at random? We prove that for perfect binary trees it remains linear in $D$, though the constant factor becomes less than 1. Note that, as above, the root of a perfect binary tree has degree 3 (see the definition of a perfect $k$-ary tree in Section~\ref{sc:pre}). 

\begin{theorem} Let $\xi_0=\xi_0(T)\in\{-1,1\}^{V(T)}$ be chosen uniformly at random on vertices of a perfect binary tree $T=T(D)$ of diameter $D$. There exist $c_- >1/4$ and $c_+ <1/3$ such that whp, as $D\to\infty$, 
$$
c_- \cdot D < \tau(T; \xi_0) < c_+ \cdot D.
$$ 
\label{th:2}
\end{theorem}


For other $k$-ary trees, we are only able to prove an $\Omega(\sqrt{D})$ bound. 

\begin{theorem} There exists $c_->0$ such that, for every 
even $k\geq 4$ and a uniformly random $\xi_0\in\{-1,1\}^{V(T)}$ on vertices of a perfect $k$-ary tree $T$, whp
$$
c_- \cdot \frac{\sqrt{\ln k}}{k}\sqrt{D(T)} < \tau(T; \xi_0) \leq D(T)-3.
$$ 
\label{th:3}
\end{theorem}


\paragraph{Proof strategy.}

The main observation that we use to prove the results is that for a vertex $v$ of a tree, in order to change its opinion to $\xi\in\{-1,1\}$ at some time step $t\geq 2$, there should exist a path of $t-1$ vertices ending at $v$ such that vertices from the path change their opinion to $\xi$ one by one along the path. 
 In particular, in  Theorem~\ref{th:1} we construct and in Theorems~\ref{th:2}~and~\ref{th:3} we find such paths that let opinions pass through, in order to prove the lower bounds. 

In order to prove the upper bound in Theorem~\ref{th:1}, we will observe that only vertices $v$ that have less than $\frac{1}{2}(\deg_T(v)-1)$ leaf-neighbours (we will call them {\it{active}}) are able to convey an opinion to the next non-leaf vertex in a path with potentially good `conductivity' properties. So, $\tau(T)$ is approximately the length of the longest path of active vertices. The lower bound is obtained via an explicit assignment of opinions that allow to pass an opinion through the longest path of active vertices. 


For random initial opinions and a fixed path $P$ of active vertices, the main challenge is to extract independent events that allow to get tight lower and upper bounds on the probability that initial opinions of vertices allow some opinion to pass through the path. These independent events are naturally associated with opinions of vertices that belong to disjoint sets. We define them using notions of {\it weak} and {\it strong stability}. The proof of the lower bound in Theorem~\ref{th:2} relies on the notion of strong stability: a vertex $v$ is {\it strongly $t$-stable} if the initial opinions of the subtree $T_v$ induced by the descendants of $v$ guarantee $t$-stability of $v$ whatever the opinions of all the other vertices. We will prove that a vertex is strongly $t$-stable with at least a constant positive probability. It is worth noting that this immediately implies the following refinement of the result from~\cite{BenjaminiEtAl} in the particular case of infinite regular trees of degree 3:
\begin{claim}
Let $T$ be an infinite regular tree of degree 3 and let $\xi_0$ have uniform distribution on $\{-1,1\}^{V(T)}$. Fix a root $v\in V(T)$. Then $\mathbb{P}(\xi_{2t}(v)=\xi_0(v)\text{ for all $t\in\mathbb{N}$})>0.5$.
\label{cl:infinite_3-regular_tree}
\end{claim}

Since the event that $v$ is $t$-stable is described in terms of initial opinions of vertices of $T_v$ only, we may define an event that 1) occurs when certain $O(d)$ vertices with non-overlapping pendant trees are strongly stable (and some other auxiliary technical properties of these vertices and associated trees hold, that we omit here for the sake of clarity); 2) implies that vertices along $P$ change their opinions one by one. This gives us that the probability that initial opinions of vertices allow some opinion to pass through $P$ is at least $\exp(-O(d))$. It then remains to find $\exp(\Omega(d))$ such paths so that the approximation events are independent.

For the upper bound in Theorem~\ref{th:2}, we introduce the notion of weak stability: in particular, $v$ is {\it weakly $0$-stable} if the initial opinions of the vertices of $T_v$ guarantee the existence of opinions of all the other vertices so that $v$ is $0$-stable (the notion of weak $t$-stability is more complex and thus we omit it here; the definition for an arbitrary $t$ is given in Section~\ref{sc:4:1}). As in the case of strong stability, the event that $v$ is weakly $0$-stable depends on initial opinions of vertices of $T_v$ only (nevertheless, the event that $v$ is weakly $t$-stable depends on opinions of vertices of $T_v$ at time $t$). Using this notion, we prove that the probability that vertices along $P$ change their opinion one by one is at most $\exp(-\Omega(d))$. Note that this fact is sufficient due to the union bound. The crucial observation that allows to prove this fact is that a weakly $t$-stable vertex is `half-stable' --- when two weakly $t$-stable vertices have the same opinion at time $t$ and are connected by a path consisting of vertices with the same opinion at time $t$, they stabilise immediately.  
Our intuition here is that, roughly speaking, there are typically many weakly $t$-stable vertices in the tree. Therefore, they frequently stabilise, and once they stabilise, an opinion cannot travel through them, preventing changes of any opinion after some fairly long time. A little more formally, we show that the fact that an opinion travels through some path $P$ implies linearly in $|P|$ many equations of the form ``$\xi_t(v) = \xi$'' for vertices $v$ from the neighbourhood of $P$. On the other hand, we describe a linear set of independent events that contradict to these equations (we use Claim~\ref{cl:weak_stabilization} and Claim~\ref{cl:one_far_stabilization} to show that). Then, we show that $\mathbb{P}(v \in V \text{ is weakly 0-stable})$ and $\mathbb{P}(v\in V \text{ is weakly stable after the first change of its opinion})$ are bounded away from zero which, in turn, implies an exponential bound for the probability that neither from the set of contradicting events happens. This immediately implies the desired exponential upper bound on the probability the travel of some opinion along $P$ does not happen.

In the proof of the lower bound of Theorem~\ref{th:3} we introduce a new notion of $(\leq t)$-stability. The proof is similar to the proof of Theorem~\ref{th:2}. Roughly speaking, in the proof, $(\leq t)$-stability plays the same role as strong $t$-stability in Theorem~\ref{th:2}. 
 We replace $t$-stability by $(\leq t)$-stability because, in contrast to the perfect binary tree, for an arbitrary perfect $k$-ary tree we did not manage to prove a positive constant lower bound on the probability of $t$-stability. However, our bound on the probability of $(\leq t)$-stability for an arbitrary perfect $k$-ary tree is worse than the bound on strong $t$-stability for a perfect binary tree, which results in a weaker lower bound.\\ 

This paper is organised as follows. In Section~\ref{sc:pre}, we introduce all the necessary definitions including the properties of stability. Further, in Section~\ref{sc:3} we prove Theorem~\ref{th:1}. The average-case bounds are proven in the next two sections: Section~\ref{sc:4} is devoted to the proof of Theorem~\ref{th:2} and Theorem~\ref{th:3} is proven in Section~\ref{sc:5}. Finally, in Section~\ref{sc:dis} we discuss our results as well as further questions and avenues for the future research.

\section{Preliminaries}
\label{sc:pre}

Let $T$ be a tree and $R$ be its vertex. The pair $(T,R)$ is called a {\it rooted tree}, and $R$ is called the {\it root}. For brevity, we will denote $(T,R)$ simply by $T$ and say that $T$ is {\it $R$-rooted} or {\it rooted in $R$}. We only consider trees with odd degrees of vertices. A vertex $u$ is a {\it descendant} of a vertex $v$ in $T$, if the path between $u$ and $R$ contains $v$. A descendant of $v$ which is also its neighbour is called a {\it child} of $v$. For every vertex $v$ of $T$, we denote by $T_v$ the subtree of $T$ induced by all descendants of $v$ and $v$ itself rooted in $v$. The {\it height} of $T$ is the maximum distance between its root and its leaf. Let us denote the height of $T_v$ by $h_T(v)$ and call it the {\it height of $v$}.

We call a rooted tree $T$ {\it $k$-ary}, if all vertices other than leaves have $k+1$ neighbours: for consistency with the definition of the majority dynamics, we let the root of a $k$-ary tree to have exactly the same degree as other non-leaf vertices. We call a $k$-ary tree {\it perfect}, if all its leaves are at the same distance from the root. 
A (perfect) 2-ary tree is also called a {\it (perfect) binary tree}. In particular, the root of a (perfect) binary tree has degree 3.\\



Let us now run the majority dynamics on a rooted tree $T$ with an arbitrary initial vector of opinions $\xi_0$. Let $t\in\mathbb{Z}_{\geq 0}$. Let us call a vertex $u$ {\it $t$-stable}, if for every $s\geq t$ with the same parity as $t$, $\xi_{s}(u)=\xi_{s+2}(u)$ (here we assume that $0$ is even). The vertex is {\it $t$-stationary}, if it is $t$-stable and $(t+1)$-stable.




\section{Proof of Theorem~\ref{th:1}}
\label{sc:3}


In Section~\ref{sc:th1_proof_upper}, we prove the upper bound $\tau(T)\leq\max_{Q\in\mathcal{Q}}t(Q)$. In Section~\ref{sc:th1_proof_lower} we constructively prove the lower bound $\tau(T)\geq\max_{Q\in\mathcal{Q}}t(Q)$ by assigning the initial opinions in a proper way.

\subsection{Upper bound}
\label{sc:th1_proof_upper}

Here we prove the upper bound in Theorem~\ref{th:1}. We distinguish between three types of vertices. We call a vertex $v$ of $T$ {\it active}, if less then $\frac12(\mathrm{deg}_Tv - 1)$ of its neighbours are leaves. We call $v$ {\it balky}, if exactly  $\frac12(\mathrm{deg}_Tv - 1)$ of its neighbours are leaves. In particular, all leaves are balky. All the other vertices are {\it passive}. Let us observe the following crucial property of balky vertices.

\begin{claim}
If $v$ is balky, $u$ is a non-leaf neighbour of $v$, and, for some $s\in\mathbb{Z}_{\geq 0}$, $\xi_s(v)=\xi_{s+1}(u)$, then $\xi_{s+2}(v)=\xi_{s}(v)$.
\label{cl:balky_property}
\end{claim}


\begin{proof}
Indeed, let $u_1, \ldots, u_m$, $m={\frac12(\mathrm{deg}_Tv - 1)}$, be all the leaves adjacent to $v$. Then, $\xi_{s}(v) = \xi_{s+1}(u_1) = \ldots = \xi_{s+1}(u_m)$ and $\xi_{s+1}(u)$ is the same by the assumption of the claim. As $u$ is not a leaf, we get that more than half of neighbours of $v$ share the same opinion $\xi_{s}(v)$ at time $s+1$. Thus, $\xi_{s+2}(v) = \xi_{s}(v)$.
\end{proof}

Now, note that all passive vertices are $0$-stationary, so there is nothing to prove for them. Claim~\ref{cl:active} formulated in the end of this section states that every active vertex $u$ stabilises in time no more then the number of vertices in the longest path $Q$ of active vertices starting in $u$. In particular, it matches the upper bound in the statement of Theorem~\ref{th:1} since $Q\in\mathcal{Q}$. Due to Claim~\ref{cl:balky_property}, a balky vertex $v$ may $\it{switch}$ at time $s\geq 2$ (i.e. $\xi_s(v) \neq \xi_{s-2}(v)$) only if $\xi_{s-1}(u) = \xi_s(v)$ for every non-leaf $u \in N_T(v)$. Therefore, $v$ cannot cause a change of an opinion of its non-leaf neighbour. This means that, as soon as all active neighbours of $v$ stabilise, then non-leaf balky vertices stabilise in at most one additional step. Moreover, a balky non-leaf vertex can change its value only if it has an active neighbour, while a leaf can switch at time $s\geq 3$ if its neighbour is either active or balky. Therefore, for a balky vertex $v$ with at least one active neighbour the stabilisation time is at most 1 bigger than the number of vertices in the longest path of active vertices starting at an active neighbour of $v$. Again, we shall note that this time is at most the upper bound stated in Theorem~\ref{th:1}: for any path $Q$ initiated at a balky vertex with all the other vertices being active we have that $Q\in\mathcal{Q}$. Finally, at most one additional step is needed for leaves adjacent to balky vertices. Any path $v_{n+1} v_n v_{n-1}\ldots v_1$ with leaf $v_{n+1}$, balky vertex $v_n$, and active vertices $v_1,\ldots,v_{n-1}$ has $t(Q)=|V(Q)|+1$ vertices for $Q=v_1\ldots v_n\in\mathcal{Q}$. Thus, leaves stabilise in time $\max_{Q\in\mathcal{Q}}t(Q)$ as well.\\

It remains to prove that, if any path consisting of active vertices initiated at an active vertex $v$ has at most $t$ vertices, then $v$ is $t$-stable. Indeed, this would also imply that $v$ is $(t+1)$-stable and, therefore, $t$-stationary by applying the same assertion to the process with the vector of initial opinions $\xi_1$. The following claim completes the proof of the upper bound in Theorem~\ref{th:1}.

\begin{claim} Let all vertices of a tree $T$ have odd degrees, and let $v$ be an active vertex of $T$. Then $v$ is $t$-stable, where $t$ is the number of vertices in the longest path of active vertices starting in $v$.
\label{cl:active}
\end{claim}

\begin{proof}
Suppose towards a contradiction that there is a time step $t_1 \geq t$ such that $\xi := \xi_{t_1+2}(v) \neq \xi_{t_1}(v)$. Then the multisets of opinions of neighbours of $v$ at time steps $t-1$ and $t+1$ should differ. In particular, there is a neighbour $v_{t_1-1}$ of $v$ such that $\xi = \xi_{t_1+1}(v_{t_1-1}) \neq \xi_{t_1-1}(v_{t_1-1})$. We can apply the same reasoning to $v_{t_1-1}$ at time step $t_1-1$, so we find a neighbour $v_{t_1-2}$ of $v_{t_1-1}$ such that $\xi = \xi_{t_1}(v_{t_1-2}) \neq \xi_{t_1-2}(v_{t_1-2})$. In addition $v_{t_1-2} \neq v$ as their opinions differ at time step $t_1$. Applying the same reasoning $t_1-2$ times more, we find a path $v=v_{t_1}v_{t_1-1}\ldots v_0$ such that, for every $t' \in \{0, \ldots, t_1\}$, $\xi = \xi_{t'+2}(v_{t'}) \neq \xi_{t'}(v_{t'})$. Since $t_1+1>t$, the length of this path is longer than the maximum length of a path of active vertices initiated at $v$. Thus, there exists $t' \in \{0, \ldots, t_1-1\}$ such that $v_{t'}$ is not active. Then $v_{t'}$ is either balky or passive. Actually, it cannot be passive, since $\xi_{t'+2}(v_{t'}) \neq \xi_{t'}(v_{t'})$ while a passive vertex is 0-stationary. So it should be balky. However, $\xi_{t'}(v_{t'}) = \xi_{t'+1}(v_{t'+1}) \neq \xi$ and $v_{t'+1}$ is not a leaf --- either it is adjacent to different vertices $v_{t'+2}$ and $v_{t'}$ or it coincides with $v$ which is active and, thus, cannot be a leaf. So, by Claim~\ref{cl:balky_property}, $\xi_{t'+2}(v_{t'}) = \xi_{t'}(v_{t'})$ --- a contradiction, completing the proof.
\end{proof}

\subsection{Lower bound}
\label{sc:th1_proof_lower}

Consider a path $Q = v_1\ldots v_n\in\mathcal{Q}$ maximising $t(Q)$. 
 To prove the lower bound in Theorem~\ref{th:1}, it is sufficient to present $\xi_0 : V(T) \rightarrow \{-1, 1\}$ such that $v_n$ changes its opinion at time $n+1$. Indeed, if $t(Q) = n$, then $v_n$ is not $(t(Q) - 1)$-stable. If $t(Q) = n + 1$, then $v_n$ is adjacent to a leaf that, in turn, changes its opinion at time $n+2$. 
  Thus, this leaf is not $(t(Q) - 1)$-stable, concluding the proof. 
   It remains to prove the following claim.





\begin{claim}
Let $Q = v_1 \ldots v_n\in\mathcal{Q}$. 
Then there is $\xi_0: V(T)\to\{-1,1\}$ such that
\begin{itemize}
\item for every $i \in [n]$ holds $\xi_{i+1}(v_i) = -1$,
\item for every $i \in [n]$ and $t \leq i$ holds $\xi_t(v_i) = 1$.
\end{itemize}
\label{cl:main_optimal}
\end{claim}


\begin{proof}

Let $v_n$ be the root of $T$. Let us set $\xi_0(v_i) = 1$ for all $i\in[n]$ and for every child $u$ of $v_1$ set $\xi_0(u) = 1$. For every $i\in[n]\setminus\{1\}$ and every leaf $u$ adjacent to $v_i$, we set $\xi_0(u)=1$. By the definition of $\mathcal{Q}$, every vertex $v_i$, $i\in[n]\setminus\{1\}$, has at least $\frac{1}{2}(\mathrm{deg}_Tv_i-1)$ children other then leaves and $v_{i-1}$. We then set $\xi_0(u)=-1$ for exactly $\frac{1}{2}(\mathrm{deg}_Tv_i-1)$ non-leaf children $u$ of $v_i$ other than $v_{i-1}$. All the other children of every $v_i$, $i\in[n]\setminus\{1\}$, get $\xi_0=1$. It only remains to assign initial opinions to vertices that are not in the (closed) $1$-neighbourhood of $Q$. For every $i\in[n]\setminus\{1\}$ and every non-leaf child $u$ of $v_i$ we assign the opinion $\xi_0(u)$ to every vertex of $T_u$. Lastly, every descendant of $v_1$ at distance at least 2 from $v_1$ gets initial opinion $-1$.


First of all, note that $\xi_1(v_1) = \xi_1(v_2) = \ldots = \xi_1(v_n) = 1$ by the definition of $\xi_0$. Next, note that, for every $i \in [n]\setminus\{1\}$, every non-leaf child of $v_i$ is 0-stationary by the definition of $\xi_0$. In addition, every leaf-child of $v_i$ has opinion $1$ at time step $1$. So, $\xi_2(v_2) = \ldots = \xi_2(v_n) = 1$. However, every non-leaf child of $v_1$ has opinion $-1$ at time step 1. So, $\xi_2(v_1) = -1$, as $v_1$ is active vertex.

It is sufficient for us to prove by induction that, for every $t\in[n]$, at time $t$ we have the following: $\xi_{t+1}(v_t) = -1$ while $\xi_{t+1}(v_{t+1}) = \ldots = \xi_{t+1}(v_n) = 1$. The base of induction for $t=1$ is already verified. We suppose that this is true for some $t\in[n-1]$. Let $s> t+1$ be the first moment when a certain $v_{\tau}$, $\tau\geq t+1$, changes its opinion (we let $s=\infty$ if all $v_{\tau}$, $\tau\geq t+1$, are 0-stationary). Then, $s-1$ was the first time step when more than a half of the neighbours of $v_{\tau}$ had opinion $-1$. Note that all children of $v_{\tau}$ other than $v_{\tau-1}$ that are not leaves are 0-stationary. Also, all children $u$ of $v_{\tau}$ that are leaves should have had $\xi_{s-1}(u)=1$ since otherwise $\xi_{s-2}(v_{\tau}) = -1$, which contradicts the induction hypothesis. Thus, $s-1$ was the first time step when at least one of the neighbours of $v_{\tau}$ from $Q$, namely $v_{\tau-1}$ or $v_{\tau+1}$, had opinion $-1$. By the induction hypothesis, it might only be possible if $s=t+2$ and $\tau=t+1$. Hence, $\xi_{t+2}(v_{t'})=1$ for all $t' \geq t+2$, completing the proof.

\end{proof}

\section{Perfect binary trees: proof of Theorem~\ref{th:2}}
\label{sc:4}

In this section, we prove Theorem~\ref{th:2}. It asserts both the lower and the upper bound for the typical stabilisation time on a perfect binary tree. We prove these two bounds separately and, for convenience, state them as two lemmas in what follows. Before proceeding with their formulations and proofs, we introduce concepts of weak and strong stability. 

The entire proof is organised as follows. In Section~\ref{sc:4:1}, we define weak and strong stability and discuss their basic properties. In particular, in Claim~\ref{cl:binary_calculations} we state lower bounds on probabilities of weak and strong stability for early time steps. In Section~\ref{sc:4:2}, we prove the lower bound from Theorem~\ref{th:2}, and in Section~\ref{sc:4:3}, we prove the upper bound. Finally, in Section~\ref{sc:4:4}, we prove Claim~\ref{cl:binary_calculations}.

\subsection{Weak and strong stability}
\label{sc:4:1}



For every positive integer $h$, we fix a perfect binary tree $T^{(h)}$ of depth $h$ rooted in a vertex $R$. When the height is clear from the context we omit the superscript and write simply $T$. For a uniformly random $\xi_0$ on $V(T)$, we let $\tau_T:=\tau(T;\xi_0)$.

Let us now fix some $h\in\mathbb{Z}_{>0}$, $t\in\mathbb{Z}_{\geq 0}$, a vertex $v\in V(T=T^{(h)})$, $v \neq R$, and a vector of initial opinions $\xi_0:\, V(T)\to\{-1,1\}$. If $v$ is not a leaf, then we call $v$ {\it strongly $t$-stable} (w.r.t. $\xi_0$), if it is $t$-stable with the same opinion (i.e. $\xi_t(v) = \tilde\xi_t(v)$) for {\it any} other vector of initial opinions $\tilde\xi_0:\, V(T)\to\{-1,1\}$ such that $\tilde\xi_0|_{V(T_v)} = \xi_0|_{V(T_v)}$. 
 Similarly, $v$ (whether it is a leaf or not) is {\it weakly $t$-stable} (w.r.t. $\xi_t$), if $v$ is $0$-stable for {\it some} other vector of initial opinions $\tilde\xi_0:\, V(T)\to\{-1,1\}$ such that $\tilde\xi_0|_{V(T_v)} = \xi_t|_{V(T_v)}$.


\begin{claim}
Let $u \in V(T)$ be a parent of a vertex $v \in V(T)$ and $t \in \mathbb{Z}_{\geq0}$. The following three statements are equivalent.

\begin{enumerate}
    \item The vertex $v$ is weakly $t$-stable.
    \item For the vector $\tilde\xi^*_0:\, V(T)\to\{-1,1\}$ such that $\tilde\xi^*_0|_{V(T_v)} = \xi_t|_{V(T_v)}$ and $\tilde\xi^*_0(\{V(T)\setminus V(T_v)\}) = \{\xi_t(v)\}$, the vertex $v$ is $0$-stable w.r.t. $\tilde\xi^*_0$.
    \item For every odd $k > 0$ and $\hat\xi_0:\, V(T)\to\{-1,1\}$ such that $\hat\xi_0|_{V(T_v)} = \xi_t|_{V(T_v)}$ and $\hat\xi_{j}(u) = \hat\xi_{0}(v)$ for every odd $j \in \{1, 3, \ldots, k\}$, it holds that $\hat\xi_{k+1}(v) = \hat\xi_{0}(v)$.
\end{enumerate}
\label{cl:equivalent_weak_definitions}
\end{claim}

\begin{proof}

$(1 \implies 2)$ Let $\tilde\xi$ be any vector from the definition of weak $t$-stability. Let us note that for every $w \in V(T)$, if $\tilde\xi_0(w) = \xi_t(v)$ then $\tilde\xi_0^*(w) = \xi_t(v)$. So, by induction on $k$, we get that for every $k \in \mathbb{Z}_{\geq 0}$, if $\tilde\xi_k(w) = \xi_t(v)$, then $\tilde\xi_k^*(w) = \xi_t(v)$. Hence, if $v$ is 0-stable w.r.t. $\tilde\xi_0$, then it is also 0-stable w.r.t. $\tilde\xi_0^*$.

$(2 \implies 3)$ Fix an odd $k > 0$, a vector $\hat\xi_0$ from the third definition of weak stability, and the vector $\tilde\xi_0^*$ from the second definition. Let us note that $u$ is 1-stable w.r.t. $\tilde\xi^*_0$ with $\tilde\xi_1^*(u)=\tilde\xi^*_0(v)=\xi_t(v)$ since its child $v$ is 0-stable and all vertices outside $V(T_v)$ have initial opinion $\tilde\xi^*_0(v)$. So, for every odd $j \in \{1, 3, \ldots, k\}$, 
$$
\hat\xi_{j}(u) = \hat\xi_0(v)=\xi_t(v)=\tilde\xi_1^*(u)=\tilde\xi^*_{j}(u).
$$
Therefore, for every $i \in \{0, 1, \ldots k\}$ and for every vertex $w \in V(T_v)$ such that $d_T(w, v)$ has the same parity as $i$, $\hat\xi_i(w)=\tilde\xi^*_i(w)$. Indeed, for $i=0$, this is clearly true since $\tilde\xi_0^*|_{V(T_v)}=\hat\xi_0|_{V(T_v)}$. Assuming that $\hat\xi_{i-1}(w)=\tilde\xi^*_{i-1}(w)$ for every vertex $w \in V(T_v)$ such that $d_T(w, v)$ has the same parity as $i-1$, we immediately get that $\hat\xi_{i}(w)=\tilde\xi^*_{i}(w)$ for every vertex $w \in V(T_v)$ such that $w\neq v$ and $d_T(w, v)$ has the same parity as $i$. Moreover, if $i$ is odd, then we do not have to consider the vertex $w=v$ at time $i$. If $i$ is even, then $\hat\xi_{i}(v)=\tilde\xi^*_{i}(v)$, since opinions in these two processes coincide for children of $v$ at time $i-1$, and the parent $u$ has equal opinions $\hat\xi_{i-1}(u)=\tilde\xi^*_{i-1}(u)$ as we have just proved as $i-1$ is odd.
 So, indeed, $\hat\xi_{k+1}(v) = \tilde\xi^*_{k+1}(v) =\tilde\xi^*_0(v) = \hat\xi_0(v)$.

$(3 \implies 1)$ Let us note that $u$ is 1-stable w.r.t. the vector $\tilde\xi^*_0$ introduced in the second definition of weak stability with $\tilde\xi^*_1(u)=\tilde\xi^*_0(v)$, since $u$ has degree 3 and all vertices outside $V(T_v)$ have initial opinion $\tilde\xi^*_0(v)$. So, the vector $\tilde\xi^*_0$ satisfies conditions in the third definition of weak stability, thus we set $\hat\xi_0 := \tilde\xi^*_0$. We know that $\hat\xi_{k+1}(v)=\hat\xi_0(v)$ for all odd $k$, that is $v$ is 0-stable w.r.t. $\hat\xi_0$. Hence, $\tilde\xi_0:=\hat\xi_0$ witnesses the weak $t$-stability of $v$.
    
\end{proof}

\begin{claim}

Let $u \in V(T)$ be a parent of a weakly $t_1$-stable vertex $v \in V(T)$ for $t_1 \in \mathbb{Z}_{\geq0}$. Let $t_2 > t_1$ have the same parity as $t_1$. For every time step $t \in \{t_1+1, t_1+3, \ldots t_2-1\}$ let $\xi_{t}(u) = \xi_{t_1}(v)$. Then, $\xi_{t_2}(v) = \xi_{t_1}(v)$.

\label{cl:maintain_weak_value}
    
\end{claim}

\begin{proof}
Let us apply the third equivalent definition from Claim~\ref{cl:equivalent_weak_definitions} with $\hat{\xi}_0 := \xi_{t_1}$ and $k := t_2-t_1 - 1$. Then, 
$\hat{\xi}_t = \xi_{t_1 + t}$ for any $t > 0$. 
So, $\hat\xi_{j}(u) = \hat\xi_{0}(v)$ for every odd $j \in \{1, 3, \ldots, k\}$. 
So it holds that $\hat\xi_{k+1}(v) = \hat\xi_{0}(v)$. 
Hence, $\xi_{t_2}(v) = \hat{\xi}_{k+1}(v) = \hat{\xi}_0(v) = \xi_{t_1}(v)$, as needed.
    
\end{proof}

\begin{claim}
Let $t_1, t_2 \in \mathbb{Z}_{\geq0}$ be such that $t_2 - t_1 \in 2\mathbb{Z}_{>0}$. Let $v \in V(T)\setminus\{R\}$ be weakly $t_1$-stable, and, for every $t \in \{t_1, t_1+2, \ldots, t_2\}$, let $\xi_t(v) = \xi_{t_1}(v)$, then $v$ is weakly $t_2$-stable.
\label{cl:maintain_weak_stability}
\end{claim}

\begin{proof}

It is sufficient for us to prove the statement when $t_2 = t_1+2$, as the former statement follows from applying the latter several times.

Let us consider the extension $\tilde \xi_0^*$ from the second definition of weak $t_1$-stability from Claim~\ref{cl:equivalent_weak_definitions}: $\tilde\xi_0^*|_{V(T_{v})}=\xi_{t_1}|_{V(T_{v})}$ and $\tilde\xi_0^*|_{V(T)\setminus V(T_{v})}\equiv\xi:=\xi_{t_1}(v)$. Then $v$ is 0-stable and, hence, 2-stable (w.r.t. $\tilde\xi^*$). Moreover, all vertices of $T$ outside $T_{v}$ share opinion $\xi$ at time step $2$. 

In order to apply the second definition from Claim~\ref{cl:equivalent_weak_definitions} to show the weak $(t_1+2)$-stability of $v$, it is sufficient to prove that $\tilde\xi^*_2|_{V(T_{v})}=\xi_{t_1+2}|_{V(T_{v})}$ for all vertices at even distance from $v$ (including $v$ itself). The latter is obvious for all vertices from $T_v$ that are at positive even distance from $v$ since their $\tilde\xi^*_2$- and $\xi_{t_1+2}$-opinions are defined by equal vectors $\tilde\xi^*_0|_{T_v}$ and $\xi_{t_1}|_{T_v}$ respectively.  
We complete the proof by recalling that $\xi_{t_1+2}(v)=\tilde\xi^*_2(v) = \xi$, given by the conditions of the claim. 
\end{proof}


Note that both 
 weak 0-stability and strong $t$-stability 
  properties are defined by $\xi_0|_{V(T_v)}$ only.  Now, letting $\xi_0$ to be chosen uniformly at random, we set
$$
p_{\mathrm{s}}(t,h,v):=\mathbb{P}(v\text{ is strongly }t\text{-stable in }T^{(h)}),\quad
p_{\mathrm{w}}(t,h,v):=\mathbb{P}(v\text{ is weakly }t\text{-stable in }T^{(h)}).
$$


\begin{remark}
Note that, for any non-overlapping subsets $U_1,U_2\subset V(T=T^{(h)})$ events defined by $\xi_0|_{U_1}$ and $\xi_0|_{U_2}$ are independent. In particular, if any vertex from a set $U\subset V(T)$ is not a descendant of any other vertex from $U$, then the events 
 $\{v\text{ is strongly }t\text{-stable / weakly } 
  0\text{-stable in }T^{(h)}\}$, $v\in U$, are (mutually) independent.
\label{rm:independancy_general}
\end{remark}
\begin{remark}
Due to symmetry, a strongly $t$-stable vertex has equal probabilities for both opinions at time $t$, i.e. the probability that $v \in V(T)\setminus \{R\}$ is strongly $t$-stable and $\xi_t(v) = -1$ equals to $p_{\mathrm{s}}(t,h,v)/2$.
\label{rm:equal_strong_probabilities}
\end{remark}
We also set
$$
p_{\mathrm{s}}(t) := \inf_{h\in\mathbb{Z}_{>0},\, v \in V(T=T^{(h)}) \setminus R:\, h_{T}(v) > 0} p_{\mathrm{s}}(t, h, v),\quad
p_{\mathrm{w}}(t) := \inf_{h\in\mathbb{Z}_{>0},\, v \in V(T=T^{(h)}) \setminus R} p_{\mathrm{w}}(t, h, v).
$$

The key  for the most of our calculations in the proof of Theorem~\ref{th:2} is the following Claim~\ref{cl:binary_calculations}.

\begin{claim}
We have the following bounds on $p_{\mathrm{s}}(3)$, $p_{\mathrm{s}}(2)$, $p_{\mathrm{s}}(1)$, $p_{\mathrm{s}}(0)$, and $p_{\mathrm{w}}(0)$:
$$
p_{\mathrm{s}}(3) > 0.4453,
\quad
p_{\mathrm{s}}(2) > 0.5617,
\quad
p_{\mathrm{s}}(1) \geq \frac{1}{2}p_{\mathrm{w}}(0)^2,
\quad
p_{\mathrm{s}}(0) > 0.5,\quad
p_{\mathrm{w}}(0) > 0.925.
$$
\label{cl:binary_calculations}
\end{claim}

We postpone the proof of Claim~\ref{cl:binary_calculations} to Section~\ref{sc:4:4}. In the subsequent section, we state and prove the lemma that asserts the lower bound in Theorem~\ref{th:2}.

\subsection{Proof of the lower bound}
\label{sc:4:2}

Here we prove the following:




\begin{lemma}
Let $T$ be an $R$-rooted perfect binary tree of height $h$. Let $\xi_0\in\{-1,1\}^{V(T)}$ be chosen uniformly at random. Then there exists $c_- > 0.5$ such that whp there is at least one not $\lfloor c_- \cdot h\rfloor$-stationary vertex.
\label{lm:binary_lower}   
\end{lemma}

\begin{proof}

Let us note that 
\begin{equation}
0.25003 < \left(\log_2\left(\frac{4}{0.2501}\right)\right)^{-1} < \left(\log_2\left(\frac{4}{p_{\mathrm{s}}(2)p_{\mathrm{s}}(3)}\right)\right)^{-1},
\label{eq:constants}
\end{equation}
due to Claim~\ref{cl:binary_calculations}. We fix an odd number $d$ satisfying
\begin{equation}
0.25003 \cdot h < d < \left(\log_2\left(\frac{4}{0.2501}\right)\right)^{-1}\cdot h.
\label{eq:restrictions_on_d}
\end{equation}

Let $R^* \in V(T)$ be at distance $d$
from a leaf, define $T^* := T_{R^*}$. First, we will define a class $\mathcal{F}$ of functions $\xi_0^*: V(T^*) \rightarrow \{-1, 1\}$ such that for each initial opinion $\xi_0^*$ on $V(T^*)$ from $\mathcal{F}$ there exists a vertex $u \in V(T^*)$ which is not $(2d-9)$-stable 
for every extension of $\xi_0^*$ to the entire $V(T)$. Secondly, we will show that this class is so large that whp, for a uniformly random $\xi_0\in\{-1,1\}^{V(T)}$, there exists $R^*$ as above so that $\xi_0|_{V(T^*)}\in\mathcal{F}$  leading to the fact that 
$\tau_T\geq$ $2d-8$ whp. The latter immediately implies the statement of Lemma~\ref{lm:binary_lower}   due to~\eqref{eq:restrictions_on_d}.

\paragraph{Definition of the desired set $\mathcal{F}$.}

Let $v_0, v_1$ be the children of $R^*$, $v_2, v_3$ be the children of $v_1$, and $v_4, v_5$ be the children of $v_3$. 
For a vertex $v \in V(T^*)$ and $r := h_{T^*}(v) \bmod 2$, let us define the event 
$$E(v) = \{v \textit{ is strongly }(r+2)\textit{-stable and } \xi_{r+2}(v) = -1\textit{, for any extension }\xi_0\textit{ to the entire }V(T)\}.$$
Note that this event is entirely described by the values of $\xi_0^*$ at $V(T_v)$. The following requirements on $\xi_0^*$ define the set $\mathcal{F}$ (for an illustration, we refer to Figure~\ref{fig:graph}).

\begin{figure}[h]
    \centering
    \includegraphics[scale=0.20]{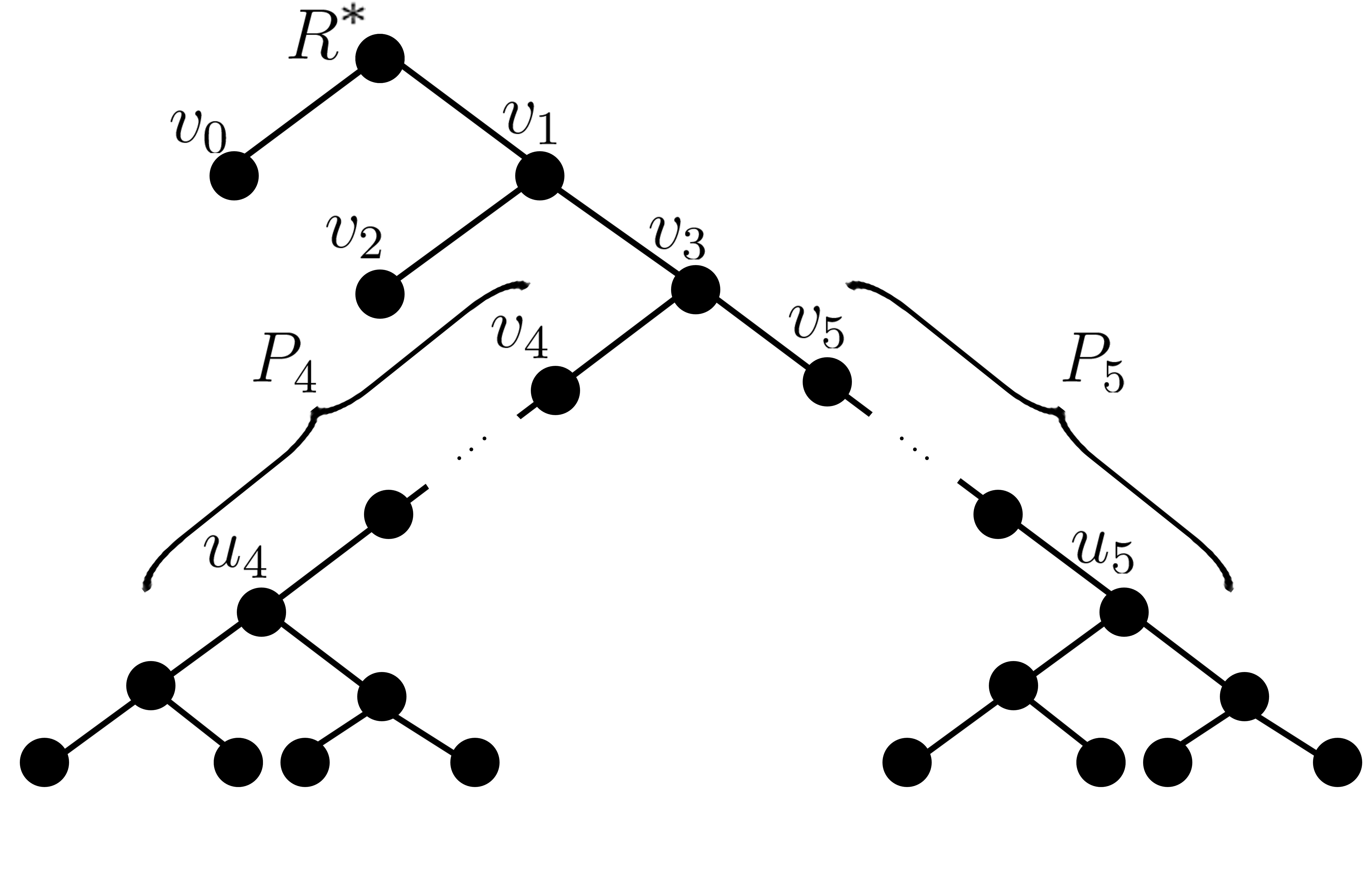}
    \caption{Structure of the tree $T^*$.}
    \label{fig:graph}
\end{figure}

\begin{enumerate}

\item $\xi^*_0|_{V(T_{R^*}) \setminus V(T_{v_3})}$ satisfies the following condition: for any its extension $\xi_0$ to the entire $V(T)$, $v_1$ is $0$-$stable$ with $\xi_0(v_1) = -1$.


\item 
there exists a path $P_4\subset T_{v_4}$ from $v_4$ to some vertex $u_4$ at distance $2$ from a leaf such that

\begin{itemize}
    \item $\xi^*_0(v) = 1$ for every $v\in V(P_4)$ at even distance from a leaf;
    

    \item $E(u)$ holds for the child $u\notin V(P_4)$ of every vertex $v \in V(P_4)\setminus\{u_4\}$;
    
    \item $\xi^*_0(u) =-1$ for all grandchildren $u$ of $u_4$.

\end{itemize}

\item 
there exists a path $P_5\subset T_{v_5}$ from $v_5$ to some vertex $u_5$ at distance $2$ from a leaf such that

\begin{itemize}
    \item $\xi_0^*(v) = 1$ for every $v\in V(P_5)$ at even distance from a leaf;
    

    \item $E(u)$ holds for the child $u\notin V(P_5)$ of every vertex $v \in V(P_5)\setminus\{u_5\}$;
    
    \item $\xi^*_0(u) = 1$ for all grandchildren $u$ of $u_5$.

\end{itemize}


\end{enumerate}

Now we show that the elements $\xi_0^*$ of $\mathcal{F}$ have the required property of the existence of not $(2d-9)$-stable vertices (i.e. if $\xi_0|_{T_{R^*}} = \xi_0^* \in \mathcal{F}$, than there exists a not $(2d-9)$-stable vertex). Namely, we will prove that: 
\begin{itemize}
\item[(1)] first, sequentially from $u_4$~to~$v_4$, vertices along $P_4$ adopt opinion $-1$ for the first time (among those time steps that have the same parity as their heights), 

\item[(2)] then, $v_3$ changes its opinion for the first time (among those time steps that have the same parity as $h_T(v_3)$),

\item[(3)] finally, after the above two `phases', vertices of $P_5\setminus\{u_5\}$ change their opinions for the first time (among those time steps that have the same parity as their heights) sequentially in the opposite direction from $v_5$ to $u_5$. 
\end{itemize}
This can be easily shown by induction, as we will see soon.  Consider the set of vertices $S:= V(P_4) \cup \{v_3\} \cup V(P_5) \setminus \{u_5\}$. The path induced by $S$ in $T$ defines on $S$ the order with the least element $u_4$.
 Let $t: S \rightarrow \mathbb{Z}_{\geq 0} \cup \{\infty\}$ maps $v \in S$ to the first time step such that $\xi_{t(v)}(v) = -1$, among the time steps with the same parity as the height of $v$. In order to show the existence of a not $(2d-9)$-stable vertex, it suffices to show by induction that, if $v \in S$ is a preceding neighbour of $u  \in S$ and $t(v) < \infty$, then $t(v) = t(u) - 1$. Indeed, by the definition of $\mathcal{F}$, $t(u_4) = 2$. Therefore, we get (1), (2), (3), and eventually, for the parent $u$ of $u_5$, $t(u) = 2d-7$, as required. So, let us finally show $t(v)=t(u)-1$ for $v$ and $u$ as above.

First, let us prove the base of induction: $t(v) = t(u) - 1$ if $t(v)<\infty$, where $u$ is the parent of $u_5$ and $v$ is the parent of $u$. Indeed, $u_5$ is $0$-stable since all its grandchildren are leaves and have initial opinion 1 and $u_5$ has opinion $1$ by the definition of $\mathcal{F}$. So, each of these five vertices remain opinion 1 at every even time step.

Now, suppose that the induction hypothesis is true for some $v' \in S$. We have to prove that it is also true for the preceding neighbour $u$ of $v'$. Indeed, let $v$ be the preceding neighbour of $u$. By the definition of $\mathcal{F}$, for every $z \in \{v, u, v'\}$ holds $\xi_{(h_{T^*}(z) \bmod 2)}(z) = 1$, hence $t(u) \geq 2$. In addition, the neighbour $w$ of $u$ other than $v$ and $v'$ is $(r+2)$-stable with opinion $\xi_{r+2}(w)=-1$, where $r = (h_{T^*}(v') \bmod 2)$. Therefore, $\min(t(v), t(v')) \geq t(w)$. So, $u$ changes its opinion to -1 exactly when at least one of its two neighbours $v,v'$ changes its opinion to -1, that is $t(u) = \min(t(v), t(v')) + 1$ (even for infinite $t(u)$). So, if $t(v) < \infty$, then $t(u) < \infty$. Also, by the induction step, $t(v') > t(u)$. Therefore, $t(u) - 1 = t(v)$, completing the inductive proof.



\paragraph{The existence of $R^*$.} We have to show that the three requirements in the definition of $\mathcal{F}$ meet sufficiently frequently. Let us note, that the first requirement is achieved if the following holds:
\begin{itemize}
    \item $\xi_0(v_0) = \xi_0(v_1) = \xi_1(v_2) = -1$;
    \item $v_0$ is strongly $0$-stable;
    \item $v_2$ is strongly $1$-stable.
\end{itemize}

Indeed, this implies that $-1$ is the common opinion of both $v_0$ and $v_1$ at every even time step and both $R^*$ and $v_2$ at every odd time step.
The probability of the latter event is at least $\frac{1}{8} p_{\mathrm{s}}(1) p_{\mathrm{s}}(0)\geq \frac{1}{16} p_{\mathrm{w}}(0)^2 p_{\mathrm{s}}(0) > 0$ due to Remark~\ref{rm:independancy_general}, Remark~\ref{rm:equal_strong_probabilities}, and Claim~\ref{cl:binary_calculations}. 


In what follows, we prove coinciding bounds for probabilities of the second and third requirements in the definition of $\mathcal{F}$. Since the proofs are literally the same (one is obtained from another by replacing $P_4, v_4, u_4$ with $P_5, v_5, u_5$ respectively, and by changing the number of the requirement), we provide the proof of the bound for the second requirement only.


There are exactly $2^{d-5}$ different paths from $v_4$ to a grandparent $u_4 \in V(T_{v_4})$ of a leaf. Fix one such path $P$. For every $i \in \{2, \ldots ,d-4\}$, we denote by $w_i$ the unique vertex at distance 1 from $P$ with $h_T(w_i) = i$. 
 The probability that the requirement $2$ from the definition of $\mathcal{F}$ holds for this $P$ playing the role of $P_4$ (we denote this event by $\mathcal{B}(P)$) is 
$$
q := 2^{-(d-3)/2-4} \cdot \prod\limits_{i=2}^{d-4}\frac{p_{\mathrm{s}}(({i\bmod 2}) + 2, h, w_i)}{2},
$$
due to Remark~\ref{rm:independancy_general} and Remark~\ref{rm:equal_strong_probabilities}. 

By the definitions,  
$$
q \geq 2^{-(d-3)/2-4} \cdot \left(\frac{p_{\mathrm{s}}(2)\cdot p_{\mathrm{s}}(3)}{4}\right)^{(d-5)/2}.
$$

Let us consider two paths $P,P'$ in $T_{v_4}$ from $v_4$ to grandparents of leaves and let assume that they fulfill the requirement 2 and separate at some vertex $v$ (i.e. $v$ is the closest to leaves vertex such that $v \in V(P) \cap V(P')$), then if $d(v, v_4) \leq d-7$ then the respective events $\mathcal{B}(P)$ and $\mathcal{B}(P')$ are disjoint. Indeed, let $w\in P$, $w'\in P'$ be two children of $v$. Since $P$ satisfies the second requirement, as we proved above, the vertex $w$ is not $((d(v, v_4) \bmod 2)+2)$-stable while the vertex $w'$ is $((d(v, v_4) \bmod 2)+2)$-stable (the latter follows immediately from the second requirement as $E(w')$ holds). At the same time, the opposite is also true since we have to replace $w$ and $w'$ in the second requirement for the path $P'$ --- a contradiction. 

So, if $P$ and $P'$ differ by at least one of their $d-7$ vertices closest to $v_4$, excluding $v_4$, then events $\mathcal{B}(P)$ and $\mathcal{B}(P')$ do not intersect. Therefore, probability that $\mathcal{B}(P)$ holds for some $P$ is at least

$$
2^{d-7} \cdot q \geq 2^{-(d-3)/2-6} \cdot (p_{\mathrm{s}}(2) \cdot p_{\mathrm{s}}(3))^{(d-5)/2}.
$$

Let us notice, that the three requirements in the definition of $\mathcal{F}$ are defined by initial opinions of vertices from three non-overlapping sets. In particular, the second and the third requirements are supported by vertices from $V(T_{v_4})$ and $V(T_{v_5})$, respectively. Thus, due to Remark~\ref{rm:independancy_general}, the three events that a uniformly random $\xi_0^*$ satisfies each of the requirements are independent.
We conclude the asymptotic lower bound
$$
\mathbb{P}(\xi_0^* \in \mathcal{F}) =  \Omega\left(\left(\frac{p_{\mathrm{s}}(2)p_{\mathrm{s}}(3)}{2}\right)^{d}\right).
$$


The events $\xi_0^*:=\xi_0|_{T_{R^*}} \in \mathcal{F}$ are independent over all $R^*$ with $h_T(R^*)=d$ by Remark~\ref{rm:independancy_general}. Due to~\eqref{eq:constants}~and~\eqref{eq:restrictions_on_d}, 
\begin{equation}
3 \cdot 2^{h-d-1} \cdot \mathbb{P}(\xi_0^* \in \mathcal{F}) = \Omega\left(2^{h-2d}(p_{\mathrm{s}}(2)p_{\mathrm{s}}(3))^{d}\right) = \omega(1),
\label{eq:prob_final_steps}
\end{equation}
and then whp the event $\xi_0^* \in\mathcal{F}$ holds for at least one $R^*$, as needed.

\end{proof}



The proof of the lower bound in Theorem~\ref{th:2} is completed, and we now switch to the upper bound which is separately stated in the subsequent section.

\subsection{Proof of the upper bound}
\label{sc:4:3}

This section is devoted to the proof of the upper bound in Theorem~\ref{th:2}, which is, in somewhat different terms, stated below.

\begin{lemma}
\label{lm:binary_upper}   
Let $T$ be an $R$-rooted perfect binary tree of height $h$. Let $\xi_0\in\{-1,1\}^{V(T)}$ be chosen uniformly at random. Then there exists $c_+ < 2/3$ such that whp all vertices are $\lceil c_+ \cdot h\rceil$-stationary.
\end{lemma}

We postpone the proof of Lemma~\ref{lm:binary_upper}  to the end of this section since it uses several auxiliary stability properties that are stated and proved in Section~\ref{sc:upper_auxuliary}. We complete the proof of Lemma~\ref{lm:binary_upper} in Section~\ref{sc:4:3:2}. 


Let us first briefly overview the proof. We note that a vertex $v$ changes its opinion at a time step $t$ if there is a path $P$ ending in $v$ such that the opinion ``travels'' along this path. We show that, for a fixed path $P$, the probability that an opinion ``travels'' along $P$ is exponentially small in terms of the length of the path and then apply the union bound over all paths of the same length, concluding the proof of the upper bound.

In order to get the exponentially small probability bound of the described event, we observe specific properties of the neighbourhood of $P$ that prevent the travel of an opinion along $P$ (due to Claim~\ref{cl:weak_stabilization} and Claim~\ref{cl:one_far_stabilization} stated in the next section). Probabilities of these preventing properties are estimated via Claim~\ref{cl:weak_calculations} and Claim~\ref{cl:strong_calculations} stated in the next section.

\subsubsection{Auxiliary claims}
\label{sc:upper_auxuliary}

In this section, we prove six claims that we will use in the proof of Lemma~\ref{lm:binary_upper} in Section~\ref{sc:4:3:2}.  Claim~\ref{cl:weak_grandparent} and Claim~\ref{cl:weak_rising} describe sufficient conditions for the weak $t$-stability to be inherited from some descendants of a vertex. Claim~\ref{cl:weak_stabilization} and Claim~\ref{cl:one_far_stabilization} describe sufficient conditions for stabilisation of opinions of a path between two weakly $t$-stable vertices. In particular, Claim~\ref{cl:weak_stabilization} deals with the case when these weakly $t$-stable vertices have the same opinion and Claim~\ref{cl:one_far_stabilization} deals with the case when these vertices have different opinions, requiring an additional property of being 1-close to stability. Claim~\ref{cl:weak_calculations} and Claim~\ref{cl:strong_calculations} describe bounds on probabilities of some specific events.

\begin{claim}
    
Let $t \in \mathbb{Z}_{\geq 0}$ and let $u, v \in V(T) \setminus \{R\}$ be such that 
\begin{itemize}
    \item $u$ is a grandchild of $v$;
    \item $u$ is weakly $t$-stable;
    \item $\xi_{t}(v) = \xi_t(u)$.
\end{itemize}
Then $v$ is weakly $t$-stable.
\label{cl:weak_grandparent}   
\end{claim}

\begin{proof}

Consider the vector of initial opinions $\tilde\xi_0$ defined as
\begin{itemize}
    \item $\tilde\xi_0(w) = \xi_{t}(w)$ for every $w \in V(T_v)$;
    \item $\tilde\xi_0(w) = \xi_{t}(v) =: \xi$ for every $w \in V(T) \setminus V(T_v)$.
\end{itemize}

Let us show that $\tilde\xi_0$ is the desired vector from the definition of $t$-stability, i.e. that $v$ is 0-stable for this vector of initial opinions. So, further in the proof we consider majority dynamics on $T$ with the vector of initial opinions $\tilde\xi_0$. Clearly, the parent of $v$ has opinion $\xi$ at every time step since it has degree 3 (note that $R$ has degree 3 by the definition of a binary tree). Let us prove that both $u$ and $v$ have opinion $\xi$ at every even time step. Assume the contrary and let $t_1 > 0$ be the first even time step such that this does not hold. Therefore, for every odd time step $t_0 < t_1$ the parent of $u$ has opinion $\xi$ as it is adjacent to both $u$ and $v$. So, $\tilde\xi_{t_1}(u) = \xi$ by the third equivalent definition of weak $t$-stability from Claim~\ref{cl:equivalent_weak_definitions}. Moreover, $\tilde \xi_{t_1}(v) = \xi$ since the parent of $v$ has opinion $\xi$ at time $t_1-1$ as we mentioned above --- a contradiction.
\end{proof}

\begin{claim}
Let $t \in \mathbb{Z}_{\geq 0}$ and let $u, v \in V(T) \setminus \{R\}$ be such that 
\begin{itemize}
    \item $u$ is a child of $v$;
    \item $u$ is weakly $t$-stable;
    \item $\xi_{t+1}(v) = \xi_t(u)$.
\end{itemize}
Then $v$ is weakly $(t+1)$-stable.
\label{cl:weak_rising}   
\end{claim}

\begin{proof}

Consider the vector of initial opinions $\tilde\xi_0$ defined as
\begin{itemize}
    \item $\tilde\xi_0(w) = \xi_{t+1}(w)$ for every $w \in V(T_v)$;
    \item $\tilde\xi_0(w) = \xi_{t+1}(v) =: \xi$ for every $w \in V(T) \setminus V(T_v)$.
\end{itemize}
As in the proof of the previous claim, we consider majority dynamics on $T$ with the vector of initial opinions $\tilde\xi_0$, and so it remains to prove that $v$ is 0-stable in this process. We note that the parent of $v$ has opinion $\xi$ always. Let us show that $u$ has opinion $\xi$ at every odd time step. This will clearly lead to $0$-stability of $v$. Assume the contrary: let $t_1 > 0$ be the first odd moment such that $\tilde\xi_{t_1}(u) \neq \xi$. Therefore, for every even time step $t_0 < t_1$ vertex $v$ has opinion $\xi$. So, $\tilde\xi_{t_1}(u) = \xi$ by the third equivalent definition of weak $t$-stability from Claim~\ref{cl:equivalent_weak_definitions}  --- a contradiction.
\end{proof}

\begin{claim}
Let $t \in \mathbb{Z}_{\geq 0}$, $\xi \in \{-1, 1\}$, $d \in 2\mathbb{Z}_{> 0}$, and a path $P=v_0v_1\ldots v_d$ in $T$ be such that 
\begin{itemize}
    \item $T_{v_0} \cap T_{v_d} = \varnothing$ (i.e. the path is not monotone);
    \item $\xi_t(v_i) = \xi$ for every even $i \in \{0, \ldots, d\}$;
    \item $v_0$ and $v_d$ are weakly $t$-stable.
\end{itemize}
Then, $v_i$ is $t$-stable for every even $i$.
\label{cl:weak_stabilization}   
\end{claim}

    

\begin{proof}

Let us prove Claim~\ref{cl:weak_stabilization} by contradiction. Let $t_1 > t$ be the first time step with the same parity as $t$ such that, for some even $i\in\{0,\ldots,d\}$, $\xi_{t_1}(v_i) \neq \xi$. If $i \notin \{0, d\}$ then, $\xi_{t_1-2}(v_{i-2}) = \xi_{t_1-2}(v_{i}) = \xi_{t_1-2}(v_{i+2}) = \xi$, hence $\xi_{t_1-1}(v_{i-1}) = \xi_{t_1-1}(v_{i+1}) = \xi$, and hence $\xi_t(v_i) = \xi$ --- a contradiction. If $i \in \{0, d\}$, then the only neighbour of $v_i$ from $V(P)$ is the parent of $v_i$ in $T$. Clearly, this neighbour has opinion $\xi$ at every time step $t_0\in\{t+1,t+3,\ldots,t_1-1\}$. 
Therefore, by Claim~\ref{cl:maintain_weak_value}, we get $\xi_{t_1}(v_i) = \xi$ --- a contradiction.

\end{proof}

Let us say that a non-leaf vertex $v \in V(T)$, $v \neq R$, is {\it 1-close to stability} if the following property holds. Let $\tilde\xi_0:\, V(T)\to\{-1,1\}$  be
such that $\tilde\xi_0|_{V(T_v)} = \xi_0|_{V(T_v)}$ and $\tilde\xi_t(v) \neq \tilde\xi_0(v)$ for some even $t>0$. Then $v$ is weakly $t$-stable w.r.t. $\tilde\xi_t$, where $t$ is the first even moment when $\tilde\xi_t(v) \neq \tilde\xi_0(v)$.


\begin{claim}
Let $t \in 2\mathbb{Z}_{\geq 0}$, $d \in 2\mathbb{Z}_{> 0}$, even $\ell \in [d]$, and a path $P=v_0 v_1 \ldots v_d$ in $T$ be such that 
\begin{itemize}
    \item $T_{v_0} \cap T_{v_d} = \varnothing$ (i.e. the path is not monotone);
    \item $\xi_t(v_i) = \xi_t(v_0)$ for every even $i \in [\ell-2]$;
    \item $\xi_t(v_i) = \xi_t(v_d)$ for every even $i \in [d-2] \setminus [\ell-2]$;
    \item $\xi_s(v_0) = \xi_0(v_0) \neq \xi_0(v_d) = \xi_s(v_d)$ for every even $s \in [t]$;
    \item $v_0$ and $v_d$ are weakly $0$-stable;
    \item $v_0$ and $v_d$ are 1-close to stability.
\end{itemize}
Then, for every even $t'\geq t$ and every even $i \in \{2,\ldots,d-2\}$, $\xi_{t'}(v_i) \in \{\xi_{t'}(v_{i-2}),\xi_{t'}(v_{i+2})\}$. In addition, for every even $t'\geq t$, $v_d$ is either weakly $t'$-stable or $t'$-stable.
\label{cl:one_far_stabilization}
\end{claim}

    

\begin{proof}

Let $t_1$ be the first even time step such that either $\xi_{t_1}(v_0) \neq \xi_0(v_0)$ or $\xi_{t_1}(v_d) \neq \xi_0(v_d)$. From the fourth condition of the claim, it follows that $t_1>t$. We shall prove that,
\begin{itemize}
    \item[A1]
    for every time step $t' \in \{t, t+1, \ldots, t_1-1\}$, there exists $\ell(t') \in [d+1]$ with the same parity as $t'$, such that,
    for every $i \in \{0, 1, \ldots, d\}$ with the same parity as $t'$, 
    $$
    \xi_{t'}(v_i) = \xi_t(v_0)\quad\text{ if and only if }\quad i < \ell(t');
    $$
    \item[A2] all $v_i$, over even $i \in \{0, 1, \ldots, d\}$, share the same opinion at time step $t_1$ and are $t_1$-stable.
\end{itemize}
Let us note that Claim~\ref{cl:one_far_stabilization} indeed follows from the latter statement. First, we have that $\xi_{t'}(v_i) \in \{\xi_{t'}(v_{i-2}),\xi_{t'}(v_{i+2})\}$ for every even $i \in \{2, \ldots, d-2\}$ both for all even $t' \in \{t, t+2, \ldots, t_1 - 2\}$ due to A1 and all even $t' \geq t_1$ due to A2. In addition, the vertex $v_d$ has not changed its opinion up to time $t_1-2$, so $v_d$ is weakly $t'$-stable for all even $t' \in \{t, t+2, \ldots, t_1 - 2\}$ due to Claim~\ref{cl:maintain_weak_stability} and the fifth requirement of the current claim. Finally, from A2, for every even $t' \geq t_1$, the vertex $v_d$ is $t'$-stable.

Let us prove A1 by induction on $t'\in \{t, t+1, \ldots, t_1-1\}$. The base of induction $t' = t$ follows from the second, the third, and the fourth requirements in the statement of the claim by taking $\ell(t') := \ell$. For the induction step from $t'-1$ to $t'$, let us note that a vertex sandwiched between two vertices with the same opinion will adopt that opinion at the next step. Also, for even time steps $t'$, we assumed that $\xi_{t'}(v_0) = \xi_{t}(v_0)$ and $\xi_{t'}(v_d) = \xi_{t}(v_d)$. So, $v_{\ell(t'-1)-1}$ is the only vertex, which opinion is not determined by the previous step. Hence, depending on its opinion $\ell(t')$ either equals $\ell(t'-1)-1$ or $\ell(t'-1)+1$. Finally, let us prove A2. Without loss of generality, let $v_0$ change its opinion at time step $t_1$. By Claim~\ref{cl:maintain_weak_stability}, $v_0$ is weakly $(t_1-2)$-stable.
By weak $(t_1-2)$-stability and Claim~\ref{cl:maintain_weak_value}, if $\xi_{t_1-1}(v_1) = \xi_{t_1-2}(v_0)$, then $\xi_{t_1}(v_1) = \xi_{t_1-2}(v_1)$.
Hence, $\xi_{t_1-1}(v_1) \neq \xi_{t_1-2}(v_0)$, which means that $\ell(t_1 - 1) = 1$. Hence, $v_d$ does not change its opinion at time step $t_1$ because $\ell(t_1 - 1) \neq d+1$.  Due to Claim~\ref{cl:maintain_weak_stability}, $v_d$ is weakly $t_1$-stable, and $v_0$ is weakly $t_1$-stable since it is 1-close to stability. Lastly, as $\ell(t_1 - 1) = 1$, for every even $i \in [d]$ holds $\xi_{t_1}(v_i) = \xi_{t_1}(v_d) = \xi_{t_1}(v_0)$. So, the statement follows from application of Claim~\ref{cl:weak_stabilization} to $v_0v_1\ldots v_d$.

\end{proof}


\begin{claim}
Let $\xi_0\in\{-1,1\}^{V(T)}$ be uniformly random.
For any non-leaf vertex $v \in V(T)$, $v \neq R$,
$$
 \mathbb{P}(v\text{ is }1\text{-close to stability})\geq 0.9957. 
$$
\label{cl:weak_calculations}   

\end{claim}

\begin{proof}
Let us note that if $h_T(v) = 1$, then $v$ is strongly 0-stable and, hence, $v$ is 1-close to stability. So, we may assume that $h_T(v) \geq 2$.
 Let us first describe the proof strategy. For a child $u$ of $v$, we will define three specific events expressed only in terms of initial opinions on $V(T_u)$. We will show that if $v$ is not 1-close to stability, then for some child of $v$ at least one of the three events holds. Then, we will bound probabilities of these three events and use these bounds to derive the statement of Claim~\ref{cl:weak_calculations}.\\
 
 

Let $\xi^*_0: V(T) \to \{-1, 1\}$ be any vector of initial opinions such that $v$ is not 1-close to stability with respect to $\xi^*_0$. Let $\xi := \xi^*_0(v)$. Then, there exists a vector of initial opinions $\tilde\xi^*_0: V(T) \to \{-1, 1\}$ and a time step $t_1 \in \mathbb{Z}_{\geq 0}$ such that the following holds:
\begin{itemize}
    \item $\tilde\xi^*_0|_{V(T_v)} = \xi^*_0|_{V(T_v)}$;
    \item $t_1$ is the first even time step such that $\tilde\xi^*_{t_1}(v) \neq \xi$;
    \item $v$ is not weakly $t_1$-stable.
\end{itemize}
 
Let us prove that for some child $u$ of $v$ one of the following disjoint properties holds.

\begin{itemize}
    \item [A1] At least one of the children of $u$ has initial opinion $\xi$. For every child $v_1$ of $u$ with initial opinion $\xi$, there exists an even time step $t_2 < t_1$ such that $\tilde\xi^*_{t_2}(v_1) \neq \xi$ for the first time and $v_1$ is not $t_2$-stable.
    \item [A2] Both children of $u$ have initial opinion opposite to $\xi$ and are not weakly $0$-stable.
    \item [A3] Both children of $u$ have initial opinion opposite to $\xi$. One child $v_1$ of $u$ is weakly $0$-stable. There exists an even time step $t_2 < t_1$ such that the other child $v_2$ of $u$ has opinion $\xi$ at time step $t_2$ for the first time and $v_2$ is not weakly $t_2$-stable.
\end{itemize}

Suppose towards contradiction that A1, A2, and A3 do not hold. Let us prove that $v$ is weakly $t_2$-stable, which will lead to a contradiction.

Note that $t_1\geq 2$. Let $u$ be a child of $v$ such that $\tilde\xi^*_{t_1-1}(u) \neq \xi$. Then, both children of $u$ have opinion opposite to $\xi$ at time step $t_1-2$. 

Suppose at least one child of $u$ has opinion $\xi$ at time step zero. For a child $v'$ of $u$, let $t_2(v')\leq t_1-2$ be the first even time step such that $\tilde\xi^*_{t_2}(v') \neq \xi$. Since, by assumption, A1 does not hold for $u$, there exist a child $v_1$ of $u$ such that $v_1$ is $t_2(v_1)$-stable. So, $\tilde\xi^*_{t_1}(v) = \tilde\xi^*_{t_1}(v_1)$ and $v_1$ is $t_1$-stable. Hence, $v$ is weakly $t_1$-stable --- a contradiction. 

Now let both children of $u$ have initial opinion opposite to $\xi$. If both of them are weakly $0$-stable, then both of them are $0$-stable due to Claim~\ref{cl:weak_stabilization}. So, $v$ is weakly $t_1$-stable --- a contradiction. Due to the negation of A2, at least one of the children of $u$ has to be weakly 0-stable. So it remains to consider the case when only one child $v_1$ of $u$ is weakly $0$-stable. Let $t_2>0$ be the first even time step such that the other child $v_2$ has opinion $\xi$ which is opposite to $\tilde\xi^*_0(v_2)$. If $t_2 < t_1$, then $v_2$ is weakly $t_2$-stable, due to the negation of $A3$. Since $\tilde\xi^*_{t_2}(v)=\xi$, we get that $\tilde\xi^*_{t_2+1}(u)=\xi$ as well. 
Due to weak $t_2$-stability and Claim~\ref{cl:maintain_weak_value}, $\tilde\xi^*_{t_2+2}(v_2) = \xi$.
Proceeding by induction, we get that, for every even time step $t_2 \leq t < t_1$, $\tilde\xi^*_{t}(v_2) = \xi$. Hence, $\tilde\xi^*_{t_1-1}(u) = \xi$ --- a contradiction. 
Lastly, if $t_2 \geq t_1$, then, in the same way as above, applying inductively Claim~\ref{cl:maintain_weak_value}, we get $\tilde\xi^*_{0}(v_1) = \tilde\xi^*_{2}(v_1) = \ldots = \tilde\xi^*_{t_2}(v_1)$.
So, $v_1$ is weakly $t_1$-stable by Claim~\ref{cl:maintain_weak_stability}. Due to Claim~\ref{cl:weak_grandparent}, $v$ is weakly $t_1$-stable --- a contradiction.\\

Let us fix an opinion $\xi\in\{-1,1\}$. Everywhere below we assume that the event $\xi_0(v)=\xi$ holds deterministically. Now, fix any child $u$ of $v$, and, for every $i\in\{1,2,3\}$, denote by $p_i$ the probability of the following event $\mathcal{Q}_i$: there exists an extension $\tilde\xi_0$ of $\xi_0|_{V(T_u)}$ to the whole tree and time step $t_1$ such that 
\begin{itemize}
\item $t_1$ is the first even time step when $\tilde\xi_{t_1}(v) \neq \tilde\xi_0(v)=\xi$; \item  the property A$i$ holds with $\tilde\xi^*_t$ replaced by $\tilde\xi_t$.
\end{itemize}
The rest of the proof is devoted to estimating probabilities $p_1,p_2,p_3$ since, due to Remark~\ref{rm:independancy_general} and the union bound, it is sufficient for us to show $(1 - p_1 - p_2 - p_3)^2 > 0.9957$.


{\bf Let us start with $p_1$.} First, note that at least one child of $u$ has initial opinion $\xi$ with probability $3/4$. Let us now assume that, for some $\xi_0|_{V(T_u)}$, the considered event holds, and that an extension $\tilde\xi_0|_{V(T)}\to\{-1,1\}$ certifies the validity of this event. Let $v_1$ be a child of $u$ with initial opinion $\tilde\xi_0(v_1) = \xi$. Then $v_1$ has to change its opinion for the first time at some even time step $t_2 < t_1$. 
Hence, $v_1$ is not weakly $0$-stable by Claim~\ref{cl:maintain_weak_value}.

 Let us prove that $h_T(v_1) \geq t_2$. First, let us note that for every $t > 1$ if a vertex changed its opinion to the opposite to $\xi$ at time step $t+1$, then one of its neighbours changed its opinion to the opposite to $\xi$ at time step $t$. Now, $\tilde\xi_{t_2-1}(u) = \xi$ since $\tilde\xi_{t_2-2}(v)=\tilde\xi_{t_2-2}(v_1)=\xi$, so at least one of the children of $v_1$ should change opinion at time step $t_2-1$. Thus, we may apply a similar argument to the child of $v_1$ and then proceed by induction: For every $i\geq 1$, we have a vertex $v_i$ and its child $v_{i+1}$ such that $\tilde\xi_{t_2-i-1}(v_i)=\xi$ and $v_{i+1}$ changed its opinion to the opposite to $\xi$ at time moment $t_2-i$ (in particular, $\tilde\xi_{t_2-i-2}(v_{i+1})=\xi)$, then there is a child $v_{i+2}$ of $v_{i+1}$ that changed its opinion at time $t_2-i-1$.
 Eventually we get a descending path of vertices of length $t_2-2$ staring at $v_1$, where, in particular, for the last vertex in the path $v_{t_2-1}$ the following holds: $\tilde\xi_0(v_{t_2-1}) = \xi \neq \tilde\xi_2(v_{t_2-1})$. So, $h_T(v_{t_2-1}) \neq 1$, as that would imply that $v'$ is 0-stable. Also, $h_T(v_{t_2-1}) \neq 0$, as that would mean that $\tilde\xi_2(v_{t_2-1})$ equals opinion of its parent at time step $1$, which is $\xi$ --- a contradiction. So, $h_T(v') \geq 2$, as needed.

Let us describe the way we bound the probability $p_1$. Let us specify one of the two children of $u$ and call it {\it canonical}. For an even time step $\tau > 0$, let $\mathcal{Q}_1^*(\tau)$ denote the following event: there exists an extension $\tilde\xi_0$ of $\xi_0|_{V(T_u)}$ to the whole tree such that 
\begin{itemize}
\item $\tilde\xi_t(v)= \tilde\xi_0(v)=\xi$ for all even $t\leq\tau$; 
\item at least one of the children of $u$ has initial opinion $\xi$;
\item $\tilde\xi_{\tau}(v_1) \neq \xi$ for the first time, where $v_1$ is a (random) child of $u$ defined as follows:
if both children of $u$ have initial opinion $\xi_0(u)=\xi$, then $v_1$ is the canonical child of $u$, otherwise, $v_1$ is the only child of $u$ with the initial opinion $\xi$.
\end{itemize}

Let us prove that for every time step $t \leq \tau$ and every descendant $v'$ of $v_1$, which distance from $v_1$ has the same parity as $t$, the opinion $\tilde\xi_t(v')$ is described by initial opinions of vertices from $V(T_{u})$ only, given that the event $\mathcal{Q}^*_1(\tau)$ holds. More formally, let us consider two different vectors of initial opinions $\tilde\xi^1_0$ and $\tilde\xi^2_0$ that agree with $\xi_0$ at $V(T_u)$ and certify the validity of the events $\mathcal{Q}_1^*(\tau_1)$ and $\mathcal{Q}_1^*(\tau_2)$ for some time steps $\tau_1$ and $\tau_2$ respectively. Let $\tau$ be the minimum of $\tau_1$ and $\tau_2$. Let us prove by induction on $t$ that for every time step $t \leq \tau$ and every descendant $v'$ of $v_1$, which distance from $v_1$ has the same parity as $t$, the equality $\tilde\xi^1_t(v') = \tilde\xi^2_t(v')$ holds. In particular, $\tau_1 = \tau_2$, as $\tilde\xi^1_{\tau}(v_1) = \tilde\xi^2_{\tau}(v_1) \neq \xi$ for the first time, and then the statement follows as needed. The base of induction for $t=0$ holds since $\tilde\xi^1_0|_{V(T_{v_1})} = \tilde\xi^2_0|_{V(T_{v_1})} = \xi_0|_{V(T_{v_1})}$. The induction step follows from the definition of majority dynamics for every vertex, except for $v_1$ at any even time step $t$. Finally, for an even time step $t$ the equality $\tilde\xi^1_t(v_1) = \tilde\xi^2_t(v_1)$ holds as both children of $v_1$ have same opinions at time step $t-1$ in both processes and $\tilde\xi^1_{t-1}(u) = \tilde\xi^2_{t-1}(u) = \xi$, since $\tilde\xi^1_{t-2}(v_1) = \tilde\xi^1_{t-2}(v) = \tilde\xi^2_{t-2}(v_1) = \tilde\xi^2_{t-2}(v) = \xi$.

Let $V(\tau)$ be the set of vertices of $T_u$ at distance at most $\tau + 1$ from $u$. Let us first prove the following:
$$\{\exists \tau\ \mathcal{Q}_1^*(\tau)\} = \bigsqcup_{\tau \in 2\mathbb{Z}_{>0}}\bigsqcup_{\chi \in \{-1,1\}^{V(\tau)}}\biggl[ \mathcal{Q}_1^*(\tau)\cap\{\xi_0|_{V(\tau)} = \chi\}\biggr].
$$
In other words, we have to prove that, for distinct pairs $(\tau_1, \chi_1)$ and $(\tau_2, \chi_2)$, the two events in the right-hand side of the latter equality are disjoint.
 Suppose, for some $\xi_0$, both events hold for some distinct $(\tau_1, \chi_1)$ and $(\tau_2, \chi_2)$ simultaneously. First, we have already proved that $\tau$ such that $\mathcal{Q}^*_1(\tau)$ holds is uniquely defined for $\xi_0|_{V(T_u)}$. Therefore, $\tau_1 = \tau_2=:\tau$. But then, the equality $ \chi_1 = \chi_2$ follows, since both vectors equal $\xi_0|_{V(\tau)}$.

Next, let us note that the event $\mathcal{Q}_1$ implies the event $\{\exists \tau\ \mathcal{Q}_1^*(\tau)\}$. So, 
$$
\mathbb{P} (\mathcal{Q}_1) = \mathbb{P} (\mathcal{Q}_1 \cap \{\exists \tau\ \mathcal{Q}_1^*(\tau)\}).
$$
Thus,
\begin{align*}
\mathbb{P} (\mathcal{Q}_1) 
&= \mathbb{P} \left(\mathcal{Q}_1\cap \bigsqcup_{\tau \in 2\mathbb{Z}_{>0}}\bigsqcup_{ \chi \in \{-1, 1\}^{V(\tau)}} \biggl[\mathcal{Q}_1^*(\tau)\cap\{\xi_0|_{V(\tau)} = \chi\}\biggr]\right) \\
&= \sum_{\tau \in 2\mathbb{Z}_{>0}}\ \sum_{\chi \in \{-1, 1\}^{V(\tau)}} \mathbb{P} \left(\mathcal{Q}_1\cap \mathcal{Q}_1^*(\tau)\cap\{\xi_0|_{V(\tau)} = \chi\}\right) \\
&= \sum_{\tau \in 2\mathbb{Z}_{>0}}\  \sum_{\chi \in \{-1, 1\}^{V(\tau)}} \mathbb{P} \left(\mathcal{Q}_1 \mid \mathcal{Q}_1^*(\tau)\cap \{\xi_0|_{V(\tau)} = \chi\}\right) \cdot \mathbb{P}\left(\mathcal{Q}_1^*(\tau)\cap \{\xi_0|_{V(\tau)} = \chi\}\right).
\end{align*}
Below, for all pairs $(\tau,\chi)$, we give the same upper bound $q_1$ on $\mathbb{P} \left(\mathcal{Q}_1 \mid \mathcal{Q}_1^*(\tau)\cap\{\xi_0|_{V(\tau)} = \chi\}\right)$. We determine the exact value of $q_1$ later. As soon as this bound is determined, we may conclude that
$$
\mathbb{P} (\mathcal{Q}_1) \leq q_1 \cdot \sum_{\tau \in 2\mathbb{Z}_{>0}}\  \sum_{\chi \in \{-1, 1\}^{V(\tau)}} \mathbb{P}\left(\mathcal{Q}_1^*(\tau)\cap \{\xi_0|_{V(\tau)} = \chi\}\right) = q_1 \cdot \mathbb{P} (\exists \tau\ \mathcal{Q}_1^*(\tau)).
$$

Let us note that the event $\{\exists \tau\ \mathcal{Q}_1^*(\tau)\}$ implies the event that at least one of the children of $u$ has initial opinion $\xi$ and the vertex $v_1$ (determined in the definition of $\mathcal{Q}_1^*(\tau)$) is not weakly 0-stable. Indeed, consider the random variable $\tau$ such that $\mathcal{Q}_1^*(\tau)$ holds and an extension $\tilde\xi_0$ whose existence is claimed by the event $\mathcal{Q}_1^*(\tau)$. For any odd time step $t<\tau$ the vertex $u$ has opinion $\tilde\xi_t(u)=\xi$, so the fact that $v_1$ changes its opinion at time step $\tau$ contradicts the third equivalent definition of the weak 0-stability. So, $\mathbb{P}(\exists \tau\ \mathcal{Q}_1^*(\tau))$ does not exceed the probability of the event it implies, which is at most $3/4 \cdot (1-p_{\mathrm{w}}(0))$. Thus, finally, we may conclude that \begin{equation}
p_1 \leq q_1 \cdot 3/4 \cdot (1-p_{\mathrm{w}}(0)).
\label{eq:from_q1_to_p1}
\end{equation}

Let us now prove that $q_1 \leq 1-\left(1 - (1 - p_{\mathrm{w}}(0))^2\right)^2$. The main advantage for us is that the condition $\mathcal{Q}_1^*(\tau) \cap \{\xi_0|_{V(\tau)} = \chi\}$ is defined by the initial opinions of vertices from $V(\tau)$, hence any event independent from these initial opinions is also independent from the condition. 

For the rest of the proof, let us fix some even $\tau > 0$ and some $\chi \in V(\tau)$. We also assume that the event $\mathcal{Q}_1^*(\tau) \cap \{\xi_0|_{V(\tau)} = \chi\}$ holds and that  $\mathbb{P} \left(\mathcal{Q}_1 \mid \mathcal{Q}_1^*(\tau)\cap \{\xi_0|_{V(\tau)} = \chi\}\right)\neq 0$.

In what follows, we assume that $\mathcal{Q}_1$ holds and fix an extension $\tilde\xi_0$ whose existence is claimed by this event. Let $t_2$ be the time step from the definition of $\mathcal{Q}_1$ (with respect to $\tilde\xi_0$). We know that the event $\mathcal{Q}^*_1(t_2)$ holds and its validity is certified by the extension $\tilde\xi_0$. From this, we conclude that the random variable $\tau$ such that $\mathcal{Q}_1^*(\tau)$ holds equals $t_2$ from the definition of $\mathcal{Q}_1$. 

Let us prove that all four grandchildren of $v_1$ have the opinion $\tilde\xi_{t_2-2}(\cdot)$ opposite to $\xi$ at time step $t_2-2$ for the first time (among even time steps). 
First of all we note that $\tilde\xi_{t_2-1}(u) = \xi$ since $\tilde\xi_{t_2-2}(v_1) = \tilde\xi_{t_2-2}(v) = \xi$. Let us also recall that $\tilde\xi_{t_2}(v_1)\neq\xi$. It may only happen when both children of $v_1$ have opinion opposite to $\xi$ at time step $t_2-1$. At the same time, we have that $\tilde\xi_{t_2-2}(v_1)=\xi$. We immediately get that, indeed, all four grandchildren of $v_1$ have the other opinion at time $t_2-2$. This step is indeed the first time when it happens since otherwise $v_1$ would change its opinion earlier than at $t_2$.

Let us note that if a vertex has opinion opposite to $\xi$ at time step $t+1$, then at least one of its children has opinion opposite to $\xi$ at time step $t$. So, for every $i \in [4]$, every grandchild $u_i$ of $v_1$ has a descendant $w_i \in V(T_{u_i})$ at distance $t_2-2$ such that for every $j \in \{0, 1, \ldots, t_2 - 2\}$ and every vertex $w$ in the shortest path between $w_i$ and $u_i$ such that $d(w_i, w) = j$ has $\tilde\xi_j(w) \neq \xi$. Let us note that since dynamics on $T_{v_1}$ on $[0,t_2]$ does not depend on the choice of initial opinions outside of $T_u$, we may choose these $w_i$ independently of the extension of $\xi_0|_{V(T_u)}$ to the entire tree.
In addition, for the vertex $w$ in the path between $u_i$ and $w_i$ at distance $i$ from $w_i$, the initial opinion $\tilde\xi_i(w)$ is uniquely determined by the opinions of the vertices at the distance from $w$ not exceeding $i$. So, we may choose a set $\{w_1,w_2,w_3,w_4\}$ such that each $w_i$ relates to $u_i$ in the described way by knowing the initial opinions of $V(t_2)$ only. Let us fix any such set.

Note that events that $w_i$ are weakly 0-stable are independent from each other and the fixed condition $\mathcal{Q}_1^*(\tau) \cap \{\xi_0|_{V(\tau)} = \chi\}$. So, the probability that
\begin{itemize}
    \item at least one of $w_1$ and $w_2$ is weakly 0-stable and
    \item at least one of $w_3$ and $w_4$ is weakly 0-stable
\end{itemize}
is at least $\left(1 - (1 - p_{\mathrm{w}}(0))^2\right)^2$. Due to Claim~\ref{cl:weak_rising}, the weak 0-stability of $w_i$ together with the fact that opinions of vertices along the path from $w_i$ to $u_i$ remain the same (as time grows simultaneously with moving forward towards $u_i$) imply the weak $(t_2-2)$-stability of $u_i$. Hence, the former event implies that 
\begin{itemize}
    \item at least one of $u_1$ and $u_2$ is weakly $(t_2-2)$-stable and
    \item at least one of $u_3$ and $u_4$ is weakly $(t_2-2)$-stable.
\end{itemize}
If $u_i$ is weakly $(t_2-2)$-stable then it is also weakly $t_2$-stable with the same opinion --- here we apply the fact that its parent has the same opinion at time step $t_2-1$, 
Claim~\ref{cl:maintain_weak_value}, and Claim~\ref{cl:maintain_weak_stability}. So,
\begin{itemize}
    \item at least one of $u_1$ and $u_2$ is weakly $t_2$-stable with opinion opposite to $\xi$ and
    \item at least one of $u_3$ and $u_4$ is weakly $t_2$-stable with opinion opposite to $\xi$.
\end{itemize}
Due to Claim~\ref{cl:weak_stabilization}, this leads to $t_2$-stability and hence $t_1$-stability of $v_1$ --- a contradiction. Altogether, our upper bound on $p_1$ is as follows
\begin{equation}
p_1 \leq \frac{3}{4} \cdot (1 - p_{\mathrm{w}}(0)) \cdot \left(1-\left(1 - (1 - p_{\mathrm{w}}(0))^2\right)^2\right) \stackrel{\text{Claim}~\ref{cl:binary_calculations}}\leq 0.00064.
\label{eq:p1}
\end{equation}

{\bf Next, let us bound $p_2$.} Both children of $u$ have initial opinions opposite to $\xi$ with probability $1/4$. Next, both of them are not weakly 0-stable with probability at most $(1 - p_{\mathrm{w}}(0))^2$ since these two events are independent. Moreover, they are independent of the initial opinions of these two children. So, 
\begin{equation}
p_2 \leq \frac{1}{4} \cdot (1 - p_{\mathrm{w}}(0))^2 \leq 0.00141.
\label{eq:p2}
\end{equation}

{\bf Lastly, we bound $p_3$.} The proof strategy here is analogous to the proof of the bound for $p_1$. We outline the differences below. 

There are two ways to choose which child of $u$ plays the role of $v_1$ (and then the other child plays the role of $v_2$). In the same way as in the proof for $p_1$, we get that $h_{T}(v_2) \geq t_2$.

Let $\mathcal{Q}_3^*(\tau)$ denote the following event: there exists an extension $\tilde\xi_0$ of $\xi_0|_{V(T_u)}$ to the whole tree such that
\begin{itemize}
\item $\tilde\xi_t(v)= \tilde\xi_0(v)=\xi$ for all even $t\leq\tau$; 
\item both children of $u$ have initial opinion opposite to $\xi$;
\item some (random) child $v_1$ of $u$ is weakly $0$-stable;
\item the other child $v_2$ of $u$ has opinion $\xi$ at time step $\tau$ for the first time.
\end{itemize}

Let us note that the vertex $v_2$ from the definition of the event $\mathcal{Q}_3^*(\tau)$ is not weakly 0-stable, since otherwise it is even 0-stable by Claim~\ref{cl:weak_stabilization} and hence it does not change its opinion at time step $\tau$. So, labels $v_1$ and $v_2$ are uniquely defined by $\xi_0|_{V(T_u)}\in\mathcal{Q}_3^*(\tau)$.

Let $V(\tau)$ be the union of $V(T_{v_1})$, the set $\{u\}$, and the set of vertices of $T_{v_2}$ at distance at most $\tau$ from $v_2$. Then we may prove the following in the same way as for $p_1$:
$$\{\exists \tau\ \mathcal{Q}_3^*(\tau)\} = \bigsqcup_{\tau \in 2\mathbb{Z}_{>0}}\bigsqcup_{\chi \in \{-1,1\}^{V(\tau)}}\biggl[ \mathcal{Q}_3^*(\tau)\cap\{\xi_0|_{V(\tau)} = \chi\}\biggr].
$$ 

Next, let us note that the event $\mathcal{Q}_3$ implies the event $\{\exists \tau\ \mathcal{Q}_3^*(\tau)\}$. So, as for $p_1$,
\begin{align*}
\mathbb{P} (\mathcal{Q}_3) 
&= \sum_{\tau \in 2\mathbb{Z}_{>0}}\  \sum_{\chi \in \{-1, 1\}^{V(\tau)}} \mathbb{P} \left(\mathcal{Q}_3 \mid \mathcal{Q}_3^*(\tau)\cap \{\xi_0|_{V(\tau)} = \chi\}\right) \cdot \mathbb{P}\left(\mathcal{Q}_3^*(\tau)\cap \{\xi_0|_{V(\tau)} = \chi\}\right).
\end{align*}
Below, for all pairs $(\tau,\chi)$, we give the same upper bound $q_3$ on $\mathbb{P} \left(\mathcal{Q}_3 \mid \mathcal{Q}_3^*(\tau)\cap\{\xi_0|_{V(\tau)} = \chi\}\right)$. We determine the exact value of $q_3$ later. As soon as this bound is determined, we may conclude that
$$
\mathbb{P} (\mathcal{Q}_3) \leq q_3 \cdot \mathbb{P} (\exists \tau\ \mathcal{Q}_3^*(\tau)).
$$
As was mentioned above, the event $\{\exists \tau\ \mathcal{Q}_3^*(\tau)\}$ distinguishes between the two children of $u$ and implies that one of them ($v_2$) is not weakly 0-stable. In addition, it says that $\xi_0(v_1) = \xi_{0}(v_2) \neq \xi$. So,
$$\mathbb{P} (\exists \tau\ \mathcal{Q}_3^*(\tau)) \leq 1/4 \cdot 2 \cdot (1-p_{\mathrm{w}}(0)).$$

Let us finally prove that $q_3 \leq (1-p_{\mathrm{w}}(0))^4 $. In what follows, we condition all events on $\mathcal{Q}_3^*(\tau)\cap \{\xi_0|_{V(\tau)} = \chi\}$. In particular, under this condition, the random variable $t_2$ from the definition of the event $\mathcal{Q}_3$ gets the value $\tau$  since $\mathcal{Q}_3$ implies $\mathcal{Q}_3^*(t_2)$.

By induction over odd $t\leq t_2-1$ and by Claim~\ref{cl:maintain_weak_value}, we have that $\tilde\xi_{t+1}(v_1)=\tilde\xi_t(u)\neq\xi$ for all odd $t\leq t_2-1$. So, $\tilde\xi_{t_2-1}(u) \neq \xi$. Since $\tilde\xi_{t_2}(v_2) = \xi$, we get that all 4 grandchildren of $v_2$ have opinion $\xi$ at time step $t_2 - 2$. 

Applying the same argument as for $p_1$, we may assign to every grandchild $u_i$, $i\in[4]$, of $v_2$ a canonical descendant $w_i$ located at distance $t_2 - 2$ from $u_i$. Here, we again get that each $w_i$ is weakly 0-stable independently of all the other events. Due to Claim~\ref{cl:weak_rising}, $v_2$ is weakly $t_2$-stable if at least one of $w_i$ is weakly 0-stable, which contradicts $\mathcal{Q}_3$. 

We conclude that
\begin{equation}
p_3 \leq \frac{1}{4} \cdot 2 \cdot (1 - p_{\mathrm{w}}(0)) \cdot (1-p_{\mathrm{w}}(0)) ^4 \stackrel{\text{Claim}~\ref{cl:binary_calculations}}\leq 0.00001.
\label{eq:p3}
\end{equation}
From~\eqref{eq:p1},~\eqref{eq:p2},~and~\eqref{eq:p3}, we get the required lower bound $(1 - p_1 - p_2 - p_3)^2 > 0.9957$, completing the proof.
\end{proof}

\begin{claim}
For all large enough even integer $t$ the following holds. Fix any positive integer $h$ and a perfect binary tree of depth $h$. Let $\xi_0\in\{-1,1\}^{V(T)}$ be chosen uniformly at random. Then, for any non-leaf vertex $v \in V(T)$, $v \neq R$, the probability that 
\begin{itemize}
    \item there exists an extension $\tilde\xi_0$ of $\xi_0|_{V(T_v)}$ to the whole tree $T$ such that $\tilde \xi_{t}(v) = 1$;
    \item $v$ is not weakly $0$-stable w.r.t. $\xi_0$
\end{itemize}
is less than $0.038$.
\label{cl:strong_calculations}   
\end{claim}

\begin{proof}
Let $\varepsilon>0$ be small enough and let
$t=t(\varepsilon)$ be a large even integer. Let $T$ be a perfect binary tree. Consider the following events:
\begin{itemize}
\item $\mathcal{Q}_1=\{v$ is not weakly 0-stable$\}$;

\item $\mathcal{Q}_2=\{\xi_0(v)=1$ and, for every extension $\tilde\xi_0$ of $\xi_0|_{V(T_v)}$ to the whole tree, $\tilde\xi_t(v)=-1\}$.
\end{itemize}

Then the probability of the desired event is at most
$$
 \mathbb{P}(\mathcal{Q}_1\wedge\neg\mathcal{Q}_2)=(1-\mathbb{P}(\mathcal{Q}_2\mid\mathcal{Q}_1))\mathbb{P}(\mathcal{Q}_1)\leq(1-\mathbb{P}(\mathcal{Q}_2\mid\mathcal{Q}_1))(1 - p_{\mathrm{w}}(0)).
$$
It remains to bound from below $\mathbb{P}(\mathcal{Q}_2\mid\mathcal{Q}_1)$.

 Consider $\tilde\xi_0: V(T) \rightarrow \{-1, 1\}$ such that 
\begin{itemize}
    \item $\tilde\xi_0(w) = \xi_0(w)$ for every $w \in V(T_v)$;
    \item $\tilde\xi_0(w) = 1$ for every $w \in V(T) \setminus V(T_v)$.
\end{itemize} 

If $v$ is not weakly 0-stable, then, due to the definition of weak 0-stability, there exists an even time step $\tau$ such that $\tilde\xi_{\tau}(v) = -1$ for the first time. If such a time step does not exist, let $\tau=\infty$. Consider the event $\mathcal{Q}^*=\mathcal{Q}^*(h(T_v),t)=\{\tau<t\}$. Let us observe that $\mathbb{P}(\mathcal{Q}^*\mid\mathcal{Q}_1)=1$ when $h(T_v)<t$. Let us show this by contradiction: suppose $\tau \geq t > h(T_v)$. The parent of $v$ has opinion $1$ at time step $\tau - 1$ due to the definition of $\tilde\xi_0$, so $v$ can change its opinion to $-1$ at time step $\tau$ only if it has a child that changes opinion to $-1$ at time step $\tau - 1$. Next, we proceed by induction: apply the same argument to this child, then to a specific grandchild, etc. So, since $h(T_v) < \tau$, there is  a descendant $u$ of $v$ with $d(u, v) = h(T_v)-1$ such that it changes opinion to $-1$ at time step $\tau - h(T_v) + 1 \geq 2$. However, this is impossible, since at time step $\tau - h(T_v) + 1$ vertex $u$ has the same opinion as both its leaf-children at time step $\tau - h(T_v)$ that is inherited from $u$ at time step $\tau - h(T_v) - 1 \geq 0$. So, $\tilde\xi_{\tau - h(T_v) + 1}(u) = \tilde\xi_{\tau - h(T_v) - 1}(u)$ --- a contradiction.

 Let $p(t)$ be the infimum of $\mathbb{P}(\mathcal{Q}^* | \mathcal{Q}_1)$ over all possible $h(T_v)$. We proved that $\mathbb{P}(\mathcal{Q}^*\mid\mathcal{Q}_1)=1$ when $h(T_v)<t$. On the other hand, for all $v$ and $T$ such that $h(T_v)\geq t$, $\mathbb{P}(\mathcal{Q}^*)$ does not depend on $h(T_v)$ and equals $1-p_{\mathrm{w}}(0, h, u)$ for any $h>t$ and any vertex $u$ with height $t$. So, $\mathbb{P}(\mathcal{Q}^*)=1-p_{\mathrm{w}}(0, t+1, u)=:1-p_{\mathrm{w}}(0, t)$. Note that $p_{\mathrm{w}}(0, t)$ is non-increasing and $\lim_{n\to\infty}p_{\mathrm{w}}(0, t)=p_{\mathrm{w}}(0)$. Let $t(\varepsilon)$ be so large that $1-p_{\mathrm{w}}(0, t)>(1-\varepsilon)(1-p_{\mathrm{w}}(0))$ for all $t\geq t(\varepsilon)$. We also have that $\mathcal{Q}^* \subseteq \mathcal{Q}_1$. So, since $\mathbb{P}(\mathcal{Q}_1) \leq 1 - p_{\mathrm{w}}(0)$, we get that, for all $t\geq t(\varepsilon)$, 
 $$
 p(t)=\inf_{v,T}\mathbb{P}(\mathcal{Q}^* | \mathcal{Q}_1)=\inf_{v:\,h(T_v)\geq t}\frac{\mathbb{P}(\mathcal{Q}^*)}{\mathbb{P}(\mathcal{Q}_1)}\geq\frac{\mathbb{P}(\mathcal{Q}^*)}{1 - p_{\mathrm{w}}(0)}\geq 1-\varepsilon.$$

  Now,
\begin{align*}
 \mathbb{P}(\mathcal{Q}_2\mid\mathcal{Q}_1)\geq
 \mathbb{P}(\mathcal{Q}_2\wedge\mathcal{Q}^*\mid\mathcal{Q}_1) &=
 \sum_{s<t} \mathbb{P}(\mathcal{Q}_2\wedge\{\tau=s\}\mid\mathcal{Q}_1)
\\
&= \sum_{s<t}\mathbb{P}(\mathcal{Q}_2\mid\{\tau=s\}\wedge\mathcal{Q}_1)\mathbb{P}(\tau=s\mid\mathcal{Q}_1)\\
&\geq p(t)\cdot\min_{s<t}\mathbb{P}(\mathcal{Q}_2\mid\{\tau=s\}\wedge\mathcal{Q}_1).
\end{align*}

Let us assume that $\mathcal{Q}_1\wedge\{\tau=s\}$ holds for some $s<t$. 
 For every grandchild $w$ of $v$, we have that $\tilde\xi_{s-2}(v) = -1$. 
 In the usual way (and in exactly the same way as in the proof of Claim~\ref{cl:weak_calculations}), we get that every grandchild $w$ of $v$ has a (canonical) descendant $w'=w'(w)$ such that
 \begin{itemize}
 \item $d(w, w') = s-2$, and 
 \item every vertex from the path between $w$ and $w'$ at distance $j$ from $w'$ has opinion $-1$ at time step $j$. 
 \end{itemize}
 We note that the four vertices $w'(w)$ are uniquely defined by the initial opinions of vertices from the tree $T_v$ pruned at depth $s$.

 Since validity of the event $\mathcal{Q}_1\wedge\{\tau=s\}$ as well as the vertices $w'$ are defined by the pruned tree, the events that $w'$ are weakly 0-stable are independent of each other and of the event $\mathcal{Q}_1\wedge\{\tau=s\}$.
 Moreover, weak 0-stability of any $w'$ implies weak $(s-2)$-stability of the corresponding $w$, due to Claim~\ref{cl:weak_rising}. Also, if at least two non-sibling grandchildren of $v$ are $(s-2)$-stable, then $v$ is strongly $s$-stable, by Claim~\ref{cl:maintain_weak_stability} and Claim~\ref{cl:weak_stabilization}. So, $\xi_t(v) = -1$, as $s < t$. 

We conclude that, for every $s$,
$$
\mathbb{P}(\mathcal{Q}_2\mid\{\tau=s\}\wedge\mathcal{Q}_1)\geq 0.5(1 - (1 - p_{\mathrm{w}}(0))^2)^2.
$$
Summing up, we get that the probability of the desired event from the statement of Claim~\ref{cl:strong_calculations} is at least 
$$
\left(1 - \frac{1}{2}(1-\varepsilon)(1 - (1 - p_{\mathrm{w}}(0))^2)^2\right)(1 - p_{\mathrm{w}}(0)) < 0.038,
$$
due to the choice of a  small enough $\varepsilon$ and due to Claim~\ref{cl:binary_calculations}.

\end{proof}

\subsubsection{Proof of Lemma~\ref{lm:binary_upper}}
\label{sc:4:3:2}


Let $p_+ = 0.01$, $p_- = 0.0513$. Let us choose $c_+<2/3$ large enough so that
\begin{equation}
 2^{h+d/2}(p_+ + p_-)^{d/2} = o(1)
\label{eq:L4.8_UB}
\end{equation}
for the biggest even number $d < \lceil c_+ h\rceil$. 

Let us note that if, for some vertex $v$, $\xi_d(v) \neq \xi_{d-2}(v)$ then the opinion $\xi_d(v)$ reached the vertex $v$ by traveling across some path $v_0 v_1 \ldots v_d=:v$, i.e. $\xi_i(v_i) = \xi_d(v_d)$ for every $i \in \{0,\ldots,d\}$ and $\xi_{i-2}(v_{i}) \neq \xi_d(v_d)$ for every $i \in \{2,\ldots,d\}$. Let us recall that, for some constant $C_1 > 0$, the number of paths of length $d$ is less then $2^{h+\frac{d}{2}} \cdot C_1$ (see, e.g.,~\cite[Lemma 5.7]{JMS}). 
We plan to prove that for some constant $C_2 > 0$ and any path $v_0 v_1 \ldots v_d$ the above-mentioned traveling happens with probability less than $C_2(p_+ + p_-)^{d/2}$. This will conclude the proof of Lemma~\ref{lm:binary_upper} due to~\eqref{eq:L4.8_UB}.


Consider a path $P=v_0 v_1 \ldots v_d := v$. Let us denote by $\mathcal{E}(P)$ the following event: 
$$
\mathcal{E}(P)=\biggl\{\text{
$\xi_i(v_i) = 1$ for every $i \in \{0,1,\ldots,d\}$ and $\xi_{i-2}(v_{i}) = -1$ for every $i \in \{2,\ldots,d\}$}\biggr\}.
$$

For every $i\in[d-1]$, denote by $u_i$ the neighbour of $v_i$ in $T$ that does not belong to $V(P)$. Let $t_0$ be a large enough even integer constant so that the conclusion of Claim~\ref{cl:strong_calculations} holds for all $t\geq t_0$. For every $i\in[d-1]$, we colour the children of $u_i$ in two colours --- red and blue. For $2i\in[d-4]$, let $A_{2i}$ be the set of all vertices that are either descendants\footnote{Hereinafter, a vertex is a descendant of itself.} of the blue child of $u_{2i}$, or descendants of $u_{2i+1}$, or descendants of the red child of $u_{2i+2}$. 
Let $\mathcal{I} \subset \{t_0 + 6, t_0+8, \ldots, d-4\}$ be the maximum set of even integers that satisfy the following condition: for every $i\in\mathcal{I}$, 
\begin{itemize}
    \item either $v_i$ is a grandchild of $v_{i+2}$ or $v_{i+2}$ is a grandchild of $v_{i}$;
    \item $\min(h_T(v_i), h_T(v_{i+2})) \geq t_0 + 6$.
\end{itemize}
The fact that $i \geq t_0 + 6$ allows us to apply Claim~\ref{cl:strong_calculations} 
 at time steps at least $i-4$. 
  Also, let us note that, as the sets $A_i,A_{i'}$, for $i\neq i'$, are disjoint, the probabilities of events defined by $\xi_0|_{A_i}$ and $\xi_0|_{A_{i'}}$ are independent due to Remark~\ref{rm:independancy_general}. Lastly, let us note that $||\mathcal{I}|-d/2|=O(1)$. 

\begin{figure}[h]
    \centering
    \includegraphics[scale=0.2]{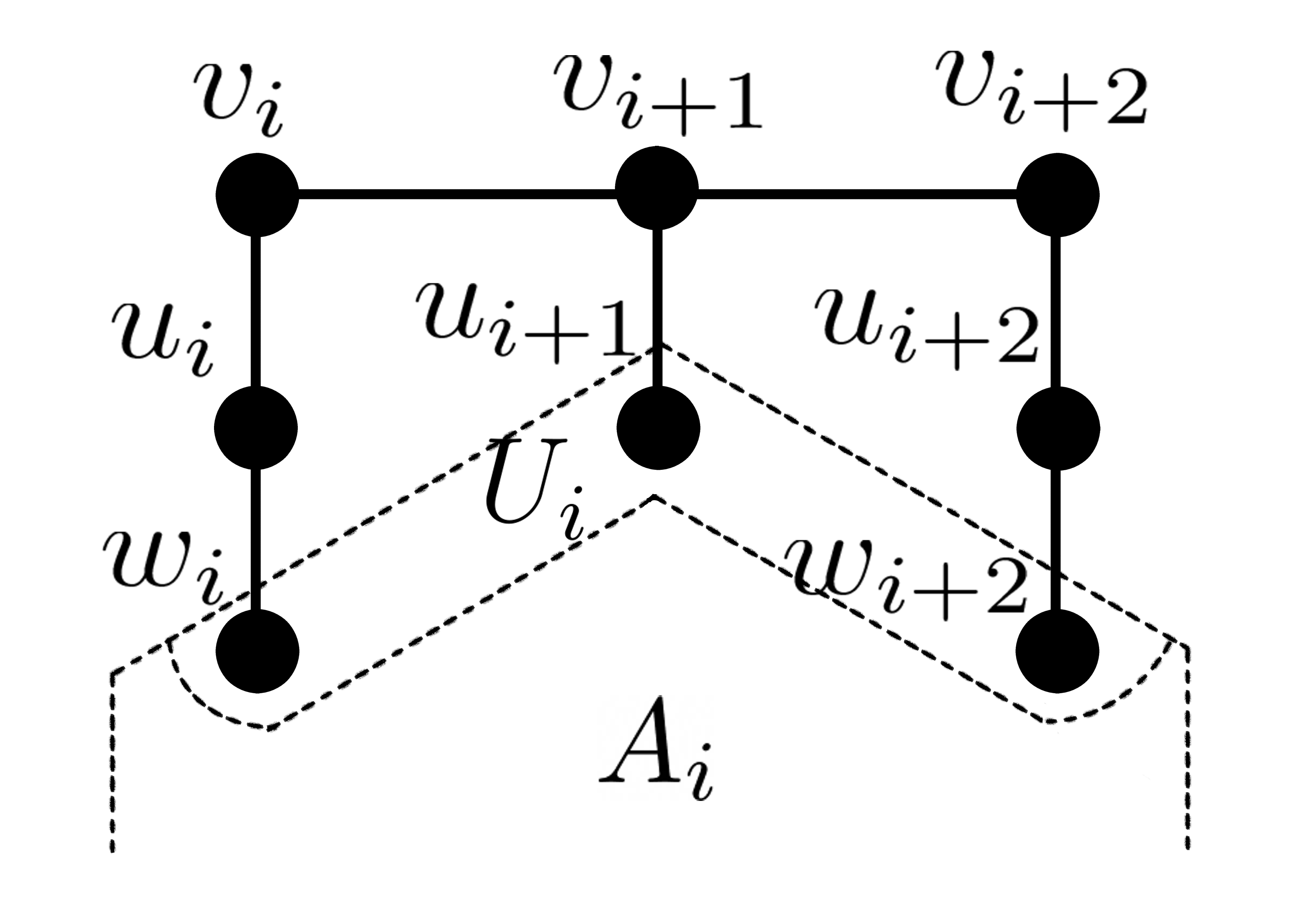}
    \caption{The structure of sets $A_i$ and $U_i$}
    \label{fig:A_i}
\end{figure}

We denote the number of vertices $u \in V(P)$ at even distance from $v_0$ such that $\xi_0(u)$ equals $1$ and $-1$ by $m_+$ and $m_-$, respectively. Fix an even $i \in [d-4]$ such that $A_i \in \mathcal{I}$. Denote the number of vertices $u \in \{v_i, v_{i+2}\}$ such that $\xi_0(u)$ equals $1$ and $-1$ by $k_+(i)$ and $k_-(i)$, respectively. Let the random variable $\Xi_i:=\Xi_i(\xi_0|_{V(P)})$ be the set of vectors of initial opinions of vertices from $A_i$ such that $\mathbb{P}(\mathcal{E}(P)\mid\xi_0|_{A_i}\in\Xi_i,\,\xi_0|_{V(P)})>0$. In the remaining part of the proof, we show that 
\begin{equation}
\mathbb{P}\left(\xi_0|_{A_i}\in\Xi_i\mid \xi_0|_{V(P)}\right)\leq 2\sqrt{p_+^{k_+(i)}p_-^{k_-(i)}}=:q(i).
\label{eq:prob_bouund_Xi_i}
\end{equation}
 This inequality completes the proof of Lemma~\ref{lm:binary_upper} since, due to Remark~\ref{rm:independancy_general},
\begin{align*}
 \mathbb{P}(\mathcal{E}(P)\mid \xi_0|_{V(P)}) &\leq
 \mathbb{P}\left(\bigcap_{i\in\mathcal{I}}\xi_0|_{A_i}\in\Xi_i\mid \xi_0|_{V(P)}\right)\\
 &\leq \prod_{i\in\mathcal{I}}q(i)=O(2^{d/2})\prod_{i\in\mathcal{I}}\sqrt{p_+^{k_+(i)}p_-^{k_-(i)}}=O\left(2^{d/2} p_+^{m_+}p_-^{m_-}\right).
\end{align*} 
The latter equality holds because each (but at most a bounded number) $i$ such that $\xi_0(v_i) = 1$ (or -1) contributes an additive 1 to both $k_+(i-2)$ and $k_+(i)$ (or $k_-(i-2)$ and $k_-(i)$) for $i-2,i\in\mathcal{I}$. So, since $||\mathcal{I}| - d/2| = O(1)$,
the product 
$$
\prod_{i\in\mathcal{I}}\sqrt{p_+^{k_+(i)}p_-^{k_-(i)}} = \sqrt{p_+^{\sum_{i\in\mathcal{I}}k_+(i)}p_-^{\sum_{i\in\mathcal{I}} k_-(i)}}
$$ 
differs from the $p_+^{m_+}p_-^{m_-}$ by a multiplicative constant.

Since all initial values on vertices along $P$ are equally likely, we conclude that 
\begin{align*}
\mathbb{P}(\mathcal{E}(P)) &=
\sum_{\mathbf{x}\in\{-1,1\}^{V(P)}}O\left(2^{d/2} p_+^{m_+(\mathbf{x})}p_-^{m_-(\mathbf{x})}\right)\mathbb{P}\left(\xi_0|_{V(P)}=\mathbf{x}\right)\\
&=O\left(\sum_{\mathbf{x}\in\{-1,1\}^{V(P)}}p_+^{m_+(\mathbf{x})}p_-^{m_-(\mathbf{x})}\right)\\
&=O\left(\sum_{i=0}^{|V(P)|}p_+^{i}p_-^{|V(P)|-i}\right)=O\left((p_+ + p_-)^{d/2}\right),
\end{align*}
as needed.
 


From now on, we assume that the vector $\xi_0|_{V(P)}$ is fixed and deterministic, we also fix an index $i \in \mathcal{I}$, and work only with the corresponding subset $A_i$. Our aim is to prove the inequality~\eqref{eq:prob_bouund_Xi_i}.  We denote the children of $u_i$ and $u_{i+2}$ that belong to $A_i$ by $w_i$ and $w_{i+2}$, respectively. Let 
$$
U_i:= \{w_i, u_{i+1}, w_{i+2}\}\subset A_i.
$$
 Let us define the event 
 $$
 \mathcal{E}_1(P) := \{\xi_{i - 2}(v_{i}) = \xi_{i}(v_{i+2}) = -1,\  \xi_{i}(v_{i}) = \xi_{i+2}(v_{i+2}) = \xi_{i-2}(w_i) = \xi_{i}(u_{i+1}) = \xi_{i}(w_{i+2}) = 1\}
 $$
and prove that it follows from $\mathcal{E}(P)$.
Indeed, the only non-trivial equalities that are not included in the definition of $\mathcal{E}(P)$ are $\xi_{i-2}(w_i) = \xi_{i}(u_{i+1}) = \xi_{i}(w_{i+2}) = 1$. Since $\xi_{i}(v_i) = 1 \neq \xi_{i-1}(v_{i+1})$, we get that $\xi_{i-1}(u_{i}) = 1$. It implies $\xi_{i-2}(w_{i}) = 1$, since $\xi_{i-2}(v_{i})=-1$. So, we get $\xi_{i-1}(u_i)=\xi_{i-2}(w_i)=1$. The remaining equalities $\xi_i(u_{i+1})=\xi_i(w_{i+2})=1$ can be proved in literally the same way by replacing, first, $i$ with $i+1$ and, then, $i$
 with $i+2$.

Thus, the event $\xi_0|_{A_i} \notin \Xi_i$ implies $\xi_0|_{A_i} \notin \Xi'_i$, where $\Xi'_i$ is the set of vectors of initial opinions of vertices from $A_i$ such that $\mathbb{P}(\mathcal{E}_1(P)\mid\xi_0|_{A_i}\in\Xi'_i)>0$ (recall that the initial opinions along $V(P)$ are fixed), and so it is sufficient to get an upper bound on the probability of the latter. In other words, we will bound from below the probability of the set of initial opinions of vertices from $A_i$ so that, whatever initial opinions of all other vertices, $\mathcal{E}_1(P)$ never holds. Let us now describe one of the main instruments that we will use to get such a lower bound. 
 Let 
$$
Q \in \mathcal{Q} := \{(w_{i}u_iv_iv_{i+1}u_{i+1}), \, (u_{i+1} v_{i+1} v_{i+2} u_{i+2} w_{i+2}), \, (w_{i} u_i v_i v_{i+1} v_{i+2} u_{i+2} w_{i+2})\}
$$ 
be a path. Let us show that 
\begin{align}
\text{if $Q$ satisfies either the requirements of Claim~\ref{cl:weak_stabilization} or the requirements of Claim~\ref{cl:one_far_stabilization}}\notag\\ 
\text{with some even $t\leq i-2$, then $\mathcal{E}_1(P)$ does not hold. }
\label{eq:proposition_E_1_contradiction_claim14}
\end{align}

First, let $Q = (w_{i} u_i v_i v_{i+1} v_{i+2} u_{i+2} w_{i+2})$. If $Q$ satisfies the requirements of Claim~\ref{cl:weak_stabilization} with some even $t\leq i-2$, then the conclusion of this claim --- $t$-stability of $w_i,v_i,v_{i+2},w_{i+2}$  --- contradicts to the inequality $\xi_i(v_{i+2}) \neq \xi_{i}(v_i)$, thus $\mathcal{E}_1(P)$ does not hold. If $Q$ satisfies the requirements of Claim~\ref{cl:one_far_stabilization} for some even $t\leq i-2$, then it asserts that either $\xi_i(v_{i+2}) = \xi_{i}(v_i)$ or $\xi_i(v_{i+2}) = \xi_i(w_{i+2})$. This is again not possible for $\xi_0\in\mathcal{E}_1(P)$. Now, let $Q = (u_{i+1} v_{i+1} v_{i+2} u_{i+2} w_{i+2})$. In the same way as above, the conclusion of
Claim~\ref{cl:weak_stabilization} for $t\leq i-2$ contradicts to the property $\xi_i(v_{i+2}) \neq \xi_{i}(u_{i+1})$ required by $\mathcal{E}_1(P)$; and the conclusion of  Claim~\ref{cl:one_far_stabilization} implies either $\xi_i(v_{i+2}) = \xi_{i}(u_{i+1})$ or $\xi_i(v_{i+2}) = \xi_i(w_{i+2})$, that contradicts to $\mathcal{E}_1(P)$ as well.
Finally, if the conclusion of Claim~\ref{cl:weak_stabilization} holds for  $Q = (w_{i} u_i v_i v_{i+1} u_{i+1})$, then 1) $\xi_{i-2}(v_i)$ equals at least one of $\xi_{i-2}(w_i)$, $\xi_{i-2}(u_{i+1})$, 2) $u_{i+1}$ is weakly $(i-2)$-stable. On the other hand, if $\mathcal{E}_1(P)$ holds as well, then $-1=\xi_{i-2}(v_i)\neq\xi_{i-2}(w_i)$. Therefore, $\xi_{i-2}(v_i)=\xi_{i-2}(u_{i+1})=-1$ and then $\xi_{i-1}(v_{i+1})=-1.$ By Claim~\ref{cl:maintain_weak_value}, since $u_{i+1}$ is weakly $(i-2)$-stable, we get $\xi_{i}(u_{i+1}) = -1$ --- contradiction with $\mathcal{E}_1(P)$.

\begin{remark}
All the events that we use below to bound $\mathbb{P}(\xi_0|_{A_i} \in \Xi'_i)$ depend on initial opinions of $A_i$ only. Therefore, it is convenient to have in mind the following distribution of initial opinions: 1) $\xi_0|_{V(P)}$ is fixed and deterministic; 2) $\xi_0|_{A_i}$ is uniformly random; 3) on all the other vertices, $\xi_0|_{V(T)\setminus(V(P)\cup A_i)}$ is a canonical vector, certifying $\mathcal{E}_1(P)$, if such a vector exists. In what follows, we denote this probability measure by $\mathbb{P}^*$. We have that $\mathbb{P}(\xi_0|_{A_i} \in \Xi'_i)=\mathbb{P}^*(\mathcal{E}_1(P))$. Therefore, if we manage to find an event $\mathcal{B}$ that depends only on initial opinions of $A_i$ such that $\mathcal{E}_1(P)$ implies $\mathcal{B}$, then $\mathbb{P}(\xi_0|_{A_i} \in \Xi'_i)=\mathbb{P}^*(\mathcal{E}_1(P))\leq\mathbb{P}^*(\mathcal{B})=\mathbb{P}(\mathcal{B})$. Thus, we may use an upper bound for $\mathbb{P}(\mathcal{B})$ to bound the desired  $\mathbb{P}(\xi_0|_{A_i} \in \Xi'_i)$.
\end{remark}



\paragraph{Proof of~\eqref{eq:prob_bouund_Xi_i}.} First of all, let 
$$
q_1=\mathbb{P}^*(\mathcal{E}_1(P)\wedge\text{not all vertices from $U_i$ are 1-close to stability}).
$$
Note that $\mathcal{E}_1(P)$ implies that none of the vertices from $U_i$ is strongly 0-stable with initial opinion $-1$. Therefore, due to Claim~\ref{cl:binary_calculations} and Claim~\ref{cl:weak_calculations},
\begin{align*}
q_1 &\leq \prod_{u \in U_i} \mathbb{P}^*(u \text{ is not strongly }0\text{-stable with }\xi_0(i)=-1)\\ 
&\quad\quad - \prod_{u \in U_i} \mathbb{P}^*(u \text{ is not strongly }0\text{-stable with }\xi_0(i)=-1, \text{ but is }1\text{-close to stability}) \\
&= \prod_{u \in U_i} \mathbb{P}(u \text{ is not strongly }0\text{-stable with }\xi_0(i)=-1)\\ 
&\quad\quad - \prod_{u \in U_i} \mathbb{P}(u \text{ is not strongly }0\text{-stable with }\xi_0(i)=-1, \text{ but is }1\text{-close to stability}) \\
&\leq \left(1-\frac{p_{\mathrm{s}}(0)}{2}\right)^3 - \left(1-\frac{p_{\mathrm{s}}(0)}{2} - (1 - 0.9957)\right)^3 < 0.0073.    
\end{align*}
Let 
$$
q_2=\mathbb{P}^*(\mathcal{E}
_1(P)\wedge\{\text{at least two vertices from $U_i$ are not weakly 0-stable}\}).
$$
Since $i\geq t_0+6$, we can apply Claim~\ref{cl:strong_calculations} to any vertex from $U_i$: the probability that $\mathcal{E}_1(P)$ holds and a fixed vertex $u\in U_i$ is weakly 0-stable is at most 0.038. Thus, due to Claim~\ref{cl:strong_calculations},
$$q_2 \leq 3 \cdot 0.038^2 < 0.0044.$$

We conclude that
$$
 \mathbb{P}^*(\mathcal{E}_1(P))=q_1+q_2+q_3<q_3+0.0117,
$$
where $q_3=\mathbb{P}^*(\mathcal{E}_2(P))$ and
\begin{align}
\mathcal{E}_2(P)=\mathcal{E}_1(P)&\wedge\{\text{every vertex from $U_i$ is 1-close to stability}\}\notag\\
&\wedge\{\text{the number of not weakly 0-stable vertices in $U_i$ does not exceed $1$\}}.
\label{eq:q3-def}
\end{align}
We bound $q_3$ separately for several different cases depending on the value of $\xi_0(v_i) + \xi_0(v_{i+2})$. For the case $\xi_0(v_i) = \xi_0(v_{i+2})$, we will need an auxiliary proposition.

\begin{claim}
Let $\xi_0(v_i) = \xi_0(v_{i+2}) =: \xi$. 
If
\begin{equation}
\text{every vertex from $U_i$ is weakly 0-stable and 1-close to stability},
\label{eq:assumption_stab}
\end{equation}
then 
\begin{equation}
\text{for every $u\in U_i$ and every even time step $t \leq i-2$, $\xi_t(u) \neq \xi$.}
\label{eq:corollary_stab}
\end{equation}
\label{cl:12_from_11}
\end{claim}
 

\begin{proof}

 First, suppose that $\xi_0(u) = \xi$ holds for some $u \in U$. Let $Q \in \mathcal{Q}$ be a path such that one of its ends coincides with $u$. Depending on the opinion of the vertex at the other end of $Q$, we get the requirements of either Claim~\ref{cl:weak_stabilization} or Claim~\ref{cl:one_far_stabilization}. As we already know, this contradicts $\mathcal{E}_1(P)$.


Next, suppose that $\xi_t(u)=\xi$ for some even $2\leq t\leq i-2$. The existence of such $t$ and the fact that all $u\in U_i$ have $\xi_0(u)\neq\xi$ implies that some $u\in U_i$ changes its opinion at least once before the time $i-2$. Let $t_1 \leq i-2$ be the first even time step such that some $v\in U_i\sqcup\{v_i, v_{i+2}\}$ has opinion different from its initial opinion $\xi_0(v)$. Suppose that none of the vertices from $U_i$ changed their opinion at time step $t_1$. Then, both $v$ and the two vertices from $U_i$ at distance 2 from $v$ share the opinion opposite to $\xi$ at time step $t_1$; those three vertices belong to the single path $Q \in \mathcal{Q}$ and the requirements from Claim~\ref{cl:weak_stabilization} are satisfied by $Q$ at time step $t_1$ --- contradiction again. So, now we may assume that $v \in U_i$. 

Let us separately consider all three possible values of $v$. 

{\bf 1.} First, assume $v=u_{i+1}$, or, in other words, $\xi_{t_1}(u_{i+1}) = \xi$. Let $t_2 > t_1$ be the first even time step that $\xi_{t_2}(u_{i+1}) \neq \xi$ ($t_2$ might be $+\infty$). Assume that, for some even time step $t \in \{t_1, t_1 + 2, \ldots, t_2\}$,  $\xi_t(w_i) = \xi$ for the first time. Our aim is to apply Claim~\ref{cl:weak_stabilization} at time step $t$ to the path between vertices $w_i$ and $u_{i+1}$. For this, we have to show that 
\begin{enumerate}
\item  $\xi_t(v_i)=\xi_t(u_{i+1})=\xi$ and
\item $w_i,u_{i+1}$ are weakly $t$-stable.
\end{enumerate}
We first note that $w_i$ is weakly $(t-2)$-stable, due to assumption~\eqref{eq:assumption_stab} and  Claim~\ref{cl:maintain_weak_stability}. So, $\xi_{t-1}(u_i) = \xi$ by Claim~\ref{cl:maintain_weak_value} and hence $ \xi_{t-2}(v_i) = \xi$. We get $\xi_{t-1}(v_{i+1}) = \xi$ --- for $t>t_1$ it follows from $\xi_{t-2}(u_{i+1})=\xi$ and for $t=t_1$ it follows from $\xi_{t-2}(v_{i+2}) = \xi$ (recall that $t_1$ is the first time moment when a vertex from $U_i\sqcup\{v_i,v_{i+2}\}$ changed its opinion). Therefore, $\xi_{t}(v_i) = \xi$, since $\xi_{t-1}(u_i) = \xi_{t-1}(v_{i+1}) = \xi$, as needed. Let us remind, that by assumption~\eqref{eq:assumption_stab}, $u_{i+1}$ is 1-close to stability, so it is weakly $t_1$-stable. Then, by Claim~\ref{cl:maintain_weak_stability}, it is weakly $(t-2)$-stable. So, $\xi_t(u_{i+1}) = \xi$ by Claim~\ref{cl:maintain_weak_value}, since $\xi_{t-1}(v_{i+1}) = \xi$, completing verification of the first condition that we had to check in order to apply Claim~\ref{cl:weak_stabilization}. We can also conclude that $t<t_2$. Moreover, $w_i$ is weakly $t$-stable, since it is 1-close to stability by assumption~\eqref{eq:assumption_stab}, and $u_{i+1}$ is weakly $t$-stable by Claim~\ref{cl:maintain_weak_stability}. Hence, indeed, we can apply Claim~\ref{cl:weak_stabilization} with time step $t$ to the path $Q$ from $w_i$ to $u_{i+1}$. Then, if $t \leq i-2$, we get
a contradiction with $\mathcal{E}_1(P)$ due to~\eqref{eq:proposition_E_1_contradiction_claim14}.
  If $i \geq t_1$, then $\xi_{i-2}(w_{i}) \neq \xi = \xi_{i}(u_{i+1})$, contradicting $\mathcal{E}_1(P)$. So $i < t_1$, completing the proof of~\eqref{eq:corollary_stab}.

{\bf 2.} Let $v = w_{i}$. By assumption~\eqref{eq:assumption_stab} and Claim~\ref{cl:maintain_weak_stability}, $w_i$ is weakly $(t_1-2)$-stable. Since $\xi_{t_1}(w_i) = \xi \neq \xi_{t_1-2}(w_i)$, we derive $\xi_{t_1-1}(u_{i}) = \xi$ due to Claim~\ref{cl:maintain_weak_value}. Let us recall that $t_1$ is the first moment when some vertex from $U_i\sqcup\{v_i,v_{i+2}\}$ changes its opinion. Therefore, $\xi_{t_1-2}(v_i)=\xi_{t_1-2}(v_{i+2})=\xi$ and hence $\xi_{t_1-1}(v_{i+1})=\xi$. Since $v_{i+1}$ and $u_i$ are neighbours of $v_i$, we get $\xi_{t_1}(v_i) = \xi$. However, $\mathcal{E}_1(P)$ implies that, at time moment $i-2$, $\xi_{i-2}(w_i) \neq \xi_{i-2}(v_i)$, in contrast to time moment $t_1$. Recalling that $t_1\leq i-2$, we get that at least one of $v_i, w_i$ should change its opinion at time step not exceeding $i-2$. Note that $w_i$ is weakly $t_1$-stable since it is 1-close to stability by~\eqref{eq:assumption_stab}. Let $t_2$ be the first even time moment after $t_1$ when $w_i$ changes its opinion again (it can be infinite). Assuming that $t_2<\infty$, we get that $w_i$ is weakly $(t_2-2)$-stable due to Claim~\ref{cl:maintain_weak_stability}. So, by Claim~\ref{cl:maintain_weak_value}, $\xi_{t_2-1}(u_i) \neq \xi$. But then $\xi_{t_2-2}(v_i)\neq\xi$ as well since $\xi_{t_2-2}(w_i)=\xi$ and $v_i,w_i$ are neighbours of $u_i$. Thus, $v_i$ changes its opinion at some even step $t_3\in[t_1+2,t_2-2]$. Due to the above conclusion, we also get that $t_3\leq i-2$. The neighbour $u_i$ of $v_i$ at any odd time step between $t_1+1$ and $t_3-1$ has opinion $\xi$ since vertices $v_2,w_i$ agree on $[t_1,t_3-2]$. Thus, $\xi_{t_3-1}(v_{i+1})\neq\xi$, implying $\xi_{t_3-2}(v_{i+2})\neq\xi$ (recall that $\xi_{t_3-2}(v_i)=\xi$). Since $t_1$ is the first  moment when some vertex from $U_i\sqcup\{v_i,v_{i+2}\}$ changes its opinion, we get that $v_{i+2}$ changes opinion for the first time at some even time step $t_4 \in \{t_1, t_1 + 2, \ldots, i-4\}$.

To summarise, we get the following dynamics of opinions of vertices $w_i,v_i,v_{i+2}$ at even time moments from $[0,t_4]$: the vertex $w_i$ changes its opinion once at time $t_1\leq t_4$ to $\xi$, the vertex $v_i$ holds the opinion $\xi$ throughout the entire time interval, and the vertex $v_{i+2}$ changes its opinion once to $-\xi$ at time $t_4$. Let us now observe the behaviour of opinions of vertices $u_{i+1}$ and $w_{i+2}$ within the same time interval. It will allow us to arrive to a contradiction with $\mathcal{E}_1(P)$ in a similar way to the case {\bf 1} (when $v=u_{i+1}$) due to Claim~\ref{cl:weak_stabilization}. We consider three scenarios.

\begin{itemize}

\item  Assume that $u_{i+1}$ changes its opinion for the first time at some even time step $t \leq t_4$. Here, we may apply Claim~\ref{cl:weak_stabilization} at time step $t$ to the path $Q$ between $w_i$ and $u_{i+1}$: First, since $t \leq t_4$, as we have already proved above, $w_i$ and $v_i$ preserve opinion $\xi$ up to time step $t$. This means that $\xi_{t}(w_i) = \xi_{t}(v_i) = \xi_{t}(u_{i+1})$ and that $w_i$ is weakly $t$-stable by Claim~\ref{cl:maintain_weak_stability} (since it is 1-close to stability by~\eqref{eq:assumption_stab}). Finally, $u_{i+1}$ is weakly $t$-stable as well since it is 1-close to stability by~\eqref{eq:assumption_stab}. Applying Claim~\ref{cl:weak_stabilization} we get a contradiction with $\mathcal{E}_1(P)$ due to~\eqref{eq:proposition_E_1_contradiction_claim14}.

\item Assume that both $u_{i+1},w_{i+2}$ do not change their opinion in any even time step from $[0,t_4]$. As above, we apply Claim~\ref{cl:weak_stabilization} to the path $Q$ from $u_{i+1}$ to $w_{i+2}$ and get a contradiction with $\mathcal{E}_1(P)$ due to~\eqref{eq:proposition_E_1_contradiction_claim14}. Indeed,
 $\xi_{t_4}(w_{i+2}) = \xi_{t_4}(v_{i+2}) = \xi_{t_4}(u_{i+1}) \neq \xi$ 
  and both $u_{i+1}$ and $w_{i+2}$ are weakly $t_4$-stable due to Claim~\ref{cl:maintain_weak_stability} and weak 0-stability that is assumed in~\eqref{eq:assumption_stab}.

\item Assume that $u_{i+1}$ does not change its opinion in any even time step from $[0,t_4]$ but $w_{i+2}$ changes its opinion at some time step $t_1\leq t<t_4$. Here we have that $\xi_{t}(w_i) = \xi_{t}(v_i) = \xi_{t}(v_{i+2}) = \xi_{t}(w_{i+2}) = \xi$ and both $w_i$ and $w_{i+2}$ are  weakly $t$-stable due to Claim~\ref{cl:maintain_weak_stability} and the property of being 1-close to stability assumed in~\eqref{eq:assumption_stab}. We can apply Claim~\ref{cl:weak_stabilization} at time step $t$ to the path $Q$ between $w_{i}$ to $w_{i+2}$  and get a contradiction with $\mathcal{E}_1(P)$ due to~\eqref{eq:proposition_E_1_contradiction_claim14}.


\item Finally, assume that $u_{i+1}$ does not change its opinion in any even time step from $[0,t_4]$ and $t_4$ is the first even time step from $[0,t_4]$ when $w_{i+2}$ changes its opinion. By Claim~\ref{cl:maintain_weak_stability} and due to~\eqref{eq:assumption_stab}, $w_{i+2}$ is weakly $(t_4-2)$-stable. So, by Claim~\ref{cl:maintain_weak_value}, $\xi_{t_4-1}(u_{i+2}) = \xi$. Since $\xi_{t_4-1}(v_{i+1}) = \xi$, we get a contradiction with $\xi_{t_4}(v_{i+2}) \neq \xi$.

\end{itemize}

{\bf 3.} In case $v = w_{i+2}$ we can repeat the reasoning from the previous case, with the following amendments. First, everywhere in the proof we make the following replacements: $w_i\leftrightarrow w_{i+2}$, $u_i\leftrightarrow u_{i+2}$, $v_i\leftrightarrow v_{i+2}$.
  As in the previous case, we first get $\xi_{t_1}(w_{i+2}) = \xi_{t_1}(v_{i+2})$. Nevertheless, $\mathcal{E}_1(P)$ implies that $\xi_{i}(w_{i+2}) \neq \xi_{i}(v_{i+2})$. Recalling that $t_1 \leq i-2$, at least one of $\{v_{i+2}, w_{i+2}\}$ should change its opinion at time step not exceeding $i$. In the remaining part of the proof it remains to make the following replacements: $i-2 \to i$ and $i-4\to i-2$. Since now $t_4 \leq i-2$, in the same way as in the case ``$v=w_i$'', we can apply~\eqref{eq:proposition_E_1_contradiction_claim14} to get a contradiction with $\mathcal{E}_1(P)$. This completes the proof of the claim.
  
  \end{proof} 


In what follows, assuming that all vertices of $U_i$ are 1-close to stability, we will get a contradiction with $\mathcal{E}_1(P)$ in some of the considered cases due to the observation~\eqref{eq:proposition_E_1_contradiction_claim14}. For the sake of convenience, let us list the conditions a path $Q \in \mathcal{Q}$ has to satisfy in order to meet the requirements of Claim~\ref{cl:weak_stabilization} and Claim~\ref{cl:one_far_stabilization} applied to this path. To apply Claim~\ref{cl:weak_stabilization} at time step $t$ we need to verify that
\begin{itemize}
    \item at time $t$, every second vertex (starting from any end) of $Q$ has the same opinion;
    \item both ends of $Q$ are weakly $t$-stable.
\end{itemize}
To apply Claim~\ref{cl:one_far_stabilization} at time step $t$ we need to verify that $Q$ can be partitioned into two subpaths so that
\begin{itemize}
    \item both ends of $Q$ are weakly $t$-stable and have different opinions at time $t$;
    \item at time $t$, every second vertex $u$ of $Q$ (starting from any end) has the same opinion as the end of $Q$ that belongs to the same subpath as $u$.
\end{itemize}

\paragraph{Consider the case $\xi_0(v_i) = \xi_0(v_{i+2}) = 1$.} 
 First, notice that if $\mathcal{E}_2(P)$ holds and all three vertices from $U_i$ are weakly 0-stable then~\eqref{eq:assumption_stab} holds. Then, due to Claim~\ref{cl:12_from_11}, $\xi_{i-2}(w_{i})=-1$, which contradicts  $\mathcal{E}_1(P)$.

 Now, let $\mathcal{E}_2(P)$ hold and there be a unique vertex $u \in U_i$ which is not weakly 0-stable. Then, let us prove that, in order to satisfy $\mathcal{E}_1(P)$, the following properties must hold:
\begin{itemize}
    \item[A1] the two vertices from $U_i\setminus\{u\}$ have initial opinion $-1$;
    \item[A2] the two vertices from $U_i\setminus\{u\}$ are not strongly 0-stable;
    \item[A3] for every vertex from $U_i\setminus\{u\}$, at least one of its grandchildren has initial opinion $-1$;
    \item[A4]  $\xi_t(u) = 1$ for some deterministic (that does not depend on $\xi_0\in\mathcal{E}_1(P)$) time step $t := t(u,i) \in \{i-2,i\}$.
\end{itemize}

A1 follows from~\eqref{eq:proposition_E_1_contradiction_claim14} and the fact that 
Claim~\ref{cl:weak_stabilization} and Claim~\ref{cl:one_far_stabilization} are applicable to the path between the two vertices from $U_i\setminus\{u\}$ and time moment $t=0$. Indeed, if both initial opinions are~$1$, then the considered path clearly satisfies the requirements of Claim~\ref{cl:weak_stabilization}. If the initial opinions are different, then the only non-trivial case that may interfere the requirements of Claim~\ref{cl:one_far_stabilization} is when $u=u_{i+1}$ and the path between the remaining two vertices from $U_i$ has seven vertices. However, since we assumed that $\xi_0(v_i)=\xi_0(v_{i+2})$, the initial opinions of $w_i,v_i,v_{i+2},w_{i+2}$ along the path change only once, as required. The event $\mathcal{E}_1(P)$ along with A1 immediately imply A2. Assuming that some $u'\in U_i\setminus\{u\}$ has all four grandchildren with initial opinion 1 and recalling that $\xi_0(u')=-1$ due to A1, we get that $u'$ is not weakly 0-stable. This proves A3. A4 is an instant corollary of $\mathcal{E}_1(P)$ --- actually, $t$ is either $i-2$ or $i$ depending on the choice of $u\in U_i$. Due to Claim~\ref{cl:strong_calculations}, the probability that $u \text{ is not weakly }0\text{-stable and A4 holds}$ can be bounded from above by the probability of an event that only depends on $\xi_0|_{V(T_u)}$ which is, in turn, at most $0.038$. Hence, by Remark~\ref{rm:independancy_general} the latter event is independent from the events A1, A2, and A3. Recalling that, for any $u' \in U_i\setminus\{u\}$, the event that $u'$ is strongly 0-stable implies A3,  we get
\begin{align}
q_3 &\leq 3 \cdot \mathbb{P}^*(\{u \text{ is not weakly }0\text{-stable}\}\wedge \text{A1} \wedge \text{A2} \wedge  \text{A3} \wedge \text{A4}) \notag \\
 &= 3 \cdot \mathbb{P}(\{u \text{ is not weakly }0\text{-stable}\}\wedge \text{A1} \wedge \text{A2} \wedge  \text{A3} \wedge \text{A4}) \notag \\
&\leq 3 \cdot 0.038 \cdot \mathbb{P}(\text{A1} \wedge \text{A2} \wedge  \text{A3}) \notag \\
&\leq 3 \cdot 0.038 \cdot \left(\frac{1}{2}\left(1 - \frac{1}{16} - p_{\mathrm{s}}(0)\right)\right)^2 \stackrel{\text{Claim}~\ref{cl:binary_calculations}}< 0.0055.
\label{eq:q-first_case}
\end{align}

\paragraph{Consider the case $\xi_0(v_i) = \xi_0(v_{i+2}) = -1$.} We need to bound $q_3=q_3^1+q_3^2$, where
$$
 q_3^1:=\mathbb{P}^*(\mathcal{E}_2(P)\wedge\{\text{all vertices of $U_i$ are weakly 0-stable}\}),
$$
$$
 q_3^2:=\mathbb{P}^*(\mathcal{E}_2(P)\wedge\{\text{exactly one vertex of $U_i$ is not weakly 0-stable}\}).
$$
In what follows, we prove bounds on $q_3^2$ and $q_3^3$ separately.


We start from $q_3^2$. Claim~\ref{cl:12_from_11} gives us that, for every $u \in U_i$ and every even $t \leq i-2$, $\xi_t(u) = 1$. In particular, $\xi_0(u) = \xi_2(u)$. The strategy of the remaining proof is as follows: We assume that the event $\xi_0(w_i)=\xi_0(u_{i+1})=\xi_0(w_{i+2})=1$ holds deterministically. Then, for every $u \in U_i$, we introduce an event that implies $\xi_0(u) \neq \xi_2(u)$ and such that all these three events (corresponding to $w_i,u_{i+1},w_{i+2}$) are independent of each other and of initial opinions of all vertices from $U_i\cup\{v_i,v_{i+2}\}$. 
For $u\in\{w_i,w_{i+2}\}$, this is the event that all four grandchildren of $u$ have initial opinion -1. Probability of this event equals $1/16$. For $u=u_{i+1}$, the desired event is that $u$ has a child with both children having initial opinion $-1$ (this event has probability $7/16$). Indeed, since $\xi_0(v_i) = \xi_0(v_{i+2}) = -1$, we have that $\xi_1(v_{i+1}) = -1$, and thus, for $u=u_{i+1}$, in order to change its opinion to $-1$ at time 2, it is sufficient to have one child that has opinion $-1$ at time~1. Due to Remark~\ref{rm:independancy_general}, the introduced ``lower bound events'' are independent, and thus
\begin{align}  
q_3^1 &\leq \mathbb{P}^*\left(\bigwedge_{u\in U_i}\{\xi_0(u)=1\}\wedge \bigwedge_{u\in U_i}\{\xi_0(u) \neq \xi_2(u)\}\right)\notag\\
&\leq\mathbb{P}\left(\bigwedge_{u\in U_i}\{\xi_0(u)=1\}\right)\left(1-\frac{1}{16}\right)^2\left(1-\frac{7}{16}\right)\notag\\
&=\frac{1}{8} \cdot \left(\frac{15}{16}\right)^2 \cdot \frac{9}{16} < 0.0618.
\label{eq:q-second_case_1}
\end{align}

Next, let us give an upper bound on $q_3^3$. Let $u \in U_i$ be the unique vertex that is not weakly 0-stable.
Let us prove that, in order to satisfy $\mathcal{E}_1(P)$, the following properties must hold:
\begin{itemize}
    \item[B1] the two vertices from $U_i\setminus\{u\}$ have initial opinion $1$;
    \item[B2] $\xi_t(u) = 1$ for some deterministic (that does not depend on $\xi_0\in\mathcal{E}_1(P)$) time step $t := t(u,i) \in \{i-2,i\}$.
\end{itemize}

B1 follows from~\eqref{eq:proposition_E_1_contradiction_claim14} and the fact that 
Claim~\ref{cl:weak_stabilization} and Claim~\ref{cl:one_far_stabilization} are applicable to the path between the two vertices from $U_i\setminus\{u\}$ and time moment $t=0$. Indeed, if both initial opinions are~$-1$, then the considered path clearly satisfies the requirements of Claim~\ref{cl:weak_stabilization}. If the initial opinions are different, then the only non-trivial case that may interfere the requirements of Claim~\ref{cl:one_far_stabilization} is when $u=u_{i+1}$ and the path between the remaining two vertices from $U_i$ has seven vertices. However, since we assumed that $\xi_0(v_i)=\xi_0(v_{i+2})$, the initial opinions of $w_i,v_i,v_{i+2},w_{i+2}$ along the path change only once, as required.
B2 is an instant corollary of $\mathcal{E}_1(P)$ --- actually, $t$ is either $i-2$ or $i$ depending on the choice of $u\in U_i$. Due to Claim~\ref{cl:strong_calculations}, the probability that $u \text{ is not weakly }0\text{-stable and B2 holds}$ can be bounded from above by the probability of an event that only depends on $\xi_0|_{V(T_u)}$ which is, in turn, at most $0.038$. Hence, by Remark~\ref{rm:independancy_general} the latter event is independent from the event B1. So,
\begin{equation}
q_3^2 \leq 3 \cdot \mathbb{P}(\{u \text{ is not weakly }0\text{-stable}\}\wedge \text{B1} \wedge \text{B2}) < 3 \cdot 0.038 \cdot \frac 1 4 = 0.0285.
\label{eq:q-second_case_2}
\end{equation}
Finally, due to~\eqref{eq:q-second_case_1}~and~\label{eq:q-second_case_2},
\begin{equation}
 q_3=q_3^1+q_3^2<0.0618+0.0285=0.0903.
\label{eq:q-second_case}
\end{equation}

\paragraph{Consider the last case $\xi_0(v_i) \neq \xi_0(v_{i+2})$.} Here, we actually have two ways of giving initial opinions to $v_i$ and $v_{i+2}$ but we plan to prove the same upper on $q_3$ in those cases. 

Let us prove that if every $u\in U_i$ is weakly 0-stable, then $\mathcal{E}_1(P)$ does not hold. Indeed, since $\xi_0(v_i) \neq \xi_0(v_{i+2})$, there is an index $j \in \{i, i+2\}$ such that $\xi_0(u_{i+1}) = \xi_0(v_j)$. Then we can apply either Claim~\ref{cl:weak_stabilization} or Claim~\ref{cl:one_far_stabilization} to the path between $u_{i+1}$ and $w_j$ at time step 0, getting a contradiction with~\eqref{eq:proposition_E_1_contradiction_claim14}. Indeed, both $u_{i+1}$ and $w_j$ are weakly 0-stable. So, since $\xi_0(u_{i+1}) = \xi_0(v_j)$, if $\xi_0(u_{i+1}) = \xi_0(w_j)$ then Claim~\ref{cl:weak_stabilization} applies and if $\xi_0(u_{i+1}) \neq \xi_0(w_j)$ then Claim~\ref{cl:one_far_stabilization} applies.

Therefore, if $\mathcal{E}_2(P)$ holds, then there is a single vertex from $U_i$ which is not weakly 0-stable; denote this vertex by $u$. 
 Due to $\mathcal{E}_1(P)$, there exists a deterministic (that does not depend on $\xi_0\in\mathcal{E}_1(P)$) time step $t := t(u,i) \in \{i-2,i\}$ such that $\xi_t(u) = 1$. Due to Claim~\ref{cl:strong_calculations}, the probability of the latter event jointly with the event that $u$ is not weakly 0-stable can be bounded from above by the probability of an event that only depends on $\xi_0|_{V(T_u)}$ which is, in turn, at most $0.038$. We further consider separately three cases, depending on the particular choice of $u$.


\paragraph{1) $u=w_{i+2}$.} 
We have $\xi_0(w_{i}) = \xi_0(u_{i+1}) \neq \xi_0(v_i)$ due to~\eqref{eq:proposition_E_1_contradiction_claim14} by applying Claim~\ref{cl:weak_stabilization} or Claim~\ref{cl:one_far_stabilization} at time step 0 to the path between $w_i$ and $u_{i+1}$. Indeed, both $u_{i+1}$ and $w_i$ are weakly 0-stable. So, if $\xi_0(u_{i+1}) = \xi_0(w_i) = \xi_0(v_i)$ then Claim~\ref{cl:weak_stabilization} applies and if $\xi_0(u_{i+1}) \neq \xi_0(w)$ then Claim~\ref{cl:one_far_stabilization} applies. So, $q_3\leq 0.038\cdot\frac{1}{4}$. Moreover, if $\xi_0(v_i) = 1$, we can improve this bound: since $\xi_0(w_i)=\xi_0(u_{i+1})=-1$, the event $\mathcal{E}_1(P)$ implies that both $w_i$ and $u_{i+1}$ are not strongly 0-stable. So, in this case, $\mathbb{P}^*(\mathcal{E}_2(P)\wedge\{u=w_{i+2}\})\leq 0.038\cdot\frac{1}{4}(1-p_{\mathrm{s}}(0))^2$. We get
\begin{equation}
\mathbb{P}^*(\mathcal{E}_2(P)\wedge\{u=w_{i+2}\})\leq 0.038\cdot\frac{1}{4}\left(\1_{\xi_0(v_i) =- 1}+(1-p_{\mathrm{s}}(0))^2\cdot \1_{\xi_0(v_i) = 1}\right).
\label{eq:q_3_1}
\end{equation}

\paragraph{2) $u=w_{i}$.} Since $w_{i+2}$ and $u_{i+1}$ are weakly 0-stable, in the same way as in the case 1), we get $\xi_0(w_{i+2}) = \xi_0(u_{i+1}) \neq \xi_0(v_{i+2})$. So, here we get that $\mathbb{P}^*(\mathcal{E}_2(P)\wedge\{u=w_{i+2}\})\leq\frac{1}{4}$ as well. In a similar way, by distinguishing between $\xi_0(v_{i+2}) = -1$ and $\xi_0(v_{i+2}) = 1$, we get
\begin{equation}
\mathbb{P}^*(\mathcal{E}_2(P)\wedge\{u=w_{i}\})\leq 0.038\cdot\frac{1}{4}\left(\1_{\xi_0(v_{i+2}) =- 1}+(1-p_{\mathrm{s}}(0))^2\cdot \1_{\xi_0(v_{i+2}) = 1}\right).
\label{eq:q_3_2}
\end{equation}


\paragraph{3) $u=u_{i+1}$.} Then $w_i$ and $w_{i+2}$ are weakly 0-stable. If $\xi_0(w_{i}) = \xi_0(v_{i})$ and $\xi_0(w_{i+2}) = \xi_0(v_{i+2})$, then $\mathcal{E}_1(P)$ does not hold due to~\eqref{eq:proposition_E_1_contradiction_claim14} and Claim~\ref{cl:one_far_stabilization} applied to the path between $w_{i}$ and $w_{i+2}$ at time moment 0. Indeed, the initial opinion changes just once in the sequence $w_i, v_{i}, v_{i+2}, w_{i+2}$ and $w_i$ and $w_{i+2}$ are weakly 0-stable. Therefore, $\xi_0|_{\{w_i,w_{i+2}\}}$ has three possible initial values, excluding $(\xi_0(v_i),\xi_0(v_{i+2}))$; the latter vector has a single $-1$. Note that $\mathcal{E}_1(P)$ implies the following: if $w \in \{w_i, w_{i+2}\}$ has initial opinion $-1$ then it is not strongly 0-stable. Summing up, we get 
\begin{equation}
\mathbb{P}^*(\mathcal{E}_2(P)\wedge\{u=u_{i+1}\})\leq 0.038\cdot\frac{1}{4}\left(1 + (1-p_{\mathrm{s}}(0)) + (1-p_{\mathrm{s}}(0))^2\right).
\label{eq:q_3_3}
\end{equation}

The bounds~\eqref{eq:q_3_1},~\eqref{eq:q_3_2},~and~\eqref{eq:q_3_3} imply
\begin{equation}
q_3 \leq 0.038 \cdot \left(\frac 1 4 + \frac 1 4 (1-p_{\mathrm{s}}(0))^2 + \frac 1 4 (1 + (1-p_{\mathrm{s}}(0)) + (1-p_{\mathrm{s}}(0))^2)\right) \stackrel{\text{Claim}~\ref{cl:binary_calculations}}< 0.0285.
\label{eq:q-third_case}
\end{equation}

Finally, bounds~\eqref{eq:q-first_case},~\eqref{eq:q-second_case},~and~\eqref{eq:q-third_case} conclude the proof Lemma~\ref{lm:binary_upper}:
\begin{itemize}
\item if $\xi_0(v_i)=1$ and $\xi_0(v_{i+2})=1$, then $k_+=2$, $k_-=0$, and 
$$
\mathbb{P}^*(\mathcal{E}_1(P))<0.0117+q_3\leq 0.0117 + 0.0055<0.02=2p_+=q(i)
$$ 
due to~\eqref{eq:q-first_case};
\item if $\xi_0(v_i)=-1$ and $\xi_0(v_{i+2})=-1$, then $k_+=0$, $k_-=2$, and 
$$
\mathbb{P}^*(\mathcal{E}_1(P))<0.0117+q_3\leq 0.0117 + 0.0903<0.1026=2p_-=q(i)
$$ 
due to~\eqref{eq:q-second_case};  
\item if $\xi_0(v_i)\neq\xi_0(v_{i+2})$, then $k_+=k_-=1$, and 
$$
\mathbb{P}^*(\mathcal{E}_1(P))<0.0117+q_3\leq 0.0117 + 0.0285 = 0.0402< 2\sqrt{p_-p_+}=q(i)
$$ 
due to~\eqref{eq:q-third_case}.
\end{itemize}

\subsection{Proof of Claim~\ref{cl:binary_calculations}}
\label{sc:4:4}


Let us start with strong $t$-stability inequalities. We shall prove that there exists an $\varepsilon>0$ such that, for every $h\in\mathbb{Z}_{>0}$ and for every non-leaf non-root $v \in V(T=T^{(h)})$,
\begin{equation}
p_{\mathrm{s}}(3,h,v) > 0.4453+\varepsilon,
\quad
p_{\mathrm{s}}(2,h,v) > 0.5617+\varepsilon,
\quad
p_{\mathrm{s}}(1,h,v) \geq \frac{1}{2}p_{\mathrm{w}}(0)^2,
\quad
p_{\mathrm{s}}(0,h,w) > 0.5+\varepsilon.
\label{eq:four_inequalities_strong}
\end{equation}
If $h_{T}(v) = 1$, then $v$ is 0-stationary, hence $p_{\mathrm{s}}(i, h, v) = 1$ for every $i \in \{0 ,1, 2, 3\}$. This proves all the four inequalities. If $h_{T}(v) = 2$, then $v$ is 0-stable (and hence strongly 2-stable) if two non-sibling grandchildren of $v$ share its initial opinion. The latter holds with probability $9/16>0.5617$ --- this proves the second and the fourth inequalities. Also $v$ is strongly $1$-stable (and hence strongly 3-stable) if both of its children share an initial opinion. This event holds with probability $1/2$, implying the first and the third inequalities. Now, we can assume that $h_{T}(v) \geq 3$.

Let us prove all the four inequalities from~\eqref{eq:four_inequalities_strong} one by one from the leftmost to the rightmost. In the proof of the $j$th bound, we will assume that all bounds to the right of the $j$th bound from the statement of Claim~\ref{cl:binary_calculations} hold true. First, let us bound $p_{\mathrm{s}}(3, h, v)$ from below by $p_{\mathrm{s}}(1, h, v)$+$\mathbb{P}(\mathcal{A})$ where $\mathcal{A}$ is an event, defined below, that implies that
$v$ is strongly 3-stable but not strongly 1-stable. This is indeed possible since strong 1-stability implies strong 3-stability. Let $u_1,u_2$ be the children of $v$. We define the event $\mathcal{A}$ as follows:
\begin{itemize}
 \item $u_1$ is strongly 2-stable;
 \item initial opinion of $u_2$ is opposite to $\xi:= \xi_2(u_1)$;
 \item all 4 grandchildren of $u_2$ have have initial opinion $\xi$;
 \item at least one of the 4 grandchildren of $u_2$ is weakly 0-stable. 
\end{itemize} 
Let us first show that $\mathcal{A}$ implies that $v$ is not strongly 1-stable. Indeed, notice that the vertex $u_2$ has opinion opposite to $\xi$ at time 0, while $\xi_2(u_2)=\xi$ since all grandchildren of $u_2$ have initial opinion $\xi$. Therefore, for the extension $\tilde\xi_0$ of $\xi_0|_{V(T_v)}$ to the whole tree such that all the vertices outside of the subtree $T_v$ share the same initial opinion opposite to $\xi$, we get $\tilde\xi_1(v)\neq\xi$, while, for the extension $\tilde\xi'_0$ such that all the vertices outside $T_v$ share the initial opinion $\xi$, we get $\tilde\xi'_3(v)=\xi$.

It remains to show that $\mathcal{A}$ implies strong 3-stability of $v$. Let us prove that, for any extension $\tilde\xi_0$ of $\xi|_{v(T_v)}$ to the whole tree, $v$ is 3-stable and $\tilde\xi_3(v) = \xi$. Indeed, $\tilde\xi_2(u_2) = \xi$ since all grandchildren of $u_2$ have initial opinion $\xi$. Applying Claim~\ref{cl:weak_rising} twice --- to the parent of a weakly 0-stable grandchild of $u_2$ and then to $u_2$ itself, we get that $u_2$ is weakly 2-stable. Also $u_1$ is weakly 2-stable as it is strongly 2-stable due to the definition of $\mathcal{A}$. By Claim~\ref{cl:weak_stabilization}, both $u_1$ and $u_2$ are 2-stable, and hence $v$ is 3-stable. Since we conclude this only from the event $\mathcal{A}$ that is defined solely by the initial opinions on vertices of $T_v$, we get that $v$ is strongly 3-stable, as needed.

Finally, 
\begin{align*}
p_{\mathrm{s}}(3, h, v) & \geq p_{\mathrm{s}}(1, h, v)+\mathbb{P}(\mathcal{A})\\
&\geq p_{\mathrm{s}}(1) + p_{\mathrm{s}}(2) \cdot \frac{1}{2} \cdot \frac{1}{16}\cdot \biggl(1 - (1 - p_{\mathrm{w}}(0))^4\biggr) \\
&> 0.5 \cdot 0.925^2 + 0.5617 \cdot \frac{1}{32}\cdot(1 - 0.075^4) > 0.4453.
\end{align*}
The first bound in~\eqref{eq:four_inequalities_strong} follows.

Now, let us bound $p_{\mathrm{s}}(2, h, v)$. Similarly, we bound it from below by $p_{\mathrm{s}}(0, h, v)$+$\mathbb{P}(\mathcal{B})$ for some $\mathcal{B}$ implying that
$v$ is strongly 2-stable but not strongly 0-stable. This event is defined as follows:
\begin{itemize}
 \item all 4 grandchildren of $v$ share the opposite to $v$ initial opinion;
 \item 2 non-siblings among the grandchildren of $v$ are weakly 0-stable. 
\end{itemize} 
Let us show that $\mathcal{B}$ is the desired event. Parents $u_1$ and $u_2$ of $w_1$ and $w_2$ (which are the weakly 0-stable non-siblings), that are precisely the two children of $v$, have opinion $\xi := \xi_0(w_1)$ at time step 1. So, $\xi_2(v) = \xi \neq \xi_0(v)$ --- thus, $v$ is not strongly 0-stable. It remains to show that $v$ is strongly 2-stable. 
By Claim~\ref{cl:maintain_weak_value}, $\xi_2(w_1) = \xi_2(w_2) = \xi$.
Hence, $w_1$ and $w_2$ also remain weakly 2-stable, due to Claim~\ref{cl:maintain_weak_stability}. Due to Claim~\ref{cl:weak_stabilization}, $v$ is 2-stable. Since we conclude this only from the event $\mathcal{B}$ that is defined solely by the initial opinions on vertices of $T_v$, we get that $v$ is strongly 2-stable, as needed.

Finally, 
$$
p_{\mathrm{s}}(2, h, v) \geq p_{\mathrm{s}}(0) + \frac{1}{16}\biggl(1 - (1 - p_{\mathrm{w}}(0))^2\biggr)^2 > 0.5+\frac{1}{16}(1-0.075^2)^2 > 0.5617.
$$
The second bound in~\eqref{eq:four_inequalities_strong} follows. \\

Let us prove the third inequality from~\eqref{eq:four_inequalities_strong}. Let $u_1, u_2$ be the children of $v$. By Claim~\ref{cl:weak_stabilization}, if $u_1$ and $u_2$ are both weakly 0-stable and share the initial opinion, then they are 0-stable, and hence $v$ is strongly 1-stable. Due to Remark~\ref{rm:independancy_general} and Remark~\ref{rm:equal_strong_probabilities}, we have that
$$
p_{\mathrm{s}}(1, h, v) \geq \frac{1}{2}p_{\mathrm{w}}(0, h, u_1) p_{\mathrm{w}}(0, h, u_2).
$$
 So, $p_{\mathrm{s}}(1,h,v) \geq\frac{1}{2} p_{\mathrm{w}}(0)^2$ follows from taking infimum of the right-hand side of the last inequality.\\

Let us finally prove the last inequality in~\eqref{eq:four_inequalities_strong}: $p_{\mathrm{s}}(0,h,v) > 0.5+\varepsilon$. Let $\xi := \xi_0(v)$. Let $u_1, u_2$ be children of $v$. If $u_1$ has child $w_1$ and $u_2$ has child $w_2$ such that $\xi_0(w_1) = \xi_0(w_2) = \xi$, and $w_1$ and $w_2$ are weakly $0$-stable, then $v$ is $0$-stable, due to Claim~\ref{cl:weak_stabilization}. 
So, $v$ is strongly $0$-stable with probability at least $\left(1 - \left(1 - \frac{p_{\mathrm{w}}(0)}{2}\right)^2\right)^2$. This probability is more than $1/2 + \varepsilon$ if and only if $p_{\mathrm{w}}(0) > 2 - \sqrt{2(2-\sqrt{2})}$. It remains to prove that $p_{\mathrm{w}}(0) > 0.925$, as $0.925 >  2 - \sqrt{2(2-\sqrt{2})}$. \\


Let us prove that $p_{\mathrm{w}}(0) > 0.925$. Note that this inequality would complete the proof of Claim~\ref{cl:binary_calculations}. First, if $v$ is a leaf, then 
$p_{\mathrm{w}}(0, h, v) = 1$ due to the definition. So, further we assume that $v$ is not a leaf. Let us denote by $\mathcal{N}$ the subset of all numbers from $\{0,\ldots,h_T(v)\}$ with the same parity as $h_T(v)$, assuming that 0 is even. For every $t \in \mathcal{N}$, all vertices $u \in V(T_v)$ at distance $t$ from leaves have the same $p_{\mathrm{w}}(0, h, u)$, due to the definition of weak stability. For brevity, let us denote $1 - p_{\mathrm{w}}(0, h, u)$ by $q_{t}$. Let $q>0$ be the smallest positive real solution of $x = P(x)$, where 
$$
P(x) := P_1(x) + P_2(x)P_3(x)
$$ 
and $P_1,P_2,P_3$ are defined by 
\begin{align} 
P_1(x) & := \left(\frac{1}{4}+\frac{x^2}{4}\right)^2,\label{eq:P_1}\\
P_2(x) & := \left(1-\frac{1-x^2}{4}\right)^2,\label{eq:P_2}\\
P_3(x) & := \left(\frac{1}{2}+\frac{x}{2}\right)^4 - \left(\frac{1}{4}+\frac{x^2}{4}\right)^2.\label{eq:P_3}
\end{align}
Since $P(0)=1/16$, $P(3/40)<3/40$, and $P$ has only non-negative coefficients (implying that $P$ increases on $(0,\infty)$),
 we get that $q\in(1/16,3/40)$.  
 
It is sufficient to prove by induction, that $q_t \leq q$ for every $t \in \mathcal{N}$, as $p_{\mathrm{w}}(0, h, v) = 1 - q_{h_T(v)}$ and $q < 3/16=0.075$. For the base of induction, we check $t\in\{0,1,2,3\}$. Clearly, $q_t = 0$ for $t \in \{0, 1\} \cap \mathcal{N}$. Moreover, $q_t = \frac{1}{16} < q$ for $t \in \{2, 3\} \cap \mathcal{N}$ since in these two cases $u$ is weakly $0$-stable if and only if it has a grandchild with the same initial opinion as $u$ has. 

For the induction step let us fix $t \in \mathcal{N}$ ($t \geq 4$) and a vertex $u$ at distance $t$ from leaves. Let $\hat\xi := \xi_0(u)$. Let $\mathcal{A}, \mathcal{B}, \mathcal{C}$ be three disjoint events defined by $\xi_0|_{T_u}$ in the following way:
\begin{enumerate}
    \item $\mathcal{A}$ holds when 
    \begin{itemize}
        \item every grandchild of $u$ with initial opinion $\hat\xi$ is not weakly $0$-stable;
        \item for every child of $u$, its children share the initial opinion (i.e. initial opinions of siblings that are grandchildren of $u$ coincide).
    \end{itemize}
    \item $\mathcal{B}$ holds when at least one grandchild of $u$ is weakly $0$-stable with initial opinion $\hat\xi$;
    \item $\mathcal{C} = \bar{\mathcal{A}} \cap \bar{\mathcal{B}}$.
\end{enumerate}

If $\mathcal{B}$ holds then $u$ is weakly $0$-stable due to Claim~\ref{cl:weak_grandparent}. Therefore, 
$$
\mathbb{P}\left(u \text{ is not weakly 0-stable}\mid\mathcal{B}\right)=0
$$
implying that
$$q_t=\sum_{\mathcal{Q}\in\{\mathcal{A},\mathcal{B},\mathcal{C}\}}
\mathbb{P}\left(u \text{ is not weakly 0-stable}\mid\mathcal{Q}\right)\mathbb{P}(\mathcal{Q})
\leq \mathbb{P}(\mathcal{A}) + \mathbb{P}\left(u \text{ is not weakly 0-stable}\mid\mathcal{C}\right)\mathbb{P}(\mathcal{C}).$$
From the definition of events $\mathcal{A},\mathcal{B},\mathcal{C}$, it immediately follows that 
$$
\mathbb{P}(\mathcal{A}) = \left(\frac{1}{4}+\frac{q_{t-2}^2} {4}\right)^2 = P_1(q_{t-2}),
$$
where $P_1$ is defined in~\eqref{eq:P_1}, and 
$$
\mathbb{P}(\mathcal{C}) = \mathbb{P}(\bar{\mathcal{B}}) - \mathbb{P}(\mathcal{A}) = \left(\frac{1}{2}+\frac{q_{t-2}}{2}\right)^4 - \left(\frac{1}{4}+\frac{q_{t-2}^2}{4}\right)^2 = P_3(q_{t-2}),
$$
where $P_3$ is defined in~\eqref{eq:P_3}.

Let us notice, that all three polynomials $P_1,P_2,P_3$ defined in~\eqref{eq:P_1}--\eqref{eq:P_3} have only non-negative coefficients, hence they increase on $[0, q]$. Since $P(q) = q$, we get that, for any $0 \leq x_1, x_2, x_3 \leq q$,
$$
P_1(x_1) + P_2(x_2)P_3(x_3) \leq P_1(q) + P_2(q)P_3(q)= P(q)=q.
$$
So, it remains to prove the first inequality in 
$$
\mathbb{P}\left(u \text{ is  not weakly 0-stable}\mid\mathcal{C}\right) \leq \left(1-\frac{1-q_{t-4}^2}{4}\right)^2 = P_2(q_{t-4}).
$$

Let us enumerate all $16$ possible initial opinions of the four grandchildren of $u$ arbitrarily. We decompose $\mathcal{C}=\sqcup_{j=1}^{16}\mathcal{C}_j$, where $\mathcal{C}_j$ says that $\mathcal{C}$ holds and the four grandchildren of $u$ have the $j$-th 4-tuple of opinions. Fix $j\in[16]$ such that $\mathbb{P}(\mathcal{C}_j) \neq 0$. 
By the definition, $\mathcal{C}$ (and therefore $\mathcal{C}_j$) implies that there is a child $v_1$ of $u$ and two children  $u_1, u_2$ of $v_1$ such that $\xi_0(u_1) \neq \hat\xi$ and $\xi_0(u_2) = \hat\xi$. Let us notice that, since $\xi_0(u_1) \neq \hat\xi$, any event defined by $\xi_0|_{V(T_{u_1}) \setminus \{u_1\}}$ is independent from $\mathcal{C}_j$. To conclude the proof, let us construct an event $\mathcal{D}$ defined by $\xi_0|_{V(T_{u_1}) \setminus \{u_1\}}$ with probability at least $1 - \left(1-(1-q_{t-4}^2)/4\right)^2$ such that $\mathcal{C}_j \cap \mathcal{D}$ implies the weak 0-stability of $u$.

The desired event $\mathcal{D}$ is defined in the following way. It says that there are two grandchildren $w_1, w_2$ of $u_1$ (see Figure~\ref{fig:tree_arrangement}) such that
\begin{itemize}
    \item $w_1,w_2$ are siblings (i.e. they share a parent);
    \item $\xi_0(w_1) = \xi_0(w_2) = \hat\xi$;
    \item at least one of $w_1, w_2$ is weakly 0-stable.
\end{itemize}
The required condition that $\mathcal{D}$ is defined by $\xi_0|_{V(T_{u_1}) \setminus \{u_1\}}$ and the inequality 
$$
\mathbb{P}(\mathcal{D})\geq
1 - \left(1-\frac{1-q_{t-4}^2}{4}\right)^2
$$
are immediate. Now, suppose that $\mathcal{C}_j \cap \mathcal{D}$ holds. 
\begin{figure}[h]
    \centering
    \includegraphics[scale=0.20]{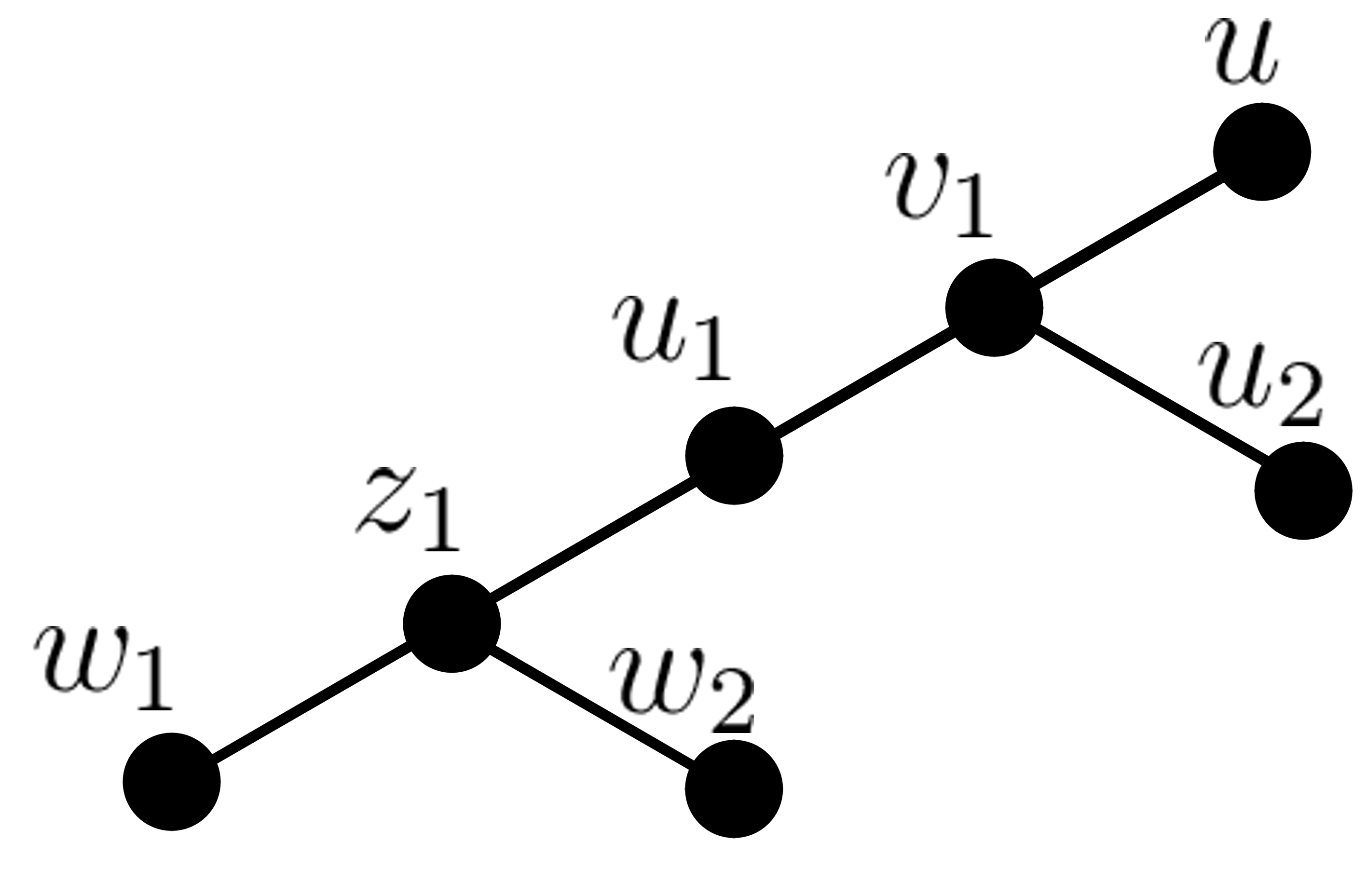}
    \caption{Descendants of $u$}
    \label{fig:tree_arrangement}
\end{figure}
We have to derive the 
weak 0-stability of $u$. 
Let us recall that $\xi_0(u) = \xi_0(u_2) = \hat\xi$, so $\xi_1(v_1) = \hat\xi$ as $u,u_2$ are neighbours of $v_1$. In addition, at time 1, the parent $z_1$ of $w_1$ and $w_2$ holds the opinion $\hat\xi$. Since $u_1$ is adjacent to $z_1$ and to $v_1$, we get that $\xi_2(u_1) = \hat\xi$. Then, applying Claim~\ref{cl:weak_rising} twice, we first get that $z_1$ is weakly 1-stable, and then that $u_1$ is weakly 2-stable. Since $\xi_1(v_1)=\hat\xi$, the extension $\tilde\xi_0$ of $\xi_0|_{T_u}$ to the entire tree $T$ satisfying $\tilde\xi_0(u^*)=\hat\xi$ for all $u^*\notin V(T_u)$ forces $u$ to inherit the opinion $\hat\xi$ at time $2$, i.e. $\tilde\xi_2(u)=\hat\xi$. Note that $u_1$ is weakly 2-stable w.r.t. $\tilde\xi_0$ as well since the event $\mathcal{C}_j\cap\mathcal{D}$ is defined by $\xi_0|_{T_u}$. Then, by Claim~\ref{cl:weak_grandparent}, $u$ is weakly 2-stable w.r.t. $\tilde\xi_0$. It is also 2-stable (w.r.t. to $\tilde\xi_0$) by the second equivalent definition of weak 2-stability from Claim~\ref{cl:equivalent_weak_definitions}, due to the substitution in the definition $\tilde\xi_0^* := \tilde\xi_2$, and hence $u$ is 0-stable. So finally we get that $u$ is weakly 0-stable w.r.t. $\xi_0$, due to the second equivalent definition of weak 0-stability, once again by substitution $\tilde\xi_0^* := \tilde\xi_0$, completing the proof.

\section{Perfect $k$-ary trees: proof of Theorem~\ref{th:3}}
\label{sc:5}

In this section, we prove Theorem~\ref{th:3}, that asserts both the lower and the upper bound for the typical stabilisation time on a perfect $k$-ary tree. Our upper bound is an immediate corollary of Theorem~\ref{th:1}. Indeed, $\mathcal{Q}$ only consists of paths with ends at distance at least 2 from the leaves. Hence, the maximum of $t(Q)$ over all paths in $\mathcal{Q}$ is achieved at the longest such path. So, the maximum equals $D-3$.

In this section, we prove the lower bound. We first introduce the notion of being $(\leq t)$-{\it{stable}}, then state the lower bound in a separate Lemma~\ref{lm:k_ary_lower} and prove it, completing the proof of Theorem~\ref{th:3}.

For every positive integer $h$, we fix a perfect $k$-ary tree $T^{(h)}$ of depth $h$ rooted in a vertex $R$. When the height is clear from the context we omit the superscript and write simply $T$. For a uniformly random $\xi_0$ on $V(T)$, we let $\tau:=\tau(T;\xi_0)$.

Let us say that a vertex $v \in V(T)$ is $(\leq t)$-{\it{stable}}, if $\tilde\xi_{t'}(v) = \tilde\xi_t(v)$ holds for every vector of initial opinions $\tilde\xi_0$ such that $\tilde\xi_0|_{V(T_v)} = \xi_0|_{V(T_v)}$ and for every time step $t' < t$ with the same parity as $t$. Let us note that being $(\leq t)$-stable is the property of $\xi_0|_{V(T_v)}$. So, Remark~\ref{rm:independancy_general} is applicable here too (i.e. events $ \{v \textit{ is } (\leq t)\textit{-stable} \} $ are independent for vertices $v$ which are not descendants of each other).

\begin{claim}
There exists $C > 0$ such that, for every even $k>2$, every opinion $\xi \in \{-1, 1\}$, every non-root non-leaf vertex $v\in V(T)$ of a perfect $k$-ary tree $T$, and every time step $t\in \mathbb{Z}_{\geq 2}$, probability that $v$ is $(\leq t)$-stable and $\xi_{t}(v) = \xi$ is at least $2^{-Ctk}$.
\label{cl:stable_until_probability}
\end{claim}

\begin{proof}
Without loss of generality we may assume that $h_T(v)\geq 2$ since otherwise the desired event follows from the property that more than a half of leaves adjacent to $v$ have opinion $\xi$; and the latter event has probability at least $1/4$.

We will show a sufficiently large class $\mathcal{F}=\mathcal{F}(k,\xi,v,t)$ of vectors of initial opinions $\xi_0^*: V(T_v) \rightarrow \{-1, 1\}$ such that, for any extension of $\xi_0^*$ to the whole tree, the vertex $v$ is $(\leq t)$-stable and $\xi_t(v) = \xi$. 

For a non-root vertex $u$, let the random vector $\hat\xi^{(u)}_0:V(T)\to\{-1,1\}$ have uniformly random distribution on $V(T_u)$ and $\hat\xi_0|_{V(T) \setminus V(T_u)} \equiv \xi$. 
 Let 
 \begin{center}
 $u_1, \ldots, u_k$ be children of $v$; \,\, $w_1, \ldots, w_k$ be children of $u_1$; \,\, and $z_1, \ldots, z_k$ be children of $u_2$. 
 \end{center}
 Let $U=\{u_3,\ldots,u_k,w_1,\ldots,w_k,z_1,\ldots,z_k\}$. 
  Let $t(u): U \to \mathbb{Z}_{\geq 0 }$ equal $t$ if $u = u_i$ for some $i \in \{3, \ldots, k\}$, and $t-1$, otherwise. For convenience, for vertices $u, w \in U$ such that $t(u) = t(w)$ we introduce the event
    $$
    E(u, w) :=
    \left\{\text{for all $0\leq i\leq t(u)-1$ with the same parity as $t(u)-1$, $\hat\xi^{(u)}_i(u)=\xi$ or $\hat\xi^{(w)}_i(w)=\xi$}\right\}.
    $$
We then define $\mathcal{F}$ as the set of $\xi_0^*$ such that: 
\begin{enumerate}
    \item for every odd $i \in \{3 , \ldots, k-1\}$, if $\hat\xi^{(u_i)}_0|_{V(T_{u_i})}=\xi^*_0|_{V(T_{u_i})}$ and $\hat\xi^{(u_{i+1})}_0|_{V(T_{u_{i+1}})}=\xi^*_0|_{V(T_{u_{i+1}})}$, then $E(u_i, u_{i+1})$ holds; 
    

    \item for every odd $i \in [k-1]$,
    if $\hat\xi^{(w_i)}_0|_{V(T_{w_i})}=\xi^*_0|_{V(T_{w_i})}$ and $\hat\xi^{(w_{i+1})}_0|_{V(T_{w_{i+1}})}=\xi^*_0|_{V(T_{w_{i+1}})}$, then
    $E(w_i, w_{i+1})$ holds; 
    

    \item for every odd $i \in [k-1]$,
    if $\hat\xi^{(z_i)}_0|_{V(T_{z_i})}=\xi^*_0|_{V(T_{z_i})}$ and $\hat\xi^{(z_{i+1})}_0|_{V(T_{z_{i+1}})}=\xi^*_0|_{V(T_{z_{i+1}})}$, then
    $E(z_i, z_{i+1})$ holds; 

    \item 
    $\xi^*_0(v) = \xi$ and $\xi^*_0(u_1) = \xi^*_0(u_2) = \xi$.
\end{enumerate}

Let $\xi_0$ be an extension of $\xi_0^*$ to the entire $T$. Let us first show that indeed $v$ is $(\leq t)$-stable w.r.t. $\xi_0$ and $\xi_t(v)=\xi$. It is sufficient for us to show the following stronger statement:
\begin{itemize}
    \item[(1)] for every $u\in U$ and every $s\in\{0,\ldots,t(u)-1\}$ with the same parity as $t(u)-1$, if $\hat\xi_0^{(u)}|_{V(T_u)}=\xi_0^*|_{V(T_u)},$
    then $\xi_s(u)=\hat\xi^{(u)}_s(u)$;
    

    \item[(2)] $\xi_s(u_1) = \xi_s(u_2) = \xi$ for every $s \in \{0, \ldots, t-1\}$ with the same parity as $t-1$;
    
    \item[(3)] $\xi_{s}(v) = \xi$, for every $s \in \{0, \ldots, t-1\}$ with the same parity as $t$.
\end{itemize}
Indeed, if that holds, recalling that $\xi_0|_{V(T_v)}\in\mathcal{F}$,  we get that more than a half of children of $v$ at time step $t-1$ have opinion $\xi$, and so $\xi_t(v) = \xi$. Along with the property (3), it implies $(\leq t)$-stability of $v$ as needed.\\


In order to prove the statement, suppose towards contradiction that $s < t$ is the earliest time step when at least one vertex from $U \sqcup \{v, u_1, u_2\}$ has opinion other than we declared. Let us first assume that one of the vertices $v, u_1, u_2$ have such a different opinion at time $s$. By the condition 4 in the definition of $\mathcal{F}$, they have the desired opinion at time step $0$, i.e. $s\geq 1$. If $s$ has the same parity as $t$, then at least $k/2+1$ of children of $v$ have opinion $\xi$ at time $s-1$ since (1) holds for time step $s-1$ and due to the condition 1 in the definition of $\mathcal{F}$. So $\xi_{s}(v) = \xi$. In addition, if $s$ has different parity than $t$, then at least $k/2$ of children of $u_1$ have opinion $\xi$ at time $s-1$ since (1) holds at time step $s-1$ and due to the condition 2 in the definition of $\mathcal{F}$. Since the parent of $u_1$ also has opinion $\xi$ at time $s-1$, we get that $\xi_{s}(u_1) = \xi$. Similarly, $\xi_s(u_2) = \xi$ --- a contradiction. Now, assume that a vertex $u \in U$ has an opinion at time $s$ which is not declared by (1) and the definition of $\mathcal{F}$. Due to assumption that (1)--(3) hold for time moments less than $s$,  the parent of $u$ has opinion $\xi$ at every time $s' < s$ with different parity, hence $u$ inherits the desired opinion --- a contradiction.

It remains to show that $\mathcal{F}$ is large. Let a pair of vertices $(u,w)$ be one of the following:
\begin{itemize}
    \item $(u_i, u_{i+1})$, for some odd $i \in \{3 , \ldots, k-1\}$; 
    \item $(w_i, w_{i+1})$, for some odd $i \in [k-1]$; 
    \item $(z_i, z_{i+1})$, for some odd $i \in [k-1]$.
\end{itemize}

Hence, $t(u) = t(w)$. In order to bound the probability of $\mathcal{F}$, first, let us bound the probability of the event $E(u, w)$. Let us note, that vectors
$$
\phi(\hat\xi_0^{(u)}):=\left(\hat\xi^{(u)}_{(t(u) - 1) \bmod 2}(u), \hat\xi^{(u)}_{((t(u) - 1) \bmod 2) + 2}(u) \ldots, \hat\xi^{(u)}_{t(u)-1}(u)\right)
$$ 
and
$$
\phi(\hat\xi_0^{(w)}):=\left(\hat\xi^{(w)}_{(t(w) - 1) \bmod 2}(w), \hat\xi^{(w)}_{((t(w) - 1) \bmod 2) + 2}(w) \ldots, \hat\xi^{(w)}_{t(w)-1}(w)\right)
$$ 
are independent and identically distributed. Let $\Phi$ be the set of all atoms of this distribution.

Let us fix some isomorphism $\varphi$ from $T_w$ to $T_u$. Suppose that $\hat\xi_0^{(w)}|_{V(T_w)} = -\hat\xi_0^{(u)} \circ\varphi|_{V(T_w)}$. Let $\hat\xi_s^{(u)}(u) \neq \xi$, for some $s$. Then, due to symmetry and the fact that $\hat\xi_0^{(w)}|_{V(T)\setminus V(T_w)}\equiv\xi$, we get $\hat\xi_s^{(w)}(w) = \xi$ as well. Therefore, $E(u, w)$ holds. So, for every $\phi \in \Phi$,
$$
\mathbb{P}(E(u,w)\mid\phi(\hat\xi_0^{(u)}) = \phi)\geq\mathbb{P}(\phi(\hat\xi_0^{(w)}) = \phi)=\mathbb{P}(\phi(\hat\xi_0^{(u)}) = \phi).
$$
Thus,
\begin{align*}
\mathbb{P}(E(u,w)) &=\sum_{\phi\in\Phi}\mathbb{P}(E(u,w)\mid \phi(\hat\xi_0^{(u)}) = \phi)\mathbb{P}(\phi(\hat\xi_0^{(u)}) = \phi)
\geq 
\sum_{\phi \in \Phi}\mathbb{P}\left(\phi(\hat\xi_s^{(u)}) = \phi\right)^2\\
& \stackrel{\text{QM-AM}}\geq \frac{1}{|\Phi|}\biggl(\sum_{\phi \in \Phi} \mathbb{P}\left(\phi(\hat\xi_s^{(u)}) = \phi\right)\biggr)^2 = \frac{1}{|\Phi|}\geq
2^{-\lceil t(u)/2 \rceil}.
\end{align*}
Since the event $E(u,w)$ is a function of random variables associated with vertices of $T_{u}$ and $T_{w}$, we get that these events are independent over the set of considered pairs $(u,w)$. So, applying the latter bound to every pair $(u,w)$, we conclude that 
$$
\mathbb{P}\left(\hat\xi_0^{(v)}|_{V(T_v)}\in\mathcal{F}\right)\geq 2^{-3} \left(\frac{1}{2^{\lceil t/2 \rceil}}\right)^{\frac{k-2}{2}}  \left(\frac{1}{2^{\lceil (t-1)/2 \rceil}}\right)^{k} 
> 2^{-Ctk},$$
for some constant $C>0$, completing the proof.

\end{proof}


\begin{lemma}

There exists $c_->0$ such that for any even $k>2$, any $R$-rooted perfect $k$-ary tree $T$ of height $h$, and uniformly random $\xi_0\in\{-1,1\}^{V(T)}$, whp there is at least one not $\lfloor c_- \cdot \sqrt{h \ln{k}}/k\rfloor$-stationary vertex.
\label{lm:k_ary_lower}   
\end{lemma}

\begin{proof}
Let $d$ be the smallest even number bigger than $\lfloor c_- \cdot \sqrt{h \ln{k}}/k\rfloor + 2$. We will choose the value of the constant $c_-$ at the end of the proof. 

Let $R^* \in V(T)$ be at distance $d$
from a leaf, define $T^* := T_{R^*}$. First, we will define a class $\mathcal{F}$ of functions $\xi_0^*: V(T^*) \rightarrow \{-1, 1\}$ such that for each initial opinion $\xi_0^*$ on $V(T^*)$ from $\mathcal{F}$ there exists not $(d-3)$-stable for every extension of $\xi_0^*$ to the entire $V(T)$. Secondly, we will show that this class is so large that whp, for a uniformly random $\xi_0\in\{-1,1\}^{V(T)}$, there exists $R^*$ as above so that $\xi_0|_{V(T^*)}\in\mathcal{F}$  leading to the fact that 
$\tau\geq$ $d-2$ whp. The latter immediately implies the statement of Lemma~\ref{lm:k_ary_lower} due to the definition of $d$.

\paragraph{Definition of the desired set $\mathcal{F}$.}

The following requirements on $\xi_0^*$ define the set $\mathcal{F}$. There exists a path $v_2v_3\ldots v_{d}$ with $v_d := R^*$ such that for any extension $\xi_0$ to the entire $V(T)$

\begin{itemize}

\item[A1] for every even $t\in[d]$, $\xi_0(v_t) = 1$ holds;

\item[A2] for every $t \in \{3, \ldots, d\}$, all children of $v_t$, other than $v_{t-1}$, are $(\leq t-1)$-stable;

\item[A3] for every $t \in \{3, \ldots, d-1\}$, the number of children $u$ of $v_t$ such that $\xi_{t-1}(u) = -1$ and $u\neq v_{t-1}$ equals $k/2$;

\item[A4] for every child $u$ of $v_d$, other than $v_{d-1}$, $\xi_{d-1}(u) = 1$ holds;

\item[A5] all grandchildren of $v_2$ are leaves with initial opinion $-1$.

\end{itemize}

Let $\xi_0^*\in\mathcal{F}$ and let $\xi_0$ be an extension of $\xi_0^*$ to the entire $V(T)$.
Let us show that the vertex $v_{d-1}$ is not $(d-3)$-stable w.r.t. $\xi_0$. In order to prove that, it is sufficient to show that for every $s \in \{2, \ldots d-1\}$ 
\begin{itemize}
    \item $\xi_s(v_s) = -1$;
    \item $\xi_{s}(v_{i}) = 1$, for every $i \in [d]\setminus[s]$ with the same parity as $s$.
\end{itemize}
Let us prove the latter statement by induction. First, by A2 and A4, $v_d$ has at least $k - 1$ children with opinion $1$ at every odd time step less than $d$. So, $\xi_0(v_d) = \xi_2(v_d) = \ldots = \xi_d(v_d) = 1$. Next, let us note that all children of $v_2$ have opinion $-1$ at time $1$, so $\xi_2(v_2) = -1$.
We will further rely on the following fact: For every time step $t \in [d-1]$ and every $i \in \{\max\{t,3\}, \ldots, d-1\}$ with the same parity as $t$, 
\begin{equation}
\text{
$\xi_t(v_i)=-1$ if and only if either $\xi_{t-1}(v_{i-1}) = -1$ or $\xi_{t-1}(v_{i+1}) = -1$.
}
\label{eq:fact}
\end{equation}
Indeed, (1) by A2, for every neighbour $u$ of $v_i$ other than $v_{i-1}$ and $v_{i+1}$, $\xi_{i-1}(u) = \xi_{t-1}(u)$, and (2) by A3, for exactly $k/2$ neighbours $u$ of $v_i$ other than $v_{i-1}$ and $v_{i+1}$, $\xi_{i-1}(u) = -1$. So, from (1) and (2) we conclude that, for exactly $k/2$ neighbours $u$ of $v_i$ other than $v_{i-1}$ and $v_{i+1}$, $\xi_{t-1}(u) = -1$, implying~\eqref{eq:fact}. 

Now, let us derive both the base of induction ($s=2$) and the induction step from~\eqref{eq:fact}. To prove the base, let us first apply~\eqref{eq:fact} with $t = 1$ and every odd $i \in \{3, \ldots, d-1\}$. Due to A1,  $\xi_0(v_{i-1}) = \xi_0(v_{i+1}) = 1$,  hence $\xi_1(v_i) = 1$. This allows us to apply~\eqref{eq:fact} with $t = 2$ and every even $i \in \{3, \ldots, d-1\}$ and get $\xi_2(v_i) = 1$, concluding the base of induction (we already showed $\xi_2(v_2) = -1$ and $\xi_2(v_d) = 1$). Finally, for the induction step, we assume that $s \in \{3, \ldots, d-1\}$ and that for smaller time steps the statement is true. We apply~\eqref{eq:fact} with $t = s$ and $i = s$ and get $\xi_{s}(v_s) = -1$, since $\xi_{s-1}(v_{s-1}) = -1$. Eventually, we observe that from the induction assumption it follows that $\xi_{s-1}(v_i)=1$ for all $i\in\{s+1,\ldots,d\}$ with the same parity as $s-1$. Thus we may
apply~\eqref{eq:fact} with $t = s$ and every $i \in \{t+2, \ldots, d-1\}$ with the same parity as $t$ and get $\xi_{s}(v_{i}) = 1$. Recalling $\xi_s(v_d)=1$ whenever $s$ is even, we complete the induction step. 

\paragraph{The existence of $R^*$.}

First, let us note that probabilities of the events A1 and A5 are $2^{-d/2}$ and $2^{-k^2}$ respectively. Next, let $U$ be the set of all children of vertices from the path $v_3\ldots v_d$, excluding the vertices of the path and $v_2$. 
 Next, let $I: U \rightarrow \{2, \ldots , d-1\}$ be such that
$$
I(u) = i-1 \quad\text{ if and only if }\quad v_i \text{ is the parent of } u.
$$
Next, consider an arbitrary function $O: U \rightarrow \{-1, 1\}$ such that
\begin{itemize}
    \item $O(u) = 1$ for every $u$ with $I(u) = d-1$;
    \item for every $i \in \{2, \ldots, d-2\}$, the number of vertices $u \in U$ with $I(u) = i$ and $O(u) = -1$ equals $k/2$.
\end{itemize}
Note that such a function exists, since, for every $i \in \{2, \ldots, d-2\}$, the number of vertices $u \in U$ with $I(u) = i$ equals $k-1 > k/2$.

We note that A2, A3, and A4 follow immediately if every vertex $u\in U$ is $(\leq I(u))$-stable and has $\xi_{I(u)}(u) = O(u)$. Next, let us note, that events $\{u \text{ is } (\leq I(u))\text{-stable and } \xi_{I(u)}(u) = O(u)\}$ are independent over $u\in U$ since these events depend only on vectors of initial opinions $\xi_0|_{V(T_u)}$. Therefore, the events A1, A5, and $\{u \text{ is } (\leq I(u))\text{-stable and } \xi_{I(u)}(u) = O(u)\}$, $u \in U$, are independent, as they are defined by initial opinions on disjoint sets of vertices. Therefore, 
$$
\mathbb{P}(\xi_0^* \in \mathcal{F}) \geq 2^{-d/2 - k^2} \prod_{u \in U} \mathbb{P}\biggl(u \text{ is } (\leq I(u))\text{-stable and } \xi_{I(u)}(u) = O(u)\biggr).
$$
Due to Claim~\ref{cl:stable_until_probability},
$$
\mathbb{P}\biggl(u \text{ is } (\leq I(u))\text{-stable and } \xi_{I(u)}(u) = O(u)\biggr) \geq 2^{-CkI(u)}
$$
for every $u\in U$. Thus, we conclude
\begin{equation}
\mathbb{P}(\xi_0^* \in \mathcal{F}) \geq 2^{-d/2 - k^2} \prod_{u \in U} 2^{-CkI(u)} = 2^{-d/2 - k^2} \prod_{i\in\{2, \ldots, d-1\}} (2^{-Cki})^{k-1} \geq 2^{-(C+2)d^2k^2}.
 \label{eq:sec_general_k_pre_final}
 \end{equation}
 In addition, events $\{\xi_0|_{V(T_{R^*})}\in\mathcal{F}=\mathcal{F}(R^*)\}$ are independent over all vertices $R^*$ at distance $d$ from leaves. So, since the number of vertices at distance $d$ from leaves is $(k+1)k^{h-d-1}$, it is sufficient for us to show 
 \begin{equation}
 (k+1)k^{(h-d-1)} \cdot \mathbb{P}(\xi_0^* \in \mathcal{F}) = \Omega(1).
 \label{eq:sec_general_k_final}
 \end{equation}
 Due to~\eqref{eq:sec_general_k_pre_final}, we achieve~\eqref{eq:sec_general_k_final} by choosing small enough constant $c_-$ in the definition of $d$.
\end{proof}

\section{Discussions}
\label{sc:dis}

Two immediate questions remain unanswered. First, what is the exact asymptotics in Theorem~\ref{th:2}? We suspect, for a perfect binary tree $T$ and a uniformly random $\xi_0$, there exists $c\in(1/4,1/3)$ such that $\frac{\tau(T;\xi_0)}{n}\stackrel{P}\to c$. Determining this limit would be of great interest, within the scope of majority dynamics. Second, in Theorem~\ref{th:3}, our methods do not capture the correct dependency on $D$, so we leave that as an open problem. Below, we propose several other directions for future investigation.

\paragraph{Stabilisation time.} In all the processes that we observed on finite graphs, there is always at least one vertex which is stable from the beginning. Is it always the case? More formally, is it true that for every finite graph $G$ and for every vector of initial opinions $\xi_0\in\{-1,1\}^{V(G)}$, there exists a vertex $v\in V(G)$ such that $\xi_{t+2}(v)=\xi_t(v)$ for all $t\in\mathbb{Z}_{\geq 0}$?

Recall that $\tau(G)<|E(G)|$ for every graph $G$. Let $\delta(G)$ be the minimum degree of $G$. It seems plausible that there exists a universal constant $C>0$ such that, for every graph $G$, $\tau(G)\leq C\frac{|E(G)|}{\delta(G)}$. This bound clearly holds for all trees due to Theorem~\ref{th:1}. Note that it cannot be improved in such generality since, in particular, it is achieved by the wheel graph.

\paragraph{Stable partitions.} Let us say that $(U,W)$ is a {\it stable partition} of $V(G)$, if, letting $\xi_0(u)=1$ for $u\in U$ and $\xi_0(w)=-1$ for $w\in W$, we get that all vertices of $G$ are 0-stationary. Observe that any unfriendly partition (including trivial cases $U=\varnothing$ or $W=\varnothing$) is stable. In particular, these are the only stable partitions of a clique. Can one characterise graphs with only these stable partitions? What are the stable partitions of the random graph $G(n,1/2)$?

Clearly, for every $G$ and $\xi_0$, letting $\tau:=\tau(G;\xi_0)$, we get that $\xi_{\tau}$ induces a stable partition. More generally, given graph $G$ and a positive integer $t$, what is the set of partitions $\Pi_t(G)$ achievable by the process at time $t$? How fast does $|\Pi_t(G)|$ decay? Note that there are graphs with exponentially many stable partitions. For instance, identify every vertex of an $n$-vertex cubic graph with a clique of size 5, such that all $n$ cliques are vertex-disjoint. Then, every partition that has a trivial intersection with every 5-clique is stable.

\paragraph{Average-case dynamics.} Here, we assume that the initial configuration $\xi_0$ is uniformly random. Recall that, for (asymptotically) almost all graphs on $n$ vertices~\cite{Fountoulakis} and for perfect trees $G$ (due to Theorem~\ref{th:3}), whp $\tau(G;\xi_0)=O(\log n)$. It seems plausible that there exists a universal constant $C>0$ such that $\mathbb{E}\tau(G;\xi_0)<C\log n$ for every graph $G$. In order to support this conjecture, note that our proofs, in part, rely on the fact that, in perfect trees, distant vertices have low influence on each other's dynamics. This leads naturally to the following question:  For an arbitrary graph $G$, how does the distribution of the final stabilised opinion of $v$, conditioned on $\xi_0(u)=1$, depends on the graph distance between the vertices $v$ and $u$?

In perfect binary trees, exponential decay of this dependency implies that, with high probability, the final stable partition of opinions is almost balanced. How does the final partition behave on other graph families? For instance, if $G$ is planar, may one of the two stabilised opinion classes contain a large connected component?

\section*{Acknowledgements}
The authors are grateful to Elad Tzalik for useful discussions.

\bibliographystyle{abbrv}
\bibliography{reference}

\end{document}